\documentclass[final,leqno]{siamltex}
\usepackage{amssymb,amsmath,bbm}
\usepackage[usenames,dvipsnames]{color}
\usepackage{mathrsfs}
\usepackage{hyperref}

\usepackage{graphicx,xspace,bm,subfigure}
\usepackage{url}
\usepackage[noend,lined,linesnumbered,algoruled]{algorithm2e}
\usepackage{setspace}

\newtheorem{remark}[theorem]{Remark}

\newcommand{\pder}[2]{\frac{\partial #1}{\partial #2}}

\newcommand{\oprocendsymbol}{\hbox{$\bullet$}}
\newcommand{\oprocend}{\relax\ifmmode\else\unskip\hfill\fi\oprocendsymbol}

\newcommand{\longthmtitle}[1]{\mbox{}\textup{\textsl{(#1):}}}

\renewcommand{\theenumi}{(\roman{enumi})}

\newcommand{\myclearpage}{\clearpage}
\renewcommand{\myclearpage}{}

\usepackage{tikz}
\usetikzlibrary{positioning,chains,fit,shapes,calc}

\newcommand\scalemath[2]{\scalebox{#1}{\mbox{\ensuremath{\displaystyle #2}}}}

\begin{document}
\title{Controllability of coupled parabolic systems with multiple underactuations, part 2: null controllability}
  \author{Drew Steeves\footnotemark[1] \and 
         Bahman Gharesifard\footnotemark[2] \and
         Abdol-Reza Mansouri\footnotemark[2]}
\renewcommand{\thefootnote}{\fnsymbol{footnote}}
  \footnotetext[1]{Department of Mechanical and Aerospace Engineering, University of California, San Diego, EBU1 2101, La Jolla, CA, United States 92093, \texttt{dsteeves@eng.ucsd.edu}}
   \footnotetext[2]{Department of Mathematics and Statistics, Queen's University, Jeffery Hall, University Ave., Kingston, ON, Canada K7L3N6 (\texttt{bahman.gharesifard@queensu.ca}, \texttt{mansouri@mast.queensu.ca})}

\maketitle

\begin{abstract}
This paper is the second of two parts which together study the null controllability of a system of coupled parabolic PDEs. Our work specializes to an important subclass of these control problems which are coupled by first and zero-order couplings and are, additionally, underactuated. In the first part of our work~\cite{steeves2017part1}, we posed our control problem in a framework which divided the problem into interconnected components: the algebraic control problem, which was the focus of the first part; and the analytic control problem, whose treatment was deferred to this paper. We use slightly non-classical techniques to prove null controllability of the analytic control problem by means of internal controls appearing on every equation. We combine our previous results in~\cite{steeves2017part1} with the ones derived below to establish a null controllability result for the original problem.
%, we resolve the original problem (after some technical verifications on the regularity of the controls in the analytic system). 
\end{abstract}
\section{Introduction}\label{sec:intro}
We begin with defining some notation and conventions.
\subsection{Notation and conventions}
Throughout this work, we define $\mathbb{N}^*:=\mathbb{N}\setminus\{0\}$, and similarly, $\mathbb{R}^*:=\mathbb{R}\setminus\{0\}$. For $n,k\in\mathbb{N}^*$, we denote the set of $n\times k$ matrices with real-valued entries by $\mathcal{M}_{n\times k}(\mathbb{R})$, and we denote the set of $n\times n$ matrices with real-valued entries by $\mathcal{M}_{n}(\mathbb{R})$. We denote the set of linear maps from a vector space $U$ to a vector space $V$ by $\mathscr{L}(U;V)$. For $(X,\mathcal{T}_X)$ a topological space and $U\subset X$, we denote the closure of $U$ by $\bar{U}$. We now recall the coupled parabolic system of interest.\subsection{A system of interest}
In this second part of this two-part work, the primary objective is maintained from the first part: that is, for $Q_T := (0,T)\times \Omega$ and $\Sigma_T := (0,T)\times \partial\Omega$ for some $T>0$, we wish to study the controllability properties of the system of coupled parabolic PDEs given by
\begin{equation}\label{eq:coupled_system}
\left\{\begin{aligned}
    \partial_t y &=\text{div}(D\nabla y)+G\cdot\nabla y+Ay+r, &&\text{in} \; Q_T, \cr
    y&=0, &&\text{on} \; \Sigma_T, \cr
    y(0,\cdot)&=y^0(\cdot), &&\text{in}\;\Omega,
  \end{aligned}\right.
\end{equation}
where $D:=\text{diag}(d_1,\dots,d_m)$, $G:=(g_{pk})_{1\leq p,k\leq m} \in \mathcal{M}_m(\mathbb{R}^n)$ and $A:=(a_{pk})_{1\leq p,k\leq m} \in \mathcal{M}_m(\mathbb{R})$. We associate to this system the differential operator
\begin{equation}\label{eq:coupled_divergence_operator_2}
\mathcal{L}y = \sum_{p=1}^m \left(-\text{div}(d_p\nabla y_p) -\sum_{k=1}^m g_{pk} \cdot \nabla y_k -\sum_{k=1}^m a_{pk} y_k\right)\mathbf{e_p},
\end{equation}
where $g_{pk}:=(g_{pk}^1,\dots,g_{pk}^n) \in \mathbb{R}^n$, $d_p\in\mathcal{M}_n(\mathbb{R})$ is symmetric and $\mathbf{e_p}$ is the  $p$-th canonical basis vector in $\mathbb{R}^m$, for $p\in\{1,\dots,m\}$. We call $g_{pk}$ the \emph{first-order coupling coefficients} and $a_{pk}$ the \emph{zero-order coupling coefficients}, which are constant in space and time.

In this work, we assume that $\mathcal{L}$ satisfies the uniform ellipticity condition: that is, there exists $C>0$ such that,
\begin{equation}\label{eq:ellipticity}
\sum_{i,j=1}^{n} d_p^{ij}\xi_i\xi_j \geq C|\xi|^{2}, \qquad \forall \; \xi \in\mathbb{R}^n.
\end{equation}

Suppose $r\in L^2(Q_T)^m$, $y^0\in L^2(\Omega)^m$. For $u,v\in H_0^1(\Omega)^m$, we define the bilinear form
\[
B[u,v]:=\int_{\Omega} \sum_{p,k=1}^m \left(\sum_{i,j=1}^n d_p^{ij} (\partial_{x_i} u_p) (\partial_{x_j} v_p) - \sum_{i=1}^n g_{pk}^i (\partial_{x_i} u_k) v_p - a_{pk}u_kv_p\right)\mathbf{e_p} dx.
\]
One has the following definition.
\begin{definition}\label{def:weak_solution}
Suppose $r\in L^2(Q_T)^m$, $y^0\in L^2(\Omega)^m$. A function
\[
\mathbf{y} \in L^2((0,T);H_0^1(\Omega))^m\cap H^1((0,T);H^{-1}(\Omega))^m
\]
is said to be a weak solution of system~\eqref{eq:coupled_system} provided that for every $v\in H_0^1(\Omega)^m$ and almost every $t\in [0,T]$
\begin{enumerate}
\item $\langle\frac{d}{dt}\mathbf{y}, v\rangle + B[\mathbf{y},v]=\int_{\Omega} \mathbf{r}^T v dx$, and;
\item $\mathbf{y}(0)=y^0$,
\end{enumerate}
where $\langle\cdot,\cdot\rangle$ denotes the appropriate duality pairing.
%the second equality makes sense thanks to Theorem~\ref{thm:sobolev_regularity_time}.
\end{definition}

It can be deduced, for example, from~\cite[Theorems 3 and 4, Section 7.1.2]{evans1988partial}, that for any initial condition $y^0\in L^2(\Omega)^m$ and $r\in L^2(Q_T)^m$, system~\eqref{eq:coupled_system} admits a unique solution. From now on, we mean by ``solution to a coupled parabolic system" the weak solution in the sense of Definition~\ref{def:weak_solution}.
\subsection{A parabolic regularity result}
We state a regularity result for the solution of system~\eqref{eq:coupled_system} which is essential in the work to follow.
%\begin{theorem}\cite[Theorem 5, Subsection 7.1.3]{evans1988partial}\label{thm:regularity_1}
%Assume $y^0\in H_0^1(\Omega)^m$ and $\mathbf{r} \in L^2(Q_T)^m$; suppose that $\mathbf{y} \in L^2((0,T);H_0^1(\Omega))^m\cap H^1((0,T);H^{-1}(\Omega))^m$ is the weak solution of system~\eqref{eq:coupled_system}. Then in fact
%\[
%\mathbf{y} \in L^2((0,T);H^2(\Omega))^m\cap L^\infty((0,T);H_0^1(\Omega))^m\cap H^1((0,T);L^2(\Omega))^m,
%\]
%and we have the estimate
%\begin{align*}
%&\left\Vert\mathbf{y}\right\Vert_{ L^2((0,T);H^2(\Omega))^m\cap L^\infty((0,T);H_0^1(\Omega))^m\cap H^1((0,T);L^2(\Omega))^m}\\
%&\leq C\left(||\mathbf{r}||_{L^2((0,T);L^2(\Omega))^m} + ||y^0||_{H_0^1(\Omega)^m} \right),
%\end{align*}
%where $C:=C(\Omega,T)$. If, in addition, $y^0 \in H^2(\Omega)^m$ and $\mathbf{r} \in H^1((0,T);L^2(\Omega))^m$, then
%\[
%\mathbf{y} \in L^\infty((0,T);H^2(\Omega))^m \cap H^2((0,T);H^{-1}(\Omega))^m
%\]
%and
%\[
%\frac{d}{dt} \mathbf{y} \in L^\infty((0,T);L^2(\Omega))^m \cap L^2((0,T); H_0^1(\Omega))^m,
%\]
%with the estimate
%\begin{align*}
%\esssup_{0\leq t \leq T} \left( ||\mathbf{y}||_{H^2(\Omega)^m} + \left\Vert\frac{d}{dt}\mathbf{y}\right\Vert_{L^2(\Omega)^m}\right) &+ ||\mathbf{y}||_{H^1((0,T); H_0^1(\Omega))^m \cap H^2((0,T);H^{-1}(\Omega))^m}\\
%& \leq C\left(||\mathbf{r}||_{L^2((0,T);L^2(\Omega))^m} + ||y^0||_{H^2(\Omega)^m} \right).
%\end{align*}
%\end{theorem}
%Under certain conditions, one can expect an even higher regularity for the weak solutions of system~\eqref{eq:coupled_system}. We have the following regularity result.
\begin{theorem}\cite[Theorem 6, Subsection 7.1.3]{evans1988partial}\label{thm:smoothing_theorem}
For $d\in \mathbb{N}$, assume $y^0 \in H^{2d+1}(\Omega)^m$, $\mathbf{r} \in L^2((0,T);H^{2d}(\Omega))^m\cap H^d((0,T);L^2(\Omega))^m$, and assume that $\mathbf{y} \in L^2((0,T);H_0^1(\Omega))^m\cap H^1((0,T);H^{-1}(\Omega))^m$ is the solution of system~\eqref{eq:coupled_system}. Suppose also that the following compatibility conditions hold:
\[
\begin{cases}
g^0 := y^0 \in H_0^1(\Omega)^m; \\
g^1 := \mathbf{r}(0)-\mathcal{L}g^0 \in H_0^1(\Omega)^m; \\
\vdots \\
g^d := \frac{d^{d-1}\mathbf{r}}{dt^{d-1}}(0)-\mathcal{L}g^{d-1} \in H_0^1(\Omega)^m.
\end{cases}
\]
Then $\mathbf{y} \in L^2((0,T);H^{2d+2}(\Omega))^m\cap H^{d+1}((0,T);L^2(\Omega))^m$ and we have the estimate
\begin{align}\label{eq:smoothing_estimate}
||\mathbf{y}||_{L^2((0,T);H^{2d+2}(\Omega))^m\cap H^{d+1}((0,T);L^2(\Omega))^m} &\leq C\left( ||\mathbf{r}||_{L^2((0,T);H^{2d}(\Omega))^m\cap H^d((0,T);L^2(\Omega))^m}\right.\nonumber \\
&\left.+ ||y^0||_{H^{2d+1}(\Omega)^m}\right).
\end{align}
\end{theorem}
\subsection{The control problem}
This work specializes to the case of internal (or distributed) control: that is, for $\omega\subset \Omega$ nonempty and open, we study the case where $r=\mathbbm{1}_{\omega}Bu$, for $B= \left(\text{Id}_{c}\quad 0_{c\times (m-c)}\right)^T\in\mathcal{M}_{m\times c}(\mathbb{R})$ and $1\leq c\leq m$, and henceforth, we denote by $q_T$ the set $(0,T)\times\omega$. 

An interesting control problem that arises in many engineering applications is underactuation, that is, when $c<m$. Our work will further specialize to this case, where there are currently few results for \emph{first and zero-order couplings}, for arbitrary $m$ and $c<m-1$ (even for the case of constant coefficients). We focus on a particular type of controllability property, which is defined next.
\begin{definition}
We say that~\eqref{eq:coupled_system} is null controllable in time $T$ if for every initial condition $y^0\in L^2(\Omega)^m$, there exists $u\in L^2(Q_T)^{c}$ such that the solution $y \in L^2((0,T);H_0^1(\Omega))^m\cap H^1((0,T);H^{-1}(\Omega))^m$ to~\eqref{eq:coupled_system} satisfies $y(T)=0$ in $ \Omega $.
%\[
%y(T)=0 \quad \text{in} \quad \Omega.
%\]
\end{definition}

This work's main objective that we aim to achieve by selecting appropriate control inputs is null controllability of system~\eqref{eq:coupled_system}. Next, we recall the method which we introduced in~\cite{steeves2017part1}; we employ this method to achieve our main objective. 
\subsection{Fictitious control method}\label{sub:fictitious_control_method}
We first described following control problem, referred to as the \emph{analytic control problem}: for any $\tilde{y}^0\in L^2(\Omega)^m$, proving the existence of $(\tilde{y},\tilde{u})$ a solution of
\begin{equation}\label{eq:analytic_problem}
\left\{\begin{aligned}
    \partial_t \tilde{y} &=\text{div}(D\nabla \tilde{y})+G\cdot \nabla \tilde{y} + A\tilde{y} +\mathcal{N}\left(\mathbbm{1}_{\omega} \tilde{u}\right), &&\text{in} \; Q_T, \cr
    \tilde{y}&=0, &&\text{on} \; \Sigma_T, \cr
    \tilde{y}(0,\cdot)&=\tilde{y}^0(\cdot), &&\text{in} \; \Omega,
  \end{aligned}\right.
\end{equation}
such that $\tilde{y}(T,\cdot)=0$, where $\mathcal{N}$ is a differential operator that was chosen to be the identity in~\cite{steeves2017part1}, $\tilde{u}$ \emph{acts on all equations} in~\eqref{eq:analytic_problem}, and we denote by $\mathbbm{1}_{\omega}$ a \emph{smooth version} of the indicator function (this can be constructed via mollification; cf. relation~\eqref{eq:theta_definition} for its exact definition). 
%Note that $(\tilde{y},\tilde{u})$ has to be in a suitable space: in particular, $\tilde{u}$ has to be regular enough to withstand the derivatives applied by $\mathcal{N}$.
%While these restrictions make solving control system~\eqref{eq:analytic_problem} slightly non-classical (we need to use a \emph{weighted} observability inequality), there is hope to finding such a solution since controls appear in every equation in~\eqref{eq:analytic_problem}.

We next consider a different control problem, referred to as the \emph{algebraic control problem}: proving the existence of a solution $(\hat{y},\hat{u})$ of
\begin{equation}\label{eq:algebraic_problem}
\left\{\begin{aligned}
    \partial_t \hat{y} &=\text{div}(D\nabla \hat{y})+G\cdot \nabla \hat{y} + A\hat{y} +B\hat{u} + \mathcal{N}\left(\mathbbm{1}_{\omega} \tilde{u}\right), &&\text{in} \; Q_T, \cr
    \hat{y}&=0, &&\text{on} \; \Sigma_T, \cr
    \hat{y}(0,\cdot)&=\hat{y}(T,\cdot)=0, &&\text{in} \; \Omega,
  \end{aligned}\right.
\end{equation}
\emph{where} $\hat{u}$ \emph{acted only on the first} $c$ \emph{equations}.

We defined the notion of \emph{algebraic solvability} of~\eqref{eq:algebraic_problem} in~\cite{steeves2017part1}, which is a property that enabled us to algebraically ``invert" the differential operator associated to~\eqref{eq:algebraic_problem} and recover the solution to this control problem locally. We obtained the following result, see~\cite[Proposition~4.10]{steeves2017part1}.
\begin{proposition}\label{prop:algebraic_solvability}
Given $m, n$ and $c$ in $\mathbb{N}^*$ with $\lfloor\frac{m}{2}\rfloor+1\leq c\leq m$, if 
\begin{enumerate}
\item $c\geq h$, where $h:=(m-c)(n+1)$, and;
\item the matrix $C\in\mathcal{M}_{h}(\mathbb{R})$ given by 
\begin{equation}\label{eq:C_def}
\scalemath{0.9}{
C:=\left(\begin{array}{c c c c c c c c c c}
a_{(c+1)\alpha_1}&\dots&a_{m\alpha_1}&g_{(c+1)\alpha_1}^1&\dots&g_{m\alpha_1}^1&\dots&g_{(c+1)\alpha_1}^n&\dots&g_{m\alpha_1}^n\\
a_{(c+1)\alpha_2}&\dots&a_{m\alpha_2}&g_{(c+1)\alpha_2}^1&\dots&g_{m\alpha_2}^1&\dots&g_{(c+1)\alpha_2}^n&\dots&g_{m\alpha_2}^n\\
\vdots&&\vdots&\vdots&&\vdots&&\vdots&&\vdots\\
a_{(c+1)\alpha_{h}}&\dots&a_{m\alpha_{h}}&g_{(c+1)\alpha_{h}}^1&\dots&g_{m\alpha_{h}}^1&\dots&g_{(c+1)\alpha_{h}}^n&\dots&g_{m\alpha_{h}}^n
\end{array}\right)}
\end{equation}
is non-singular for any $\{\alpha_1,\dots,\alpha_{h}\}\subseteq\{1,\dots,c\}$ with $\alpha_1\neq\dots\neq\alpha_{h}$, where $g_{ij}^k$ is the $k$-th component of $g_{ij}$, for $k\in\{1,\dots,n\}$ and for $i,j\in\{1,\dots,m\}$,
\end{enumerate}
then~\eqref{eq:algebraic_problem} is algebraically solvable in $q_T$.
\end{proposition}

Importantly, the ``inverse" differential operator that we recovered in~\cite{steeves2017part1}, denoted by $\mathcal{B}$, was of differential order at most $p+2$ in space. This differential order is of consequence in Section~\eqref{sec:null_controllability_analytic_system}, where we require the controls in the analytic system~\eqref{eq:analytic_problem} to be regular enough to withstand these $p+2$ spatial derivatives. This regularity on the controls is necessary for the solution that was constructed for the algebraic problem to be well-defined.

With the algebraic problem resolved, solving the analytic problem is this work's secondary objective. Achieving this secondary objective will allow us to attain this work's main objective, as will be shown in Section~\ref{sec:null_controllability_analytic_system}.
\subsection{Statement of contributions}
%The second part of the two-part work has two main contributions. 
The first contribution is a partial generalization of~\cite[Theorem 1]{duprez2016indirect}. In particular, our result gives a sufficient condition for the null and approximate controllability of an underactuated system of coupled parabolic PDEs, with constant first and zero-order couplings, when more than half of the equations are actuated,and additionally, is large enough. Importantly, this controllability condition applies to systems with multiple underactuations. Furthermore, this condition, which requires the rank of a matrix containing some of the coupling coefficients as entries to be full rank, is generic.
%(where in this context, by ``generic" we mean that ``a generic matrix is not rank deficient"). 
The technique used to prove this result is adapted from~\cite{coron2014local}. 
%Assumption~\ref{assumption:distinct_equations}, \bah{essential to our treatment, is that the equations within a system of interest be \emph{distinct}.
%Secondly, in the cases where this controllability condition may not be generic, we characterize precisely why these non-generic conditions arise in our treatment. At the end of Section~\ref{sec:algebraic_solvability}, we demonstrate the technical nature of these non-generic cases and show how they are artifacts of our treatment.

Our second contribution is Proposition~\ref{prop:carleman_estimate}, which is an extension of~\cite[Proposition 2.2]{duprez2016indirect}. Specifically, our Carleman estimate contains higher differential order terms on its lefthand side, which allows us to construct very regular controls in Proposition~\ref{prop:analytic_null_controllability}. Importantly, these highly regular controls are necessary when applying Theorem~\ref{thm:main_theorem} to problems with many underactuations.
\section{Main result}\label{sec:problem}
The main controllability theorem of this work is stated next, where we assume that more than half of the equations in system~\eqref{eq:coupled_system} are actuated.
%Additionally, we assume that the equations within a system of interest be distinct (c.f. Assumption~\ref{assumption:distinct_equations}).
\begin{theorem}\label{thm:main_theorem}
For a fixed $m$ in $\mathbb{N}^*$, suppose $\Omega\subset \mathbb{R}^n$ nonempty, open and bounded. Furthermore, suppose $\Omega$ is connected and of class $C^{p+2}$. For $\lfloor\frac{m}{2}\rfloor+1\leq c\leq m$, if 
\begin{enumerate}
\item $c\geq h$, where $h:=(m-c)(n+1)$, and;
\item the matrix $C\in\mathcal{M}_{h}(\mathbb{R})$ in~\eqref{eq:C_def} is non-singular for any $\{\alpha_1,\dots,\alpha_{h}\}\subseteq\{1,\dots,c\}$ with $\alpha_1\neq\dots\neq\alpha_{h}$, where $g_{ij}^k$ is the $k$-th component of $g_{ij}$, for $k\in\{1,\dots,n\}$ and for $i,j\in\{1,\dots,m\}$,
\end{enumerate}
then the system~\eqref{eq:coupled_system} is null controllable in time $T$.
\end{theorem}
\begin{remark}\label{rk:c_less_than_h}
In~\cite[Example 4.11]{steeves2017part1}, we addressed the scenario $c<h$, where, for small systems in low dimension, one can employ the treatment in~\cite{steeves2017part1} to derive a generic rank condition similar to the one above that ensures algebraic solvability with $\mathcal{B}$ of differential order at most $p+2$ in space.
\end{remark}

The rest of this work is devoted to proving the above result. First, we will resolve the analytic control problem in Section~\ref{sec:null_controllability_analytic_system}; next, we will utilize the solutions to the algebraic and analytic control problems to solve our original control problem in Section~\ref{sec:carleman_estimate}, which is the null controllability of the underacted system~\eqref{eq:coupled_system}.

\section{A Carleman estimate for the analytic problem}\label{sec:carleman_estimate}
In this section, we study the analytic system:
\begin{equation}\label{eq:analytic_problem_restatement}
\left\{\begin{aligned}
    \partial_t \tilde{y} &=\text{div}(D\nabla \tilde{y})+G\cdot \nabla \tilde{y} + A\tilde{y} +\mathbbm{1}_{\omega} \tilde{u}, &&\text{in} \; Q_T, \cr
    \tilde{y}&=0, &&\text{on} \; \Sigma_T, \cr
    \tilde{y}(0,\cdot)&=\tilde{y}^0(\cdot), &&\text{in} \; \Omega.
  \end{aligned}\right.
\end{equation}
The goal of this section is to prove that the solution $(\tilde{y},\tilde{u})$ to the analytic control system~\eqref{eq:analytic_problem_restatement} satisfies the following so-called \emph{weighted observability inequality}, which will help us deduce its null controllability. To this end, we consider the adjoint system to system~\eqref{eq:analytic_problem_restatement} given by
\begin{equation}\label{eq:coupled_adjoint_2}
\left\{\begin{aligned}
    -\partial_t \tilde{\psi} &=\text{div}(D\nabla \tilde{\psi})-G^*\cdot \nabla \tilde{\psi}+A^*\tilde{\psi}, &&\text{in} \; Q_T,\cr
    \tilde{\psi}&=0, &&\text{on} \; \Sigma_T, \cr
    \tilde{\psi}(T,\cdot)&=\tilde{\psi}^0(\cdot), &&\text{in} \; \Omega,
  \end{aligned}\right.
\end{equation}
where $\tilde{\psi}^0 \in L^2(\Omega)^m$.

We state the weighted observability inequality we aim to establish.
\begin{proposition}\label{prop:observability_inequality}
For every $\tilde{\psi}^0\in L^2(\Omega)^m$, the solution $\tilde{\psi}$ of system~\eqref{eq:coupled_adjoint_2} satisfies
\begin{equation}\label{eq:observability_inequality_carleman}
\int_{\Omega} \left\Vert \tilde{\psi}(0,x)\right\Vert_1^2 dx \leq C_{obs}\iint_{(0,T)\times\omega_0} e^{-2s_1\alpha}\xi^{2p+7}\left\Vert\tilde{\psi}(t,x)\right\Vert_1^2 dx dt,
\end{equation}
where $C_{obs}:=CT^9e^{C(1+3T/4+1/T^5)}>0$ and $\Vert\cdot\Vert_1$ denotes the Euclidean norm. We call~\eqref{eq:observability_inequality_carleman} a weighted observability inequality with weight $\rho:=e^{-2s_1\alpha}\xi^{2p+7}$, for $\alpha$ and $\xi$ defined below in~\eqref{eq:alpha_weight} and~\eqref{eq:xi_weight}, respectively, where $s_1:=\sigma (T^5+T^{10})$ for $\sigma>0$ depending on $\Omega$ and $\omega_0$.
\end{proposition}

We utilize the \emph{Carleman estimate} technique to develop an estimate which will help us establish the observability inequality stated above; the proof of Proposition~\ref{prop:observability_inequality} follows from Proposition~\ref{prop:carleman_estimate} and is given in the Appendix. This section builds upon the technique developed in~\cite[Section 2.2]{duprez2016indirect}: in particular, it incorporates the higher-order terms found on the lefthand side of~\eqref{eq:carleman_estimate} which allow us to construct highly regular controls for system~\eqref{eq:analytic_problem_restatement}. Constructing a solution $(\tilde{y},\tilde{u})$ to system~\eqref{eq:analytic_problem_restatement} with highly regular controls and satisfying $\tilde{y}(T,\cdot)=0$ is treated in Section~\ref{sec:null_controllability_analytic_system}.
%
%Carleman estimates are weighted energy estimates for solutions to PDEs with exponential weights. These types of estimates for parabolic operators are derived, for example, in~\cite[Section 4.7]{alabau2012control}. Carleman estimates were initially introduced in~\cite{carleman1939probleme} to obtain uniqueness and stability results for a particular first-order initial-boundary value problem; they have since been used to derive results in many applications, including exact, approximate and null controllability results for partial differential equations with internal or boundary control.
\subsection{Some notation and technical results}
We begin by introducing some notation. For the multi-index $\beta$ of length $l$ consisting of multi-indices, consider the $l$\textsuperscript{th}-order tensor given by $C:=(C_\beta)_{\beta}$, where $\beta_i$ has length $n_i$, for $n_i\in\mathbb{N}^*$, for $i\in\{1,\dots,l\}$. We associate to $C$ the element-wise norm:
\[
\Vert\cdot\Vert_l:=\left(\sum_{i_1=1,\dots,i_l=1}^{n_1,\dots,n_l}C_{\beta_1(i_1),\dots,\beta_l(i_l)}^2\right)^{1/2}.
\]
%This is the norm associated to the natural inner product on the tensor product of two inner product spaces.
An equivalent interpretation of $\Vert\cdot\Vert_l$ is the following: given a $l$\textsuperscript{th}-order tensor $C$, one \emph{vectorizes} $C$ into a vector of length $\sum_{i=1}^l n_i$ and then applies the Euclidean norm to recover $\Vert\cdot\Vert_l$. Fix a sequence $(\omega_i)_{i=0}^{p+2}$ of nonempty open subsets of $\omega$ such that
\[
\begin{cases}
\bar{\omega}_i \subset \omega_{i-1}, \quad \text{for} \quad i\in\{1,\dots,p+2\},\\
\bar{\omega}_0 \subset \omega.
\end{cases}
\]
The next result is an adaptation of~\cite[Lemma 1.1]{fursikov1996controllability} (see also~\cite[Lemma 2.68]{coron2007control}).
%Its proof is included in the Appendix.
\begin{lemma}\label{lemma:morse_function}
Assume that $\Omega$ is of class $C^{p+2}$ and connected. Then there exists $\eta^0\in C^{p+2}(\bar{\Omega})$ such that
\begin{equation}\label{eq:existence_eta}
\left\{\begin{aligned}
\left\Vert\nabla\eta^0\right\Vert_1&\geq \kappa, \quad &&\text{in} \quad \Omega\setminus \omega_{p+2},\cr
\eta^0 &> 0, \quad &&\text{in} \quad \Omega, \cr
\eta^0 &= 0, \quad &&\text{on} \quad \partial\Omega,
  \end{aligned}\right.
\end{equation}
for some $\kappa>0$.
\end{lemma}

%For $r=p+2$, fix such an $\eta^0\in C^{p+2}(\bar{\Omega})$, where $p$ is chosen to be large enough such that the $p$-times prolonged version system~\eqref{eq:reduced_system_of_equations} be overdetermined.
\begin{remark}\label{rk:r_large_enough}
In~\eqref{eq:bound_on_spatial_derivatives_weight}, we require $\eta^0$ to be $(p+2)$--times differentiable; this is why we require spatial domain boundary regularity in Theorem~\ref{thm:main_theorem}.
\end{remark}

For $(t,x)\in Q_T$ we define
\begin{equation}\label{eq:alpha_weight}
\alpha(t,x):=\frac{e^{12\lambda||\eta^0||_\infty} - e^{\lambda(10||\eta^0||_\infty+\eta^0(x))}}{t^5(T-t)^5} 
\end{equation}
and
\begin{equation}\label{eq:xi_weight}
\xi(t,x):=\frac{e^{\lambda(10||\eta^0||_\infty+\eta^0(x))}}{t^5(T-t)^5}. 
\end{equation}
Additionally, for $t\in(0,T)$ we define
\begin{equation}\label{eq:alpha_star_weight}
\alpha^*(t):=\max_{x\in\bar{\Omega}} \alpha(t,x)
\end{equation}
and
\begin{equation}\label{eq:xi_star_weight}
\xi^*(t):=\min_{x\in\bar{\Omega}} \xi(t,x).
\end{equation}
For $s, \lambda > 0$ and $u\in L^2((0,T);H_0^1(\Omega))\cap H^1((0,T);H^{-1}(\Omega))$, let us define
\begin{equation}\label{eq:carleman_integrals}
\mathcal{I}(s,\lambda;u):=s^3\lambda^4\iint_{Q_T}e^{-2s\alpha}\xi^3|u|^2dxdt+s\lambda^2\iint_{Q_T}e^{-2s\alpha}\xi\left\Vert\nabla u\right\Vert_1^2dxdt.
\end{equation}
In the work to follow, for $u\in L^2((0,T);H_0^1(\Omega))^m\cap H^1((0,T);H^{-1}(\Omega))^m$, we use a slight abuse of notation and define $\mathcal{I}(s,\lambda;u)$ as above but with $|\cdot|$ replaced by $\Vert\cdot\Vert_1$, and with $\Vert\cdot\Vert_1$ replaced by $\Vert\cdot\Vert_2$.

We now state a Carleman estimate result for the heat equation; the proof is quite technical and is omitted here.
\begin{lemma}\cite[Theorem 1]{fernandez2006null}\label{lemma:carleman_estimate_heat}
Assume that $d>0$, $u^0\in L^2(\Omega)$, $f_1\in L^2(Q_T)$ and $f_2\in L^2(\Sigma_T)$. Then there exists a constant $C:=C(\Omega,\omega_{p+2})>0$ such that the solution to
\[
\left\{\begin{aligned}
    -\partial_t u &=\text{div}(d\nabla u)+f_1, &&\text{in} \; Q_T, \cr
    \pder{u}{n}&=f_2, &&\text{on} \; \Sigma_T, \cr
    u(T,\cdot)&=u^0(\cdot), &&\text{in} \; \Omega,
  \end{aligned}\right.
\]
satisfies
\begin{align*}
\mathcal{I}(s,\lambda;u) \leq C\left(s^3\lambda^4\iint_{(0,T)\times\omega_{p+2}} e^{-2s\alpha}\xi^3 |u|^2 dx dt\right. &+ \iint_{Q_T} e^{-2s\alpha} |f_1|^2 dx dt\\
&\left.+s\lambda \iint_{\Sigma_T} e^{-2s\alpha^*} \xi^*|f_2|^2 d\sigma dt \right)
\end{align*}
for all $\lambda \geq C$ and $s\geq C(T^5+T^{10})$.
\end{lemma}

We can adapt the Carleman estimate in Lemma~\ref{lemma:carleman_estimate_heat} to system~\eqref{eq:coupled_adjoint_2} with Neumann boundary condition; this adapted Carleman estimate will be used later (cf.~\eqref{eq:first_carleman_inequality}).
\begin{lemma}\label{lemma:carleman_estimate_heat_adapted}
Assume that $\tilde{\psi}^0\in L^2(\Omega)^m$ and $u\in L^2(\Sigma_T)^m$. Then there exists a constant $C:=C(\Omega,\omega_{p+2})>0$ such that the solution to
\begin{equation}\label{eq:general_neumann_parabolic_pde}
\left\{\begin{aligned}
    -\partial_t \tilde{\psi} &=\text{div}(D\nabla \tilde{\psi})-G^*\cdot\nabla\tilde{\psi} + A^*\tilde{\psi}, &&\text{in} \; Q_T, \cr
    \pder{\tilde{\psi}}{n}&=u, &&\text{on} \; \Sigma_T, \cr
    \tilde{\psi}(T,\cdot)&=\tilde{\psi}^0(\cdot), &&\text{in}\;\Omega,
  \end{aligned}\right.
\end{equation}
satisfies
\begin{align}\label{eq:fernandez_cara_carleman}
\mathcal{I}(s,\lambda;\tilde{\psi})\leq C\left(s^3\lambda^4\iint_{(0,T)\times\omega_{p+2}}e^{-2s\alpha}\xi^3\left\Vert\tilde{\psi}\right\Vert_1^2dxdt+s\lambda\iint_{\Sigma_T}e^{-2s\alpha^*}\xi^*\left\Vert u\right\Vert_1^2 d\sigma dt \right)
\end{align}
for all $\lambda \geq C$ and $s\geq C(T^5+T^{10})$.
\end{lemma}

The proof of Lemma~\ref{lemma:carleman_estimate_heat_adapted} can be deduced from Lemma~\ref{lemma:carleman_estimate_heat} and the definitions of $\xi$ and $\alpha$ (one can absorb the integral with coupling terms appearing as the integrand into $\mathcal{I}(s,\lambda;\tilde{\psi})$ on the lefthand side).

We will also use the following estimate in the ensuing treatment (cf.~\eqref{eq:step_2_inequality_1} and~\eqref{eq:step_2_inequality_2}).
\begin{lemma}\cite[Lemma 3]{coron2009null}\label{lemma:coron_result}
Let $r\in\mathbb{R}$. There exists a $C:=C(\Omega, \omega_{p+2}, r)>0$ such that for every $T>0$ and every $u\in L^2((0,T);H^1(\Omega))$,
\begin{align*}
s^{r+2}\lambda^{r+3}\iint_{Q_T}e^{-2s\alpha}\xi^{r+2}|u|^2dxdt\leq &C\Bigg(s^r\lambda^{r+1}\iint_{Q_T}e^{-2s\alpha}\xi^r\left\Vert\nabla u\right\Vert_1^2dxdt \\
&+s^{r+2}\lambda^{r+3}\iint_{(0,T)\times\omega_{p+2}}e^{-2s\alpha}\xi^{r+2}|u|^2dxdt\Bigg)
\end{align*}
for every $\lambda \geq C$ and $s\geq C(T^5+T^{10})$.
\end{lemma}

Finally, one can establish the following Carleman estimate for system~\eqref{eq:coupled_adjoint_2}, which is an extension of~\cite[Proposition 1]{duprez2016indirect}.
\begin{proposition}\label{prop:carleman_estimate}
There exists a constant $C:=C(\Omega,\omega_0)>0$ such that for every $\tilde{\psi}^0\in L^2(\Omega)^m$, the solution $\tilde{\psi}$ to system~\eqref{eq:coupled_adjoint_2} satisfies
\begin{align}\label{eq:carleman_estimate}
\iint_{Q_T} e^{-2s\alpha}\sum_{k=1}^{p+4} s^{2k-1}\lambda^{2k}\xi^{2k-1}&\Vert\nabla^{p+4-k}\tilde{\psi}\Vert_{p+5-k}^2dxdt\nonumber\\
&\leq Cs^{2p+7}\lambda^{2p+8}\iint_{(0,T)\times\omega_0} e^{-2s\alpha} \xi^{2p+7} \left\Vert \tilde{\psi}\right\Vert_1^2 dx dt
\end{align}
for every $\lambda\geq C$ and $s\geq C(T^5+T^{10})$.
\end{proposition}\\
\begin{remark}
It is worth pointing out to the fact that~\eqref{eq:carleman_estimate} contains spatial derivatives past order one, since $\tilde{\psi}^0$ is assumed to be in $L^2(\Omega)^m$, and hence $\tilde{\psi}\in L^2((0,T);H_0^1(\Omega))^m\cap H^1((0,T);L^2(\Omega))^m$. However, due to inequalities~\eqref{eq:p_odd_bound} and~\eqref{eq:p_even_bound} and since the weight $e^{-2s\alpha}$ absorbs the singularity of $\xi$ at $t=0$, one can deduce that these integrals exist.
\end{remark}
\section{Proof of main theorem}\label{sec:null_controllability_analytic_system}
Recall from Section~\ref{sec:problem} that our principal goal was to prove null controllability of system~\eqref{eq:coupled_system} with multiple underactuations. To this end, we studied an algebraic system and an analytic system both related to system~\eqref{eq:coupled_system}. In~\cite{steeves2017part1}, we developed an algebraic method which allowed us to solve the algebraic control problem under the assumption that the source term $\mathbbm{1}_{\omega}\tilde{u}$ be regular enough so that our algebraic solution $\mathcal{B}(\mathbbm{1}_{\omega}\tilde{u})$ be well-defined, where $\mathcal{B}$ is a differential operator of order zero in time and at most $p+2$ in space. In Section~\ref{sec:carleman_estimate}, we established the weighted observability inequality~\eqref{eq:observability_inequality_carleman} for the analytic system~\eqref{eq:analytic_problem_restatement}.

The goal of this section is to solve the analytic control problem~\eqref{eq:analytic_problem} with \emph{regular enough} controls $\mathbbm{1}_{\omega}\tilde{u}$ so that the algebraic control problem~\eqref{eq:algebraic_problem} also be solved. The treatment presented in this section is an extension of that used in~\cite[Section 2.3]{duprez2016indirect}. In particular, since the right inverse $\mathcal{B}$ of $\mathcal{L}$ derived implicitly in~\cite{steeves2017part1} is in general of order at most $p+2$ in space, we require higher regularity of controls in the analytic problem than in~\cite{duprez2016indirect}.
\subsection{An optimal control result}
We do not use the weighted observability inequality to directly deduce null controllability of system~\eqref{eq:analytic_problem_restatement}. Instead, we use a method developed in~\cite{fursikov1996controllability} to construct controls with high regularity which will help us deduce controllability results; to do this, we rely on the following unconstrained optimal control result.
\begin{theorem}\cite[Section 3, Theorem 2.2]{lions1968controle}\label{thm:lions_optimal_control}
Let $y^0\in L^2(\Omega)^m$, $u\in L^2(Q_T)^m$, $B\in\mathscr{L}(L^2(Q_T)^m; L^2(Q_T)^m)$, and suppose $\mathcal{L}$ given in~\eqref{eq:coupled_divergence_operator_2} satisfies the uniform ellipticity condition~\eqref{eq:ellipticity}. Let $N\in\mathscr{L}(L^2(Q_T)^m;L^2(Q_T)^m)$ such that $\left(N u, u\right)_{L^2(Q_T)^m}\geq\nu\Vert u\Vert_{L^2(Q_T)^m}^2$ for $\nu>0$ and for all $u\in L^2(Q_T)^m$, and let $D\in\mathscr{L}(H_0^1(\Omega))^m;H_0^1(\Omega))^m)$. Consider the optimal control problem associated to system~\eqref{eq:coupled_system} with cost functional $J(u):L^2(Q_T)^m\rightarrow \mathbb{R}^+$ given by
\begin{equation}\label{eq:optimal_control_cost_function}
J(u):=\left(Nu,u\right)_{L^2(Q_T)^m}+\left(Dy_u(T,\cdot)-z_d\right)_{L^2(\Omega)^m}^2,
\end{equation}
for some $z_d\in H_0^1(\Omega)^m$. This problem has a unique solution, and the optimal control is characterized by the following relations:
\[\left\{\begin{aligned}
    \partial_t y_u &=\text{div}(D\nabla y_u)+G\cdot\nabla y_u+ Ay_y +Bu, &&\text{in} \; Q_T, \cr
    y_u&=0, &&\text{on} \; \Sigma_T, \cr
    y_u(0,\cdot)&=y^0(\cdot), &&\text{in} \; \Omega,
  \end{aligned}\right.
\]
\[
\left\{\begin{aligned}
    -\partial_t\psi_u&=\text{div}(D\nabla \psi_u)-G^*\cdot\nabla\psi_u + A^*\psi_u,&&\text{in} \; Q_T, \cr
    \psi_u&=0, &&\text{on} \; \Sigma_T, \cr
    \psi_u(T,\cdot)&=D^*\left(Dy_u(T,\cdot)-z_d\right), &&\text{in} \; \Omega,
  \end{aligned}\right.
\]
and
\[
B^*\psi_u+Nu=0.
\]
Hence, for this unconstrained optimal control problem, the second term in~\eqref{eq:optimal_control_cost_function} has no dependence on $u$ (nor do the primal/adjoint systems).
\end{theorem}
\subsection{Null controllability of the analytic problem}
Recall that in~\cite{steeves2017part1}, we fixed a $p$ large enough such that we recovered algebraic solvability of~\eqref{eq:algebraic_problem}. In this section, we establish the following proposition.
\begin{proposition}\label{prop:analytic_null_controllability}
Consider $\theta\in C^{p+2}(\bar{\Omega})$ such that
\begin{equation}\label{eq:theta_definition}
\left\{\begin{aligned}
    &\text{Supp}(\theta) \subseteq \omega, && \cr
    &\theta=1, &&\text{in} \; \omega_0, \cr
    &0\leq\theta\leq 1, &&\text{in} \; \Omega.
  \end{aligned}\right.
\end{equation}
Then there exists $v\in L^2(Q_T)^m$ such that
\[
(\tilde{y},\theta v)\in L^2((0,T);H_0^1(\Omega))^m\cap H^1((0,T);H^{-1}(\Omega))^m \times L^2(Q_T)^m
\]
is a solution to the analytic control problem~\eqref{eq:analytic_problem} satisfying $\tilde{y}(T,\cdot)=0$ in $\Omega$. Moreover, for every $K\in(0,1)$, we have
\begin{align}
&e^{Ks_1\alpha^*}v \in L^2((0,T);H^{p+2}(\Omega)\cap H_0^1(\Omega))^m\cap H^1((0,T);L^2(\Omega))^m, \ \mathrm{and} \nonumber\\
&\Vert e^{Ks_1\alpha^*}v\Vert_{L^2((0,T);H^{p+2}(\Omega)\cap H_0^1(\Omega))^m\cap H^1((0,T);L^2(\Omega))^m} \leq C \Vert \tilde{y}^0\Vert_{L^2(\Omega)^m}.\label{eq:analytic_control_regularity}
\end{align}
\end{proposition}
\begin{proof}
Let $\tilde{y}^0\in L^2(\Omega)^m$, $\rho:=e^{-2s_1\alpha}\xi^{2p+7}$ and $C:=C(\Omega,\omega_0,T)>0$. Let $k\in\mathbb{N}^*$ and denote by $L^2(Q_T, \rho^{-1/2})^m$ the space of functions which, when multiplied by $\rho^{-1/2}$, are $L^2$-integrable (i.e., for $u\in L^2(Q_T, \rho^{-1/2})^m$, we require $\iint_{Q_T} \rho^{-1} \left\Vert u\right\Vert_1^2dxdt < \infty$). Consider the following optimal control problem:
\begin{equation}\label{eq:cost_functional}
\begin{cases}
\text{minimize} \quad J_k(v) := \frac{1}{2}\iint_{Q_T}\rho^{-1}\left\Vert v\right\Vert_1^2dxdt+\frac{k}{2}\int_{\Omega}\left\Vert\tilde{y}(T,\cdot)\right\Vert_1^2dx,\\
\text{subject to} \quad v\in L^2(Q_T, \rho^{-1/2})^m,
\end{cases}
\end{equation}
where $\tilde{y}\in L^2((0,T);H_0^1(\Omega))^m\cap H^1((0,T);H^{-1}(\Omega))^m$. The functional $J_k$ is differentiable, coercive and strictly convex on $L^2(Q_T, \rho^{-1/2})^m$. By Theorem~\ref{thm:lions_optimal_control} (for $D=\sqrt{k}$, $N=\rho^{-1}$ and $z_d=0$ in $Q_T$), there exists a unique solution to this problem, and the optimal control is characterized by the solution $\tilde{y}_k$ to the analytic system
\begin{equation}\label{eq:primal_system_sequence}
\left\{\begin{aligned}
    \partial_t \tilde{y}_k &=\text{div}(D\nabla \tilde{y}_k)+G\cdot \nabla \tilde{y}_k + A\tilde{y}_k +\theta v_k, &&\text{in} \; Q_T, \cr
    \tilde{y}_k&=0, &&\text{on} \; \Sigma_T, \cr
    \tilde{y}_k(0,\cdot)&=\tilde{y}^0(\cdot), &&\text{in} \; \Omega,
  \end{aligned}\right.
\end{equation}
the solution $\tilde{\psi}_k$ to its adjoint system
\begin{equation}\label{eq:adjoint_system_sequence}
\left\{\begin{aligned}
    -\partial_t \tilde{\psi}_k &=\text{div}(D\nabla \tilde{\psi}_k)-G^*\cdot \nabla \tilde{\psi}_k + A^*\tilde{\psi}_k, &&\text{in} \; Q_T, \cr
    \tilde{\psi}_k&=0, &&\text{on} \; \Sigma_T, \cr
    \tilde{\psi}_k(T,\cdot)&=k\tilde{y}(T,\cdot), &&\text{in} \; \Omega,
  \end{aligned}\right.
\end{equation}
and the relation
\begin{equation}\label{eq:control_sequence_equality}
\left\{\begin{aligned}
    v_k&= -\rho\theta\tilde{\psi}_k, \quad \text{in} \; Q_T, \cr
    v_k&\in L^2(Q_T, \rho^{-1/2})^m.
  \end{aligned}\right.
\end{equation}
From~\eqref{eq:primal_system_sequence} and~\eqref{eq:adjoint_system_sequence}, we calculate
\begin{align}\label{eq:calculation_1}
&\int_0^T \left(( \tilde{y}_k,\partial_t\tilde{\psi}_k)_{L^2(\Omega)^m} + (\partial_t\tilde{y}_k, \tilde{\psi}_k)_{L^2(\Omega)^m} \right) dt\nonumber\\
&=\frac{d}{dt}\int_0^T ( \tilde{y}_k,\tilde{\psi}_k)_{L^2(\Omega)^m}dt\nonumber\\
&=( \tilde{y}_k(T,\cdot),k\tilde{y}_k(T,\cdot) )_{L^2(\Omega)^m} - ( \tilde{y}^0,\tilde{\psi}_k(0,\cdot))_{L^2(\Omega)^m},
\end{align}
and
\begin{align}\label{eq:calculation_2}
&(\tilde{y}_k,\partial_t\tilde{\psi}_k)_{L^2(\Omega)^m} + (\partial_t\tilde{y}_k, \tilde{\psi}_k)_{L^2(\Omega)^m}\nonumber\\
&= (\tilde{y}_k,-\text{div}(D\nabla \tilde{\psi}_k)+G^*\cdot \nabla \tilde{\psi}_k - A^*\tilde{\psi}_k)_{L^2(\Omega)^m}\nonumber\\
&+(\text{div}(D\nabla \tilde{y}_k)+G\cdot \nabla \tilde{y}_k + A\tilde{y}_k +\theta v_k,\tilde{\psi}_k)_{L^2(\Omega)^m}\nonumber\\
&=( \theta v_k,\tilde{\psi}_k)_{L^2(\Omega)^m}.
\end{align}
It follows from~\eqref{eq:control_sequence_equality},~\eqref{eq:calculation_1} and~\eqref{eq:calculation_2} that
\begin{align}\label{eq:cost_functional_inequality1}
\scalemath{0.98}{J_k(v_k)}\;&\scalemath{0.98}{= -\frac{1}{2}\int_0^T ( \theta \tilde{\psi}_k,v_k)_{L^2(\Omega)^m} dt + \frac{1}{2} ( \tilde{y}_k(T,\cdot),\tilde{\psi}_k(T,\cdot))_{L^2(\Omega)^m}}\nonumber\\
&\scalemath{0.98}{= -\frac{1}{2}\int_0^T (\tilde{\psi}_k,\theta v_k)_{L^2(\Omega)^m} dt + \frac{1}{2} \int_0^T \left(( \tilde{y}_k,\partial_t\tilde{\psi}_k)_{L^2(\Omega)^m} + (\partial_t\tilde{y}_k, \tilde{\psi}_k)_{L^2(\Omega)^m} \right) dt}\nonumber\\
&\scalemath{0.98}{+\;\frac{1}{2}( y^0, \tilde{\psi}_k(0,\cdot))_{L^2(\Omega)^m}}\scalemath{0.98}{=\frac{1}{2} ( y^0,\tilde{\psi}_k(0,\cdot))_{L^2(\Omega)^m}.}
\end{align}
Moreover, employing the weighted observability inequality~\eqref{eq:observability_inequality_carleman} along with~\eqref{eq:theta_definition}, \eqref{eq:control_sequence_equality}, \eqref{eq:cost_functional}, \eqref{eq:cost_functional_inequality1} and the Cauchy-Schwarz inequality successively, we have
\begin{align*}
\Vert\tilde{\psi}_k(0,\cdot)\Vert_{L^2(\Omega)^m}^2&\leq C_{obs}\iint_{(0,T)\times\omega_0}\rho\theta^2\left\Vert\tilde{\psi}_k\right\Vert_1^2dxdt\\
&\leq C_{obs}\iint_{Q_T}\rho\theta^2\left\Vert\tilde{\psi}_k\right\Vert_1^2dxdt\\
& = C_{obs}\iint_{Q_T} \rho^{-1}\left\Vert v_k\right\Vert_1^2 dxdt\\
& \leq 2 C_{obs} J_k(v_k)\\
& \leq 2 C_{obs} \Vert \tilde{\psi}_k(0,\cdot)\Vert_{L^2(\Omega)^m}\Vert y^0\Vert_{L^2(\Omega)^m},
\end{align*}
from which we deduce
\begin{equation}\label{eq:adjoint_solution_at_zero_inequality}
\Vert\tilde{\psi}_k(0,\cdot)\Vert_{L^2(\Omega)^m} \leq 2 C_{obs} \Vert y^0\Vert_{L^2(\Omega)^m}.
\end{equation}
Furthermore, by~\eqref{eq:cost_functional_inequality1},~\eqref{eq:adjoint_solution_at_zero_inequality} and the Cauchy-Schwarz inequality, we obtain
\begin{equation}\label{eq:cost_functional_inequality2}
J_k(v_k) \leq C_{obs}\Vert y^0\Vert_{L^2(\Omega)^m}^2.
\end{equation}
One can deduce from parabolic regularity,~\eqref{eq:theta_definition} and~\eqref{eq:cost_functional_inequality2} that
\begin{align}\label{eq:y_tilde_sequence_bound}
\Vert \tilde{y}_k\Vert_{L^2((0,T);H_0^1(\Omega))^m\cap H^1((0,T);H^{-1}(\Omega))^m} &\leq C\left( \Vert \theta v_k\Vert_{L^2(Q_T)^m} + \Vert \tilde{y}^0\Vert_{L^2(\Omega)^m} \right) \nonumber\\
&\leq C\left( \Vert v_k\Vert_{L^2(Q_T)^m}+ \Vert \tilde{y}^0\Vert_{L^2(\Omega)^m}\right)\nonumber\\
&\leq C(1+\sqrt{2C_{obs}})\Vert \tilde{y}^0\Vert_{L^2(\Omega)^m},
\end{align}
since for our choice of $s_1$ (which depends on $p$; see~\eqref{eq:choice_of_s}) and by~\eqref{eq:alpha_weight} and~\eqref{eq:xi_weight}, $\rho\leq1$ in $Q_T$. Owing to the well-known result that in Hilbert spaces, bounded sequences have weakly convergent subsequences (see, for example,~\cite{MR1070713}), along with~\eqref{eq:cost_functional}~\eqref{eq:cost_functional_inequality2}, and~\eqref{eq:y_tilde_sequence_bound}, one can extract subsequences of $(v_k)_k$ and $(\tilde{y}_k)_k$ (which we still denote by $v_k$ and $\tilde{y}_k$) such that
\[
\left\{\begin{aligned}
    v_k&\rightharpoonup v &&\text{in} \; L^2(Q_T,\rho^{-1/2})^m, \cr
    \tilde{y}_k&\rightharpoonup \tilde{y} &&\text{in} \; L^2((0,T);H_0^1(\Omega))^m\cap H^1((0,T);H^{-1}(\Omega))^m,\cr
    \tilde{y}_k(T,\cdot)&\rightharpoonup 0 &&\text{in} \;L^2(\Omega)^m.
  \end{aligned}\right.
\]
Hence, $(\tilde{y},\theta v)$ is the solution to the analytic control problem~\eqref{eq:analytic_problem} with
%\[
$\theta v \in L^2(Q_T,\rho^{-1/2})$.
%\]
Furthermore, we deduce from~\eqref{eq:cost_functional} by taking $k \rightarrow \infty$ that $\tilde{y}(T,\cdot)=0$ (in the sense of Definition~\ref{def:weak_solution}). In addition, by~\eqref{eq:cost_functional_inequality2} and since $\rho\leq 1$ in $Q_T$ for our choice of $s_1$,
\[
\Vert v\Vert_{L^2(Q_T)}^2\leq\sqrt{2C_{obs}}\Vert y^0\Vert_{L^2(\Omega)^m}^2,
\]
as claimed. It is left to show that~\eqref{eq:analytic_control_regularity} is verified. Note that for every $K\in (0,1)$, there exists a $C_K:=C_K(\Omega)$ such that
\begin{equation}\label{eq:exponential_growth_function_bound}
e^{2Ks_1\alpha^*} \leq C_K\xi^{-2p-7}e^{2s_1\alpha},
\end{equation}
for all $(t,x)\in Q_T$. Hence, utilizing~\eqref{eq:exponential_growth_function_bound},~\eqref{eq:cost_functional} and then~\eqref{eq:cost_functional_inequality2}, we obtain
\begin{align}\label{eq:exponential_growth_function_L2_bound}
\Vert e^{2Ks_1\alpha^*}v_k\Vert_{L^2(Q_T)^m}^2&\leq C_K\iint_{Q_T}\rho^{-1}\left\Vert v_k\right\Vert_1^2dxdt\nonumber\\
&\leq C_K\Vert\tilde{y}^0\Vert_{L^2(\Omega)^m}^2.
\end{align}
For $a>0$, one has (see~\eqref{eq:rho_bound1})
\begin{equation}\label{eq:bound_on_time_derivative_weight}
|\partial_t (\xi^a e^{-2s_1\alpha})| \leq CT \xi^{a+6/5} e^{-2s_1\alpha}.
\end{equation}
Furthermore, for $r=\{0,\dots,p+2\}$ one has
\begin{equation}\label{eq:bound_on_spatial_derivatives_weight}
\left\Vert\nabla^r(\xi^a e^{-2s_1\alpha})\right\Vert_{r}\leq C\xi^{a+r}e^{-2s_1\alpha}.
\end{equation}
Indeed,
\begin{align*}
\nabla(\xi^ae^{-2s_1\alpha})&=a\xi^{a-1}\lambda\nabla\eta^0 \xi e^{-2s_1\alpha}-2s_1\xi^ae^{-2s_1\alpha}\left(-\lambda \nabla\eta^0\xi\right)\\
& = \lambda \nabla\eta^0\left(\frac{a}{\xi} + 2s_1\right)\xi^{a+1}e^{-2s_1\alpha},
\end{align*}
and since $C:=C(\Omega,\omega_0,T)$,~\eqref{eq:bound_on_spatial_derivatives_weight} is verified for $r=1$. The same reasoning can be used for the $r$-th derivative, where we have fixed $\eta^0\in C^{p+2}(\bar{\Omega})$. Hence, by~\eqref{eq:control_sequence_equality}, the triangle inequality and then~\eqref{eq:bound_on_spatial_derivatives_weight} for $a=2p+7$, we obtain
\begin{align}\label{eq:exponential_growth_function_H1_bound}
&\Vert e^{Ks_1\alpha^*} \nabla v_k\Vert_{L^2(Q_T)^m}^2\nonumber\\
& = \iint_{Q_T} e^{2Ks_1\alpha^*}\Vert\nabla v_k\Vert_{2}^2 dxdt\nonumber\\
& = \iint_{Q_T} e^{2Ks_1\alpha^*} \Vert\nabla (-\xi^{2p+7}e^{-2s_1\alpha}\theta\tilde{\psi}_k)\Vert_{2}^2 dxdt\nonumber\\
&\leq C\iint_{Q_T} e^{2Ks_1\alpha^*} \left(\left\Vert\nabla(\xi^{2p+7}e^{-2s_1\alpha})\right\Vert_1^2\left\Vert\tilde{\psi}_k\right\Vert_1^2+\left\Vert\xi^{2p+7}e^{-2s_1\alpha}\nabla\tilde{\psi}_k\right\Vert_{2}^2\right) dxdt\nonumber\\
&\leq C\iint_{Q_T} e^{2Ks_1\alpha^*-4s_1\alpha}\left(\xi^{4p+16}\left\Vert\tilde{\psi}_k\right\Vert_1^2 + \xi^{4p+14}\Vert\nabla\tilde{\psi}_k\Vert_{2}^2 \right) dxdt,
\end{align}
and similarly, for $r\in\{1,\dots,p+2\}$, we obtain
\begin{equation}\label{eq:exponential_growth_function_H2_bound}
\Vert e^{Ks_1\alpha^*} \nabla^r v_k\Vert_{L^2(Q_T)^m}^2 \leq C\iint_{Q_T} e^{2Ks_1\alpha^*-4s_1\alpha}\left(\sum_{l=0}^{r} \xi^{4p+14+2l}\Vert\nabla^{r-l}\tilde{\psi}_k\Vert_{r-l+1}^2\right)dxdt.
\end{equation}
By~\eqref{eq:bound_on_time_derivative_weight} and since $\tilde{\psi}_k$ satisfies system~\eqref{eq:adjoint_system_sequence}, we obtain
\begin{align}\label{eq:exponential_growth_function_H1time_bound}
&\Vert \partial_t(e^{Ks_1\alpha^*} v_k)\Vert_{L^2(Q_T)^m}^2\\
&\leq C\iint_{Q_T} e^{2Ks_1\alpha^*-4s_1\alpha}\left(\xi^{(20p+82)/5}\left\Vert\tilde{\psi}_k\right\Vert_1^2+ \xi^{2p+14}\left\Vert\partial_t\tilde{\psi}_k\right\Vert_1^2\right) dxdt\nonumber\\
& \leq C\iint_{Q_T} e^{2Ks_1\alpha^*-4s_1\alpha}\Bigg(\xi^{(20p+82)/5}\left\Vert\tilde{\psi}_k\right\Vert_1^2\nonumber\\
&+ \xi^{2p+14}\left(\Vert\nabla\nabla\tilde{\psi}_k\Vert_{3}^2+\Vert\nabla\tilde{\psi}_k\Vert_{2}^2+\left\Vert\tilde{\psi}_k\right\Vert_1^2\right)\Bigg) dxdt.
\end{align}
Note that for every $a,b>0$ and $K\in(0,1)$, there exists $C_{a,b,K}:=C_{a,b,K}(\Omega)>0$ such that
\begin{equation}\label{eq:exponentials_bound}
\left|\xi^a e^{2Ks_1\alpha^*-4s_1\alpha} \right| \leq C_{a,b,K} \xi^b e^{2s_1\alpha}.
\end{equation}
From~\eqref{eq:exponential_growth_function_L2_bound},~\eqref{eq:exponential_growth_function_H1_bound},~\eqref{eq:exponential_growth_function_H2_bound},~\eqref{eq:exponential_growth_function_H1time_bound} and utilizing~\eqref{eq:exponentials_bound} for appropriate $a$ and $b$,
\begin{align*}
& \Vert e^{Ks_1\alpha^*}v_k\Vert_{L^2((0,T);H^{p+2}(\Omega)\cap H_0^1(\Omega))^m\cap H^1((0,T);L^2(\Omega))^m}\\
& \leq C_{max,K}\iint_{Q_T} e^{-2s_1\alpha}\sum_{k=2}^{p+4}\xi^{2k-1}\Vert\nabla^{p+4-k}\tilde{\psi}_k\Vert_{p+5-k}^2dxdt,
\end{align*}
where
%\[
$C_{max,K}:=\max\{\max_{a,b}\{C_{a,b,K}\},C_K\}$.
%\]
Owing to~\eqref{eq:theta_definition}, Proposition~\ref{prop:carleman_estimate} and~\eqref{eq:control_sequence_equality}, we deduce
\begin{align*}
&\Vert e^{Ks_1\alpha^*}v_k\Vert_{L^2((0,T);H^{p+2}(\Omega)\cap H_0^1(\Omega))^m\cap H^1((0,T);L^2(\Omega))^m}\\
&\leq C_{max,K}C_{obs}\iint_{Q_T} e^{-2s_1\alpha}\xi^{2p+7}\left\Vert\theta\tilde{\psi}_k\right\Vert_1^2 dxdt= C_{max,K}C_{obs}\Vert v_k\Vert_{L^2(Q_T)^m}^2.
\end{align*}
Lastly, for $\bar{C}_K:=\bar{C}_K(\Omega,\omega_0,T)$,~\eqref{eq:cost_functional_inequality2} yields the inequality
\[
\Vert e^{Ks_1\alpha^*}v_k\Vert_{L^2((0,T);H^{p+2}(\Omega)\cap H_0^1(\Omega))^m\cap H^1((0,T);L^2(\Omega))^m}\leq\bar{C}_{K}\Vert\tilde{y}^0\Vert_{L^2(\Omega)^m},
\]
from which~\eqref{eq:analytic_control_regularity} is verified by taking a convergent subsequence and $k\rightarrow \infty$.
\end{proof}

With algebraic solvability of the algebraic control problem~\eqref{eq:algebraic_problem} and null controllability of the analytic control problem~\eqref{eq:analytic_problem} established for highly regular controls, we can now prove null controllability of the system~\eqref{eq:coupled_system} with internal controls $\hat{u}\in L^2(q_T)^c$, where $c<m-1$.

In Proposition~\ref{prop:analytic_null_controllability}, we showed the existence of $(\tilde{y},\theta v)\in L^2((0,T); H_0^1(\Omega))^m\cap H^1((0,T);H^{-1}(\Omega))^m \times L^2(Q_T)^m$ satisfying
\begin{equation}\label{eq:analytic_system_restatement}
\left\{\begin{aligned}
    \partial_t \tilde{y} &=\text{div}(D\nabla \tilde{y})+G\cdot\nabla\tilde{y} + A\tilde{y} + \theta v,&&\text{in} \; Q_T, \cr
    \tilde{y}&=0, &&\text{on} \; \Sigma_T, \cr
    \tilde{y}(0,\cdot)&=y^0(\cdot), &&\text{in} \; \Omega,
  \end{aligned}\right.
\end{equation}
such that $\tilde{y}(T,\cdot)=0$ in $\Omega$. Furthermore, we established the following higher regularity for $v$:
\begin{equation}\label{eq:analytic_control_regularity_restatement}
e^{Ks_1\alpha^*}v \in L^2((0,T);H^{p+2}(\Omega)\cap H_0^1(\Omega))^m\cap H^1((0,T);L^2(\Omega))^m,
\end{equation}
for all $k\in(0,1)$. Notice that~\eqref{eq:analytic_control_regularity_restatement} implies that $v$ is exponentially decaying as $t\rightarrow 0$ and $t\rightarrow T$. For the linear partial differential operator $\mathcal{B}$ (of order zero in time and at most $p+2$ in space) constructed implicitly in~\cite{steeves2017part1}, let us define
\[
\left(\begin{array}{c}
\hat{y} \\
\hat{u}
\end{array}\right) := \mathcal{B}\left(\theta v\right),
\]
which is well-defined by~\eqref{eq:analytic_control_regularity_restatement}. By virtue of $\mathcal{B}$ being a linear partial differential operator of the stated orders with constant coefficients, we conclude that
\begin{equation}\label{eq:algebraic_solution_regularity}
(\hat{y},\hat{u}) \in L^2(q_T)\times L^2(q_T)^c;
\end{equation}
we then extend $(\hat{y},\hat{u})$ by zero to $Q_T$. Since $v$ decays exponentially as $t\rightarrow 0$ and $t\rightarrow T$, $\hat{y}(0,\cdot)=\hat{y}(T,\cdot)=0$ in $\Omega$. Furthermore, it follows from the discussions in Subsection~\ref{sub:fictitious_control_method} that $(\hat{y},\hat{u})$ is the solution to
\begin{equation}\label{eq:algebraic_system_restatement}
\left\{\begin{aligned}
    \partial_t \hat{y} &=\text{div}(D\nabla \hat{y})+G\cdot\nabla\hat{y} + A\hat{y} + B\hat{u}+\theta v,&&\text{in} \; Q_T, \cr
    \hat{y}&=0, &&\text{on} \; \Sigma_T, \cr
    \hat{y}(0,\cdot)&=\hat{y}(T,\cdot)=0, &&\text{in} \; \Omega,
  \end{aligned}\right.
\end{equation}
where, by~\eqref{eq:algebraic_solution_regularity} and by parabolic regularity,
%discussed in Section~\ref{section:well_posedness},
$(\hat{y},\hat{u})$ satisfies Definition~\ref{def:weak_solution}. Defining $(y,u):=(\tilde{y}-\hat{y},-\hat{u})$, it is immediate that $(y,u)$ is the solution to~\eqref{eq:coupled_system} with $y(T,\cdot)=0$ in $\Omega$. This finishes the proof of Theorem~\ref{thm:main_theorem}.
\section{Conclusion}\label{sec:conclusion}
Using the powerful fictitious control technique, which has allowed us to pose our controllability problem as two interconnected problems, we have derived a sufficient condition for the null controllability of a system of coupled parabolic PDEs, where the couplings were constant in space and time and of first and zero-order and more than half of the equations in the system were actuated. This controllability condition is generic. 
%for the case of $c\geq h$, where $c$ denotes the number of controls and for $h$ defined in Theorem~\ref{thm:main_theorem}. Furthermore, we demonstrated that for $c<h$, the possibly non-generic nature of this controllability condition is purely technical and is an artifact of our treatment in Section~\ref{sec:algebraic_solvability}.

\myclearpage
\bibliographystyle{siam}%
\bibliography{DS}

\section*{Appendix}
\renewcommand{\theequation}{A.\arabic{equation}}
\renewcommand{\thetheorem}{A.\arabic{theorem}}
In a proof to follow, we rely on the so-called Gagliardo-Nirenberg interpolation inequality, which is stated next.
\begin{theorem}\cite{nirenberg1959elliptic}\label{eq:interpolation_inequality}
For $\Omega\subset \mathbb{R}^n$ open, for $q,r\in\mathbb{R}$ such that $1\leq q,r \leq \infty$ and for $m\in\mathbb{N}$, let $u:\Omega\rightarrow \mathbb{R}$ such that $u\in L^q(\Omega) \cap W^{m,r}(\Omega)$. For $0\leq j\leq m$, we have
\begin{equation}\label{eq:gagliardo-nirenberg}
||u||_{W^{j,p}(\Omega)} \leq C ||u||_{W^{m,r}(\Omega)}^\alpha ||u||_{L^q(\Omega)}^{1-\alpha},
\end{equation}
where $p$ satisfies
\[
\frac{1}{p}=\frac{j}{n}+\alpha\left(\frac{1}{r}-\frac{m}{n}\right)+\frac{1-\alpha}{q}
\]
for all $\alpha$ in the interval $\frac{j}{m}\leq \alpha\leq 1$, where $C:=C(n, m, j, q, r, \alpha)$, with the following exceptional assumptions:
\begin{enumerate}
\item if $j=0,$ $rm<n$, $q=\infty$, then we require $u\rightarrow 0$ at infinity, and;
\item if $1<r<\infty$ and $m-j-\frac{n}{r}$ a nonnegative integer, then~\eqref{eq:gagliardo-nirenberg} only holds for $\alpha$ satisfying $\frac{j}{m}\leq \alpha < 1$.
\end{enumerate}
\end{theorem}
\begin{proof}\longthmtitle{Proof of Proposition~\ref{prop:carleman_estimate}}\label{proof:A1}
We denote by $C$ various positive constants which depend on $\Omega$ and $\omega_0$. We define the operator
\begin{equation}\label{eq:elliptic_operator_adjoint}
\mathcal{L}^*:=\left(-\text{div}(D\nabla)+G^*\cdot\nabla-A^*\right).
\end{equation}
By density of $H^k(\Omega)^m\cap H_0^1(\Omega)^m$ in $L^2(\Omega)^m$ for $k\in\mathbb{N}$ (this follows from the inclusion $C_c^\infty(\Omega)^m\subset H^k(\Omega)^m\cap H_0^1(\Omega)^m\subset L^2(\Omega)^m$ and since $C_c^\infty(\Omega)^m$ dense in $L^2(\Omega)^m$), we assume without loss of generality that $\tilde{\psi}^0\in H^{2p+5}(\Omega)^m$ and $\left((\mathcal{L}^*)^k\tilde{\psi}^0\right)_{k=0}^{p+2} \subset H_0^1(\Omega)$. Hence by Theorem~\ref{thm:smoothing_theorem}, the solution $\tilde{\psi}$ to system~\eqref{eq:coupled_adjoint_2} is an element of 
\begin{equation}\label{eq:adjoint_regularity}
L^2((0,T);H^{2p+6}(\Omega))^m\cap H^{p+3}((0,T);L^2(\Omega))^m.
\end{equation}
We apply the differential operator $\nabla^{p+2}$ to system~\eqref{eq:coupled_adjoint_2} and, for $\beta$ a multi-index with $|\beta| = p+2$, we denote $\partial_\beta\tilde{\psi}$ by $\phi_{\beta}$ so that $\phi_{\beta}$ satisfies
\begin{equation}\label{eq:differentiated_adjoint}
\left\{\begin{aligned}
    -\partial_t \phi_{\beta} &=\text{div}(D\nabla \phi_{\beta})-G^*\cdot \nabla \phi_{\beta} + A^*\phi_{\beta}, &&\text{in} \; Q_T, \cr
    \pder{\phi_{\beta}}{n}&=\nabla \phi_{\beta} \cdot \mathbf{n}, &&\text{on} \; \Sigma_T, \cr
    \phi_{\beta}(T,\cdot)&=\partial_\beta\tilde{\psi}^0(\cdot), &&\text{in} \; \Omega.
  \end{aligned}\right.
\end{equation}
Indeed, since $D, \; G^*$ and $A^*$ are constant, $\nabla^{p+2}$ commutes with all the terms in system~\eqref{eq:coupled_adjoint_2}. We define the $(p+3)$-th order tensor $\phi:=(\phi_\beta)_{1\leq\beta_1,\dots,\beta_{p+2}\leq n}$; applying Lemma~\ref{lemma:carleman_estimate_heat_adapted} to system~\eqref{eq:differentiated_adjoint}, we have a Carleman inequality for $\phi$:
\begin{align}\label{eq:first_carleman_inequality}
&\mathcal{I}(s,\lambda;\phi)\nonumber\\
&\leq C\left( s^3\lambda^4 \iint_{(0,T)\times \omega_{p+2}} e^{-2s\alpha} \xi^3 \Vert \phi\Vert_{p+3}^2 dx dt + s\lambda \iint_{\Sigma_T} e^{-2s\alpha^*} \xi^* \Vert \nabla \phi \cdot n\Vert_{p+3}^2 d\sigma dt\right)
\end{align}
for every $\lambda \geq C$ and $s\geq C(T^5+T^{10})$. The rest of this proof follows three steps:
\begin{enumerate}
\item We will estimate the boundary term on the righthand side of~\eqref{eq:first_carleman_inequality} with a global interior term involving $\tilde{\psi}$, which will be absorbed into the lefthand side;
\item we will relate $\mathcal{I}(s,\lambda;\phi)$ with the lefthand side of~\eqref{eq:carleman_estimate};
\item we will estimate the local term on the righthand side of~\eqref{eq:first_carleman_inequality} with a local term of zero differential order (as appearing in~\eqref{eq:carleman_estimate}) and some other local terms which will be absorbed into the lefthand side.
\end{enumerate}
\textbf{Step (i):} Consider a function $\theta \in C^2(\bar{\Omega})$ such that $\nabla \theta \cdot \mathbf{n} = \theta = 1$ in $\bar{\Omega}$, where $\mathbf{n}$ is the outward pointing normal of $\partial\Omega$. With this construction, $\nabla \theta = \mathbf{n}$. Indeed, for any $q\in\partial\Omega$ and for any parametrized curve $\gamma:\mathbb{R}\rightarrow \Omega$ passing through point $q$ at time $0$, we have
\[
\frac{d}{dt}\theta(\gamma(t))\big|_{t=0}=\nabla\theta\big|_q\frac{d\gamma(t)}{dt}\Bigg|_{t=0}=0,
\]
since $\theta=1$ in $\bar{\Omega}$. Hence, since $\nabla\theta$ is orthogonal to the tangent of any curve passing through any arbitrary point $q\in\partial\Omega$ at $t=0$, it must be equal to $\mathbf{n}$. Let $\beta$ and $\gamma$ be multi-indices of length $n$; we integrate the boundary term by parts to obtain
\begin{align*}
&s\lambda\iint_{\Sigma_T}e^{-2s\alpha^*}\xi^*\Vert\nabla\phi\cdot\mathbf{n}\Vert_{p+3}^2 d\sigma dt\\
&=s\lambda\sum_{\left|\beta\right|=p+3}\iint_{\Sigma_T}e^{-2s\alpha^*}\xi^*\left(\partial_\beta\psi\cdot\nabla\theta\right)\left(\partial_\beta\psi\cdot\mathbf{n}\right) d\sigma dt\\
&=\sum_{\substack{\left|\beta\right|=p+3\\ \left|\gamma\right|=p+4}}\Big(s\lambda\iint_{Q_T}e^{-2s\alpha^*}\xi^*\left(\partial_\gamma\psi\right)\left(\partial_\beta\psi\cdot\nabla\theta\right)dxdt \\
&+ s\lambda\iint_{Q_T}e^{-2s\alpha^*} \xi^*\nabla(\partial_\beta\psi\cdot \nabla \theta)\cdot\partial_\beta\psi dxdt\Big).
\end{align*}
Next, we employ Cauchy-Schwarz and Young's inequalities to obtain
\begin{align}\label{eq:inequality_boundary_term}
&s\lambda\iint_{\Sigma_T}e^{-2s\alpha^*}\xi^*\Vert\nabla\phi\cdot\mathbf{n}\Vert_{p+3}^2d\sigma dt\nonumber\\
%&\leq\lambda\int_0^Te^{-2s\alpha^*}\left(\left(\int_{\Omega}\Vert(s\xi^*)^{k}\Delta\phi\Vert_{p+5}^2dx\right)^{1/2}\left(\int_{\Omega}\Vert(s\xi^*)^{1-k}\nabla\phi\cdot\nabla\theta\Vert_{p+3}^2dx\right)^{1/2}\right.\nonumber\\
%&\left.+\left(\int_{\Omega}\Vert(s\xi^*)^{k}\nabla(\nabla\phi\cdot\nabla\theta)\Vert_{p+4}^2dx\right)^{1/2}\left(\int_{\Omega}\Vert(s\xi^*)^{1-k}\nabla\phi\Vert_{p+4}^2dx\right)^{1/2}\right)dt\nonumber\\
%&\leq C\lambda\int_0^Te^{-2s\alpha^*}\left(||(s\xi^*)^{k}\tilde{\psi}||_{H^{p+4}(\Omega)^m}||(s\xi^*)^{1-k}\tilde{\psi}||_{H^{p+3}(\Omega)^m}\right)dt\nonumber\\
&\leq C\lambda\left(\int_0^Te^{-2s\alpha^*}(s \xi^*)^{2k}||\tilde{\psi}||_{H^{p+4}(\Omega)^m}^2dt+\int_0^Te^{-2s\alpha^*}(s \xi^*)^{2-2k}||\tilde{\psi}||_{H^{p+3}(\Omega)^m}^2dt\right),
\end{align}
for $k\in(0,1)$ to be chosen later. We define $\hat{\tilde{\psi}}:=\rho\tilde{\psi}$, with $\rho\in C^\infty([0,T])$ defined by $\rho:=(s\xi^*)^ae^{-s\alpha^*}$ for some $a\in\mathbb{R}$ to be chosen later. Note that $\hat{\tilde{\psi}}(T,\cdot)=0$ in $\Omega$, since $\rho$ decays exponentially to zero as $t\rightarrow T$. Similarly, $\frac{d^i}{dt^i}\rho(0)=0$, for all $i\in\mathbb{N}$. Furthermore, $\hat{\tilde{\psi}}$ is the solution to
\begin{equation}\label{eq:psi_hat_system}
\left\{\begin{aligned}
    -\partial_t \hat{\tilde{\psi}}&=\text{div}(D\nabla\hat{\tilde{\psi}})-G^*\cdot\nabla\hat{\tilde{\psi}}+A^*\hat{\tilde{\psi}}-\frac{d}{dt}\rho\tilde{\psi},&&\text{in}\;Q_T,\cr
    \hat{\tilde{\psi}}&=0,&&\text{on}\;\Sigma_T,\cr
    \hat{\tilde{\psi}}(T,\cdot)&=0,&&\text{in}\;\Omega.
  \end{aligned}\right.
\end{equation}
Hence, by~\eqref{eq:adjoint_regularity}, one can utilize Theorem~\ref{thm:smoothing_theorem} to get the estimate
\begin{align}\label{eq:psi_hat_regularity_estimate}
&\Vert \hat{\tilde{\psi}}\Vert_{L^2((0,T);H^{2d+2}(\Omega))^m \cap H^{d+1}((0,T);L^2(\Omega))^m}\nonumber\\
&\leq C\left\Vert\frac{d}{dt}\rho\tilde{\psi}\right\Vert_{L^2((0,T);H^{2d}(\Omega))^m\cap H^{d}((0,T);L^2(\Omega))^m}
\end{align}
for $d\in\left\{0,\dots,p+2\right\}$. Owing to~\eqref{eq:alpha_weight} and~\eqref{eq:xi_weight}, we have the bound
\begin{equation}\label{eq:rho_bound1}
\left|\frac{d}{dt}\rho\right|\leq CT(s\xi^*)^{a+6/5}e^{-s\alpha^*}.
\end{equation}
Indeed, for
%\[
$\bar{c}:=\min_{x\in\bar{\Omega}}\{e^{\lambda(10\Vert\eta^0\Vert_\infty+\eta^0(x))}\}$ and
%\]
and
%\[
$\tilde{c}:=\max_{x\in\bar{\Omega}}\{e^{12\Vert\eta^0\Vert_\infty}-e^{\lambda(10\Vert\eta^0\Vert_\infty+\eta^0(x))}\} $, we have
%\]
\begin{align*}
\left|\frac{d}{dt}\rho\right| & = \left|a s (s\xi^*)^{a-1}e^{-s\alpha^*} \frac{d}{dt} \xi^* - s (s\xi^*)^a e^{-s\alpha^*}\frac{d}{dt} \alpha^* \right| \\
&=e^{-s\alpha^*}\left|s(s\xi^*)^{a-1}\frac{5(2t-T)}{t^6(T-t)^6}\left(a\bar{c}-(s\xi^*)\tilde{c}\right)\right|\\
&=(s\xi^*)^ae^{-s\alpha^*}\left|\frac{10t-5T}{t(T-t)} \left(a-\frac{(s\xi^*)\tilde{c}}{\bar{c}}\right)\right|\\
&=(s\xi^*)^{a+6/5}e^{-s\alpha^*}\left|\frac{(10t-5T)}{\bar{c}^{6/5}}\left(\frac{at^5(T-t)^5}{s^{6/5}}-\frac{\tilde{c}}{s^{1/5}}\right)\right|,
\end{align*}
and since $s\geq C(T^5+T^{10})$, one can obtain~\eqref{eq:rho_bound1}. Similarly, we have
\begin{equation}\label{eq:rho_bound2}
\left| \frac{d^r}{dt^r} \rho\right| \leq CT^r(s\xi^*)^{a+6r/5}e^{-s\alpha^*},
\end{equation}
for $r\in\mathbb{N}$. We apply~\eqref{eq:psi_hat_regularity_estimate} to $\hat{\tilde{\psi}}$ for $a=1-k$ and $d=\left\lfloor\frac{p+1}{2}\right\rfloor$ to obtain
\begin{align}\label{eq:psi_hat_inequality_1}
&\int_0^Te^{-2s\alpha^*}(s\xi^*)^{2-2k}\Vert\tilde{\psi}\Vert_{H^{2\left\lfloor\frac{p+3}{2}\right\rfloor}(\Omega)^m}^2dt\nonumber\\
&\leq C\left(\int_0^T\left\Vert\frac{d}{dt}\left(e^{-s\alpha^*}(s\xi^*)^{1-k}\right)\tilde{\psi}\right\Vert_{H^{2\left\lfloor\frac{p+1}{2}\right\rfloor}(\Omega)^m}^2dt\right.\nonumber\\
&\left.+\sum_{r=1}^{\left\lfloor\frac{p+1}{2}\right\rfloor}\int_0^T\left\Vert\frac{d^r}{dt^r}\left(\frac{d}{dt}\left(e^{-s\alpha^*}(s\xi^*)^{1-k}\right)\tilde{\psi}\right)\right\Vert_{L^2(\Omega)^m}^2 dt\right).
\end{align}
We now apply~\eqref{eq:psi_hat_regularity_estimate} to $\hat{\tilde{\psi}}=\frac{d}{dt}\rho \tilde{\psi}$ (which satisfies a system similar to~\eqref{eq:psi_hat_system} and verifies the compatibility conditions in Theorem~\ref{thm:smoothing_theorem}) for $a=1-k$ and $d=\left\lfloor\frac{p+1}{2}\right\rfloor-1$ to obtain
\begin{align}\label{eq:psi_hat_inequality_2}
&\int_0^T\left\Vert\frac{d}{dt}\left(e^{-s\alpha^*}(s\xi^*)^{1-k}\right)\tilde{\psi}\right\Vert_{H^{2\left\lfloor\frac{p+1}{2}\right\rfloor}(\Omega)^m}^2dt\nonumber\\
&+\sum_{r=1}^{\left\lfloor\frac{p+1}{2}\right\rfloor}\int_0^T\left\Vert\frac{d^r}{dt^r}\left(\frac{d}{dt}\left(e^{-s\alpha^*}(s\xi^*)^{1-k}\right)\tilde{\psi}\right)\right\Vert_{L^2(\Omega)^m}^2dt\nonumber\\
&\leq C\int_0^T\left\Vert\frac{d^2}{dt^2}\left(e^{-s\alpha^*}(s\xi^*)^{1-k}\right)\tilde{\psi}\right\Vert_{H^{2\left\lfloor\frac{p+1}{2}\right\rfloor-2}(\Omega)^m}^2dt\nonumber\\
&+\sum_{r=1}^{\left\lfloor\frac{p+1}{2}\right\rfloor-1}\int_0^T\left\Vert\frac{d^r}{dt^r}\left(\frac{d^2}{dt^2}\left(e^{-s\alpha^*}(s\xi^*)^{1-k}\right)\tilde{\psi}\right)\right\Vert_{L^2(\Omega)^m}^2dt.
\end{align}
Repeating this way $\left\lfloor\frac{p+1}{2}\right\rfloor-1$ more times and utilizing~\eqref{eq:rho_bound2} yields the inequality
\begin{align}\label{eq:psi_hat_inequality_3}
&\int_0^Te^{-2s\alpha^*}(s\xi^*)^{2-2k}\Vert\tilde{\psi}\Vert_{H^{2\left\lfloor\frac{p+3}{2}\right\rfloor}(\Omega)^m}^2dt\nonumber\\
&\leq C\int_0^T\left\Vert\frac{d^{\left\lfloor\frac{p+1}{2}\right\rfloor+1}}{dt^{\left\lfloor\frac{p+1}{2}\right\rfloor+1}}\left(e^{-s\alpha^*}(s\xi^*)^{1-k}\right)\tilde{\psi}\right\Vert_{L^2(\Omega)^m}^2dt\nonumber\\
&\leq CT^{2\left\lfloor\frac{p+1}{2}\right\rfloor+2}\int_0^Te^{-2s\alpha^*}(s\xi^*)^{2-2k+\frac{12}{5}\left(\left\lfloor\frac{p+1}{2}\right\rfloor+1\right)}\Vert\tilde{\psi}\Vert_{L^2(\Omega)^m}^2dt.
\end{align}
We can get very similar estimates~\eqref{eq:psi_hat_inequality_1} and~\eqref{eq:psi_hat_inequality_2} for $a=3k-1$, $d=\left\lceil\frac{p+2}{2}\right\rceil$, and by using~\eqref{eq:rho_bound2}, we obtain
\begin{align}\label{eq:psi_hat_inequality_4}
&\int_0^Te^{-2s\alpha^*}(s\xi^*)^{6k-2}\Vert\tilde{\psi}\Vert_{H^{2\left\lceil\frac{p+4}{2}\right\rceil}(\Omega)^m}^2dt\nonumber\\
&\leq C\int_0^T\left\Vert\frac{d^{\left\lceil\frac{p+2}{2}\right\rceil+1}}{dt^{\left\lceil\frac{p+2}{2}\right\rceil+1}}\left(e^{-s\alpha^*}(s\xi^*)^{3k-1}\right)\tilde{\psi}\right\Vert_{L^2(\Omega)^m}^2 dt\nonumber\\
&\leq CT^{2\left\lceil\frac{p+2}{2}\right\rceil+2}\int_0^Te^{-2s\alpha^*}(s\xi^*)^{6k-2+\frac{12}{5}\left(\left\lceil\frac{p+2}{2}\right\rceil+1\right)}\Vert\tilde{\psi}\Vert_{L^2(\Omega)^m}^2dt.
\end{align}
Suppose for the moment that $p$ is odd. By applying Theorem~\ref{eq:interpolation_inequality} to the appropriate spatial derivative of $\tilde{\psi}$ with $j=1,\; m=q=p=r=2$ and $\alpha=1/2$, and then employing the Cauchy-Schwarz inequality, we obtain
\begin{align*}
&\scalemath{0.96}{\int_0^Te^{-2s\alpha^*}(s\xi^*)^{2k}\Vert\tilde{\psi}\Vert_{H^{p+4}(\Omega)^m}^2dt}\\
&\scalemath{0.96}{\leq C\int_0^T\Vert e^{-s\alpha^*}(s\xi^*)^{3k-1}\tilde{\psi}\Vert_{H^{2\left\lceil\frac{p+4}{2}\right\rceil}(\Omega)^m}\Vert e^{-s\alpha^*}(s\xi^*)^{1-k}\tilde{\psi}\Vert_{H^{2\left\lfloor\frac{p+3}{2}\right\rfloor}(\Omega)^m}dt}\\
&\scalemath{0.96}{\leq C\hspace{-4pt}\left(\int_0^T\hspace{-5pt}e^{-2s\alpha^*}(s\xi^*)^{6k-2}\Vert\tilde{\psi}\Vert_{H^{2\left\lceil\frac{p+4}{2}\right\rceil}(\Omega)^m}^2dt\right)^\frac{1}{2}\hspace{-6pt}\left(\int_0^T\hspace{-5pt}e^{-2s\alpha^*}(s\xi^*)^{2-2k}\Vert\tilde{\psi}\Vert_{H^{2\left\lfloor\frac{p+3}{2}\right\rfloor}(\Omega)^m}^2dt\right)^\frac{1}{2}.}
\end{align*}
Choosing $k=\frac{1}{2}+\frac{3}{10}\left(\left\lfloor\frac{p+1}{2}\right\rfloor-\left\lceil\frac{p+2}{2}\right\rceil\right)$ verifies
\[
2-2k+\frac{12}{5}\left(\left\lfloor\frac{p+1}{2}\right\rfloor+1\right)=6k-2+\frac{12}{5}\left(\left\lceil\frac{p+2}{2}\right\rceil+1\right),
\]
and hence by utilizing~\eqref{eq:psi_hat_inequality_3} and~\eqref{eq:psi_hat_inequality_4}, we obtain
\begin{align}\label{eq:p_odd_bound}
&\int_0^Te^{-2s\alpha^*}(s\xi^*)^{2k}\Vert\tilde{\psi}\Vert_{H^{p+4}(\Omega)^m}^2dt\nonumber\\
&\leq CT^{\left\lceil\frac{p+2}{2}\right\rceil+\left\lfloor\frac{p+1}{2}\right\rfloor+2}\int_0^Te^{-2s\alpha^*}(s\xi^*)^{\frac{17}{5}+\frac{9}{5}\left\lfloor\frac{p+1}{2}\right\rfloor+\frac{3}{5}\left\lceil\frac{p+2}{2}\right\rceil}\Vert\tilde{\psi}\Vert_{L^2(\Omega)^m}^2dt.
\end{align}
Identical steps can be followed for the case when $p$ is even to obtain
\begin{align}\label{eq:p_even_bound}
&\int_0^Te^{-2s\alpha^*}(s\xi^*)^{2-2k}\Vert\tilde{\psi}\Vert_{H^{p+3}(\Omega)^m}^2dt\nonumber\\
&\leq CT^{\left\lceil\frac{p+2}{2}\right\rceil+\left\lfloor\frac{p+1}{2}\right\rfloor+2}\int_0^Te^{-2s\alpha^*}(s\xi^*)^{\frac{17}{5}+\frac{3}{5}\left\lfloor\frac{p+1}{2}\right\rfloor+\frac{9}{5}\left\lceil\frac{p+2}{2}\right\rceil}\Vert\tilde{\psi}\Vert_{L^2(\Omega)^m}^2dt.
\end{align}
It follows from~\eqref{eq:inequality_boundary_term},~\eqref{eq:psi_hat_inequality_3} and~\eqref{eq:p_odd_bound} that
\begin{align*}
&\scalemath{0.98}{s\lambda\iint_{\Sigma_T}e^{-2s\alpha^*}\xi^*\Vert\nabla\phi\cdot n\Vert_{p+3}^2d\sigma dt}\\
&\scalemath{0.98}{\leq C\lambda\left(T^{2\left\lfloor\frac{p+1}{2}\right\rfloor+2}+T^{\left\lceil\frac{p+2}{2}\right\rceil+\left\lfloor\frac{p+1}{2}\right\rfloor+2}\right)\int_0^Te^{-2s\alpha^*}(s\xi^*)^{\frac{17}{5}+\frac{9}{5}\left\lfloor\frac{p+1}{2}\right\rfloor+\frac{3}{5}\left\lceil\frac{p+2}{2}\right\rceil}\Vert\tilde{\psi}\Vert_{L^2(\Omega)^m}^2dt,}
\end{align*}
for $p$ odd, and it follows from~\eqref{eq:inequality_boundary_term},~\eqref{eq:psi_hat_inequality_4} and~\eqref{eq:p_even_bound}
\begin{align*}
&\scalemath{0.98}{s\lambda\iint_{\Sigma_T}e^{-2s\alpha^*}\xi^*\Vert\nabla\phi\cdot n\Vert_{p+3}^2d\sigma dt}\\
&\scalemath{0.98}{\leq C\lambda\left(T^{2\left\lceil\frac{p+2}{2}\right\rceil+2}+T^{\left\lceil\frac{p+2}{2}\right\rceil+\left\lfloor\frac{p+1}{2}\right\rfloor+2}\right)\int_0^Te^{-2s\alpha^*}(s\xi^*)^{\frac{17}{5}+\frac{3}{5}\left\lfloor\frac{p+1}{2}\right\rfloor+\frac{9}{5}\left\lceil\frac{p+2}{2}\right\rceil}\Vert\tilde{\psi}\Vert_{L^2(\Omega)^m}^2dt,}
\end{align*}
for $p$ even. In what follows, we choose $p$ even without loss of generality (the exact same technique can be used for $p$ odd), and since
\[
\left(T^{2\left\lceil\frac{p+2}{2}\right\rceil+2}+T^{\left\lceil\frac{p+2}{2}\right\rceil+\left\lfloor\frac{p+1}{2}\right\rfloor+2}\right) \leq Cs^{2p-\frac{3}{5}\left\lfloor\frac{p+1}{2}\right\rfloor-\frac{9}{5}\left\lceil\frac{p+2}{2}\right\rceil +\frac{17}{5}},
\]
for $s\geq C(T^5+T^{10})$, we use~\eqref{eq:alpha_star_weight} and~\eqref{eq:xi_star_weight} to obtain
\begin{align*}
s\lambda\iint_{\Sigma_T}&e^{-2s\alpha^*}\xi^*\Vert\nabla\phi\cdot n\Vert_{p+3}^2d\sigma dt\\
&\leq Cs^{2p+34/5}\lambda\int_0^Te^{-2s\alpha^*}(\xi^*)^{\frac{17}{5}+\frac{9}{5}\left\lfloor\frac{p+1}{2}\right\rfloor+\frac{3}{5}\left\lceil\frac{p+2}{2}\right\rceil}\Vert\tilde{\psi}\Vert_{L^2(\Omega)^m}^2dt\\
&\leq Cs^{2p+34/5}\lambda\iint_{Q_T} e^{-2s\alpha}\xi^{\frac{17}{5}+\frac{9}{5}\left\lfloor\frac{p+1}{2}\right\rfloor+\frac{3}{5}\left\lceil\frac{p+2}{2}\right\rceil}\left\Vert \tilde{\psi}\right\Vert_1^2dxdt.
\end{align*}
Denoting by $l(p)$ the exponent $\frac{17}{5}+\frac{9}{5}\left\lfloor\frac{p+1}{2}\right\rfloor+\frac{3}{5}\left\lceil\frac{p+2}{2}\right\rceil$, we arrive at the end of Step~(i) to conclude that
\begin{align}\label{eq:step_one_conclusion}
&\mathcal{I}(s,\lambda;\phi)\nonumber\\
&\leq C\left(s^3\lambda^4\iint_{(0,T)\times\omega_{p+2}}e^{-2s\alpha}\xi^3\Vert \phi\Vert_{p+3}^2dxdt+s^{2p+34/5}\lambda\iint_{Q_T}e^{-2s\alpha}\xi^{l(p)}\left\Vert \tilde{\psi}\right\Vert_1^2dxdt\right)
\end{align}
for every $\lambda\geq C$ and $s\geq C(T^5+T^{10})$.\\
\textbf{Step (ii):} In this step, we relate $\mathcal{I}(s,\lambda;\phi)$ to the lefthand side of~\eqref{eq:carleman_estimate}. We apply Lemma~\ref{lemma:coron_result} to $\tilde{\psi}$ for $r=2p+5$ to obtain
\begin{align}\label{eq:step_2_inequality_1}
&s^{2p+7}\lambda^{2p+8}\iint_{Q_T}e^{-2s\alpha}\xi^{2p+7}\left\Vert \tilde{\psi}\right\Vert_1^2dxdt \nonumber\\
&\leq C\Bigg(s^{2p+5}\lambda^{2p+6}\iint_{Q_T}e^{-2s\alpha}\xi^{2p+5}\left\Vert\nabla\tilde{\psi}\right\Vert_{2}^2dxdt\nonumber\\
&\left.+s^{2p+7}\lambda^{2p+8}\iint_{(0,T)\times\omega_{p+2}}e^{-2s\alpha}\xi^{2p+7}\left\Vert \tilde{\psi}\right\Vert_1^2dxdt\right),
\end{align}
for every $\lambda\geq C$ and $s\geq C(T^5+T^{10})$. Similarly, for $k\in\{0,\dots,p\}$, we apply Lemma~\ref{lemma:coron_result} to $\nabla^{p+1-k}\tilde{\psi}$ for $r=2k+3$ to obtain
\begin{align}\label{eq:step_2_inequality_2}
&s^{2k+5}\lambda^{2k+6}\iint_{Q_T}e^{-2s\alpha}\xi^{2k+5}\Vert\nabla^{p+1-k}\tilde{\psi}\Vert_{p+2-k}^2dxdt\nonumber\\
&\leq C\Bigg(s^{2k+3}\lambda^{2k+4}\iint_{Q_T}e^{-2s\alpha}\xi^{2k+3}\left\Vert\nabla^{p+2-k}\tilde{\psi}\right\Vert_{p+3-k}^2dxdt\nonumber\\
&\left.+s^{2k+5}\lambda^{2k+6}\iint_{(0,T)\times\omega_{p+2}}e^{-2s\alpha}\xi^{2k+5}\Vert\nabla^{p+1-k}\tilde{\psi}\Vert_{p+2-k}^2dxdt\right),
\end{align}
for every $\lambda\geq C$ and $s\geq C(T^5+T^{10})$. One can upper bound the first term in the righthand side of~\eqref{eq:step_2_inequality_1} by~\eqref{eq:step_2_inequality_2} for $k=p$ and continue this way by backwards iteration on $k$. The global terms on the righthand side of~\eqref{eq:step_2_inequality_2} can be absorbed in the exact same way. Hence, a combination of~\eqref{eq:step_one_conclusion},~\eqref{eq:step_2_inequality_1} and~\eqref{eq:step_2_inequality_2} gives
\begin{align*}
&\iint_{Q_T}e^{-2s\alpha}\sum_{k=1}^{p+4}s^{2k-1}\lambda^{2k}\xi^{2k-1}\Vert\nabla^{p+4-k}\tilde{\psi}\Vert_{p+5-k}^2dxdt\\
&\leq C\Bigg(\iint_{(0,T)\times\omega_{p+2}}e^{-2s\alpha}\sum_{k=2}^{p+4}s^{2k-1}\lambda^{2k}\xi^{2k-1}\Vert\nabla^{p+4-k}\tilde{\psi}\Vert_{p+5-k}^2dxdt\\
&+s^{3}\lambda^{4}\iint_{Q_T}e^{-2s\alpha}\xi^{3}\Vert\nabla^{p+2}\tilde{\psi}\Vert_{p+3}^2dxdt+s^{2p+34/5}\lambda\iint_{Q_T} e^{-2s\alpha}\xi^{l(p)}\left\Vert \tilde{\psi}\right\Vert_1^2 dxdt\Bigg),
\end{align*}
for every $\lambda\geq C$ and $s\geq C(T^5+T^{10})$. By utilizing~\eqref{eq:step_one_conclusion} once more, we arrive at the inequality
\begin{align}\label{eq:step_2_conclusion}
&\iint_{Q_T} e^{-2s\alpha}\sum_{k=1}^{p+4} s^{2k-1}\lambda^{2k}\xi^{2k-1}\Vert\nabla^{p+4-k}\tilde{\psi}\Vert_{p+5-k}^2dxdt\nonumber\\
&\leq C\Bigg(\iint_{(0,T)\times\omega_{p+2}}e^{-2s\alpha}\sum_{k=2}^{p+4}s^{2k-1}\lambda^{2k}\xi^{2k-1}\Vert\nabla^{p+4-k}\tilde{\psi}\Vert_{p+5-k}^2dxdt\nonumber\\
&+s^{2p+34/5}\lambda\iint_{Q_T} e^{-2s\alpha}\xi^{l(p)}\left\Vert \tilde{\psi}\right\Vert_1^2 dxdt\Bigg),
\end{align}
which is verified for every $\lambda\geq C$ and $s\geq C(T^5+T^{10})$.\\
\textbf{Step (iii):} In this final step, we absorb the higher-order local terms in the righthand side of~\eqref{eq:step_2_conclusion}. Consider the function $\theta_{p+1}\in C^2(\bar{\Omega})$ satisfying
\begin{equation}\label{eq:theta_1_definition}
\left\{\begin{aligned}
    &\text{Supp}(\theta_{p+1}) \subseteq \omega_{p+1}, && \cr
    &\theta_{p+1}=1, &&\text{in} \; \omega_{p+2}, \cr
    &0\leq\theta_{p+1}\leq 1 &&\text{in} \; \Omega.
  \end{aligned}\right.
\end{equation}
Let $\beta$ be a multi-index of length $n$. Since $\bar{\omega}_{p+2}\subset \omega_{p+1}$, where $\omega_{p+1}$ is an open subset of $\Omega$, we integrate the rightmost term in~\eqref{eq:step_2_conclusion} by parts and employ the the Cauchy-Schwarz inequality to obtain
\begin{align}\label{eq:step3_bound0}
&\scalemath{0.96}{s^3\lambda^4\iint_{(0,T)\times\omega_{p+2}}e^{-2s\alpha}\xi^3\left\Vert\nabla^{p+2}\tilde{\psi}\right\Vert_{p+3}^2dxdt}\nonumber\\
&\scalemath{0.96}{\leq s^3\lambda^4\iint_{(0,T)\times\omega_{p+1}}\theta_{p+1}e^{-2s\alpha}\xi^3\left\Vert\nabla^{p+2}\tilde{\psi}\right\Vert_{p+3}^2dxdt}\nonumber\\
&\scalemath{0.96}{=-s^3\lambda^4\iint_{(0,T)\times\omega_{p+1}}\sum_{\substack{i=1\\|\beta|=p+1}}^n\left(\partial_i (\theta_{p+1} e^{-2s\alpha}\xi^3) \partial_i\partial_\beta \tilde{\psi} + \theta_{p+1}e^{-2s\alpha}\xi^3\partial_i^2\partial_\beta \tilde{\psi}\right)\left(\partial_\beta\tilde{\psi}\right)dxdt}\nonumber\\
&\scalemath{0.96}{\leq s^3\lambda^4\iint_{(0,T)\times\omega_{p+1}}\Bigg(\left\Vert \nabla\left(\theta_{p+1}e^{-2s\alpha}\xi^3\right)\right\Vert_1\left\Vert\nabla^{p+2}\tilde{\psi}\right\Vert_{p+3}\left\Vert\nabla^{p+1}\tilde{\psi}\right\Vert_{p+2}}\nonumber\\
&\scalemath{0.96}{+\theta_{p+1}e^{-2s\alpha}\xi^3\left\Vert\nabla^{p+3}\tilde{\psi}\right\Vert_{p+4}\left\Vert\nabla^{p+1}\tilde{\psi}\right\Vert_{p+2}\Bigg)dxdt.}
\end{align}
By~\eqref{eq:alpha_weight} and~\eqref{eq:xi_weight}, we have that
\begin{equation}\label{eq:step3_bound1}
\left\Vert \nabla\left(\theta_{p+1}e^{-2s\alpha}\xi^3\right)\right\Vert_1\leq Cs\lambda e^{-2s\alpha}\xi^4.
\end{equation}
Indeed,
\begin{align*}
\left\Vert \nabla\left(\theta_{p+1}e^{-2s\alpha}\xi^3\right)\right\Vert_1&=\left\Vert e^{-2s\alpha}\xi^3\left(\nabla\theta_{p+1}+2s\lambda\theta_{p+1}\xi\nabla\eta^0+3\lambda\theta_{p+1}\nabla\eta^0\right)\right\Vert_1\\
&=s\lambda e^{-2s\alpha}\xi^{4}\left\Vert \frac{\nabla\theta_{p+1}}{s\lambda\xi}+2\theta_{p+1}\nabla\eta^0+\frac{3\theta_{p+1}\nabla\eta^0}{s\xi}\right\Vert_1,
\end{align*}
and since $s\geq C(T^5+T^{10})$,~\eqref{eq:step3_bound1} is verified.
Hence, by~\eqref{eq:theta_1_definition},~\eqref{eq:step3_bound1} and using Young's inequality with $\epsilon>0$, we have
\begin{align}\label{eq:step3_bound2}
&s^3\lambda^4\iint_{(0,T)\times\omega_{p+2}}e^{-2s\alpha}\xi^3\left\Vert\nabla^{p+2}\tilde{\psi}\right\Vert_{p+3}^2dxdt\nonumber\\
&\leq Cs^3\lambda^4\iint_{(0,T)\times\omega_{p+1}}\left(s\lambda e^{-2s\alpha}\xi^4\left\Vert\nabla^{p+2}\tilde{\psi}\right\Vert_{p+3}\left\Vert\nabla^{p+1}\tilde{\psi}\right\Vert_{p+2}\right.\nonumber\\
&\left.+e^{-2s\alpha}\xi^3\left\Vert\nabla^{p+3}\tilde{\psi}\right\Vert_{p+4}\left\Vert\nabla^{p+1}\tilde{\psi}\right\Vert_{p+2}\right)dxdt\nonumber\\
& \leq C\iint_{(0,T)\times\omega_{p+1}}e^{-2s\alpha}\left(\epsilon s^3\lambda^4\xi^3\left\Vert\nabla^{p+2}\tilde{\psi}\right\Vert_{p+3}^2+\epsilon s\lambda^2\xi\left\Vert\nabla^{p+3}\tilde{\psi}\right\Vert_{p+4}^2\right.\nonumber\\
&\left.+\frac{2}{\epsilon}s^5\lambda^6\xi^5\left\Vert\nabla^{p+1}\tilde{\psi}\right\Vert_{p+2}^2\right)dxdt.
\end{align}
Observe that the first two terms in the righthand side of~\eqref{eq:step3_bound2} can be bounded above by employing~\eqref{eq:step_2_conclusion} and~\eqref{eq:step3_bound2} recursively: indeed, by positivity of the integrand in $Q_T$ and by~\eqref{eq:step_2_conclusion}, we obtain
\begin{align}\label{eq:step3_bound3}
&\scalemath{0.96}{\epsilon\iint_{(0,T)\times\omega_{p+1}}e^{-2s\alpha}\left(s^3\lambda^4\xi^3\left\Vert\nabla^{p+2}\tilde{\psi}\right\Vert_{p+3}^2+s\lambda^2\xi\left\Vert\nabla^{p+3}\tilde{\psi}\right\Vert_{p+4}^2\right)dxdt}\nonumber\\
&\scalemath{0.96}{\leq C\epsilon\Bigg(\iint_{(0,T)\times\omega_{p+2}}e^{-2s\alpha}\sum_{k=2}^{p+4}s^{2k-1}\lambda^{2k}\xi^{2k-1}\Vert\nabla^{p+4-k}\tilde{\psi}\Vert_{p+5-k}^2dxdt}\nonumber\\
&\scalemath{0.96}{+s^{2p+34/5}\lambda\iint_{Q_T}e^{-2s\alpha}\xi^{l(p)}\left\Vert \tilde{\psi}\right\Vert_1^2dxdt\Bigg)}\nonumber\\
&\scalemath{0.96}{=C\epsilon\Bigg(\iint_{(0,T)\times\omega_{p+2}}e^{-2s\alpha}\sum_{k=3}^{p+4}s^{2k-1}\lambda^{2k}\xi^{2k-1}\Vert\nabla^{p+4-k}\tilde{\psi}\Vert_{p+5-k}^2dxdt}\nonumber\\
&\scalemath{0.96}{+s^3\lambda^4\iint_{(0,T)\times\omega_{p+2}}e^{-2s\alpha}\xi^3\left\Vert\nabla^{p+2}\tilde{\psi}\right\Vert_{p+3}^2dxdt+s^{2p+34/5}\lambda\iint_{Q_T} e^{-2s\alpha}\xi^{l(p)}\left\Vert \tilde{\psi}\right\Vert_1^2dxdt\Bigg),}
\end{align}
for $\lambda\geq C$ and $s\geq C(T^5+T^{10})$. Combining~\eqref{eq:step3_bound3} and~\eqref{eq:step3_bound2} yields
\begin{align}\label{eq:step3_bound4}
&\epsilon\iint_{(0,T)\times\omega_{p+1}}e^{-2s\alpha}\left(s^3\lambda^4\xi^3\left\Vert\nabla^{p+2}\tilde{\psi}\right\Vert_{p+3}^2+s\lambda^2\xi\left\Vert\nabla^{p+3}\tilde{\psi}\right\Vert_{p+4}^2\right) dxdt\nonumber\\
&\leq C\Bigg(\epsilon\iint_{(0,T)\times\omega_{p+2}}e^{-2s\alpha}\sum_{k=3}^{p+4}s^{2k-1}\lambda^{2k}\xi^{2k-1}\Vert\nabla^{p+4-k}\tilde{\psi}\Vert_{p+5-k}^2dxdt\nonumber\\
&+\iint_{(0,T)\times\omega_{p+1}}e^{-2s\alpha}\epsilon^2\left(s^3\lambda^4\xi^3\left\Vert\nabla^{p+2}\tilde{\psi}\right\Vert_{p+3}^2+s\lambda^2\xi\left\Vert\nabla^{p+3}\tilde{\psi}\right\Vert_{p+4}^2\right)\nonumber\\
&+\iint_{(0,T)\times\omega_{p+1}}e^{-2s\alpha}2s^5\lambda^6\xi^5\left\Vert\nabla^{p+1}\tilde{\psi}\right\Vert_{p+2}^2dxdt\nonumber\\
&+\epsilon s^{2p+34/5}\lambda\iint_{Q_T} e^{-2s\alpha}\xi^{l(p)}\left\Vert \tilde{\psi}\right\Vert_1^2dxdt\Bigg),
\end{align}
for $\lambda\geq C$ and $s\geq C(T^5+T^{10})$. Using the same treatment by adapting~\eqref{eq:step3_bound2}, one can bound the terms with $\epsilon^2$ in~\eqref{eq:step3_bound4}; after $r$ of these recursions, 
\begin{align*}
&\epsilon\iint_{(0,T)\times\omega_{p+1}}e^{-2s\alpha}\left(s^3\lambda^4\xi^3\left\Vert\nabla^{p+2}\tilde{\psi}\right\Vert_{p+3}^2+s\lambda^2\xi\left\Vert\nabla^{p+3}\tilde{\psi}\right\Vert_{p+4}^2\right) dxdt\\
&\leq C\sum_{j=1}^r\Bigg(\epsilon^j \iint_{(0,T)\times\omega_{p+2}}e^{-2s\alpha}\sum_{k=3}^{p+4}s^{2k-1}\lambda^{2k}\xi^{2k-1}\Vert\nabla^{p+4-k}\tilde{\psi}\Vert_{p+5-k}^2dxdt\\
&+\epsilon^{2(r+1)}\iint_{(0,T)\times\omega_{p+1}} e^{-2s\alpha}\left(s^3\lambda^4\xi^3\left\Vert\nabla^{p+2}\tilde{\psi}\right\Vert_{p+3}^2+s\lambda^2\xi\left\Vert\nabla^{p+3}\tilde{\psi}\right\Vert_{p+4}^2\right.\\
&\left.+2js^5\lambda^6\xi^5\left\Vert\nabla^{p+1}\tilde{\psi}\right\Vert_{p+2}^2\right)dxdt+\epsilon^js^{2p+34/5}\lambda\iint_{Q_T} e^{-2s\alpha}\xi^{l(p)}\left\Vert \tilde{\psi}\right\Vert_1^2dxdt\Bigg),
\end{align*}
for $\lambda\geq C$ and $s\geq C(T^5+T^{10})$. Taking $\epsilon$ sufficiently small and using~\eqref{eq:step3_bound2},
\begin{align}\label{eq:step3_bound5}
&s^3\lambda^4\iint_{(0,T)\times\omega_{p+2}}e^{-2s\alpha}\xi^3\left\Vert\nabla^{p+2}\tilde{\psi}\right\Vert_{p+3}^2dxdt\nonumber\\
&\leq C\Bigg(\iint_{(0,T)\times\omega_{p+2}}e^{-2s\alpha}\sum_{k=3}^{p+4}s^{2k-1}\lambda^{2k}\xi^{2k-1}\Vert\nabla^{p+4-k}\tilde{\psi}\Vert_{p+5-k}^2dxdt\nonumber\\
&+s^{2p+34/5}\lambda\iint_{Q_T} e^{-2s\alpha}\xi^{l(p)}\left\Vert \tilde{\psi}\right\Vert_1^2dxdt\Bigg),
\end{align}
for $\lambda\geq C$ and $s\geq C(T^5+T^{10})$, since by~\eqref{eq:step3_bound0}, if $\left\Vert\nabla^{p+2}\tilde{\psi}\right\Vert_{p+3}=0$, then so does $\left\Vert\nabla^{p+3}\tilde{\psi}\right\Vert_{p+4}$. Hence from~\eqref{eq:step3_bound5}, we obtain
\begin{align}\label{eq:step3_bound6}
\iint_{(0,T)\times\omega_{p+2}}&e^{-2s\alpha}\sum_{k=2}^{p+4}s^{2k-1}\lambda^{2k}\xi^{2k-1}\Vert\nabla^{p+4-k}\tilde{\psi}\Vert_{p+5-k}^2dxdt\nonumber\\
&\leq C\iint_{(0,T)\times\omega_{p+1}}e^{-2s\alpha}\sum_{k=3}^{p+4}s^{2k-1}\lambda^{2k}\xi^{2k-1}\Vert\nabla^{p+4-k}\tilde{\psi}\Vert_{p+5-k}^2dxdt,
\end{align}
for $\lambda\geq C$ and $s\geq C(T^5+T^{10})$. For $r\in\{1,\dots,p+1\}$, consider the functions $\theta_{r}\in C^2(\bar{\Omega})$ satisfying
\[
\left\{\begin{aligned}
    &\text{Supp}(\theta_{p+1-r})\subseteq\omega_{p+1-r},&&\cr
    &\theta_{p+1-r}=1,&&\text{in}\;\omega_{p+2-r},\cr
    &0\leq\theta_{p+1-k}\leq 1,&&\text{in}\;\Omega.
  \end{aligned}\right.
\]
Using the exact same approach as was used for $r=0$, one obtains the estimate
\begin{align*}
s^{2r+3}\lambda^{2r+4}&\iint_{(0,T)\times\omega_{p+2-r}}e^{-2s\alpha}\xi^{2r+3}\left\Vert\nabla^{p+2-r}\tilde{\psi}\right\Vert_{p+3-r}^2dxdt\\
&\leq C\Bigg(\iint_{(0,T)\times\omega_{p+2}}e^{-2s\alpha}\sum_{k=3+r}^{p+4}s^{2k-1}\lambda^{2k}\xi^{2k-1}\Vert\nabla^{p+4-k}\tilde{\psi}\Vert_{p+5-k}^2dxdt\nonumber\\
&+s^{2p+34/5}\lambda\iint_{Q_T} e^{-2s\alpha}\xi^{l(p)}\left\Vert \tilde{\psi}\right\Vert_1^2dxdt\Bigg),
\end{align*}
for $\lambda\geq C$ and $s\geq C(T^5+T^{10})$. Hence, it follows that
\begin{align}\label{eq:step3_bound7}
&\iint_{Q_T} e^{-2s\alpha}\sum_{k=1}^{p+4} s^{2k-1}\lambda^{2k}\xi^{2k-1}\Vert\nabla^{p+4-k}\tilde{\psi}\Vert_{p+5-k}^2dxdt\nonumber\\
&\leq C\Bigg(s^{2p+7}\lambda^{2p+8}\iint_{(0,T)\times\omega_{0}}e^{-2s\alpha}\xi^{2p+7}\left\Vert \tilde{\psi}\right\Vert_1^2dxdt\nonumber\\
&+s^{2p+34/5}\lambda\iint_{Q_T} e^{-2s\alpha}\xi^{l(p)}\left\Vert \tilde{\psi}\right\Vert_1^2 dxdt\Bigg),
\end{align}
for $\lambda\geq C$ and $s\geq C(T^5+T^{10})$. Finally, by~\eqref{eq:xi_weight} we have the estimate
\[
s^{2p+34/5}\lambda\iint_{Q_T} e^{-2s\alpha}\xi^{l(p)}\left\Vert \tilde{\psi}\right\Vert_1^2dxdt\leq Cs^{2p+7}\lambda^{2p+8}\iint_{Q_T} e^{-2s\alpha}\xi^{2p+7}\left\Vert \tilde{\psi}\right\Vert_1^2 dxdt,
\]
for $\lambda\geq C$ and $s\geq C(T^5+T^{10})$ large enough; from now on, we denote this choice of $s$ by $s_0$. Hence, one can absorb the global term in the righthand side of~\eqref{eq:step3_bound7} into its lefthand side, and thus~\eqref{eq:carleman_estimate} is verified.
\end{proof}
\begin{proof}\longthmtitle{Proof of Proposition~\ref{prop:observability_inequality}}\label{proof:A2}
We denote by $C$ various positive constant depending on $\Omega$ and $\omega_0$. From~\eqref{eq:carleman_estimate}, we deduce
\begin{equation}\label{eq:carleman_estimate_condensed}
\iint_{Q_T} e^{-2s\alpha}\xi^{2p+7}\left\Vert \tilde{\psi}\right\Vert_1^2 dxdt \leq C\iint_{(0,T)\times\omega_0} e^{-2s\alpha} \xi^{2p+7}\left\Vert \tilde{\psi}\right\Vert_1^2 dx dt,
\end{equation}
for $\lambda \geq C$ and $s\geq s_0$. Note that for $t\in\left[\frac{T}{4},\frac{3T}{4}\right]$, we have
\begin{align}\label{eq:min_weight_function}
&\min_{t\in\left[\frac{T}{4},\frac{3T}{4}\right]}\{e^{-2s\alpha}\xi^{2p+7}\}\nonumber\\
&=\left(e^{-2s\alpha}\xi^{2p+7}\right)\left(\frac{T}{4},\cdot\right)=\left(e^{-2s\alpha}\xi^{2p+7}\right)\left(\frac{3T}{4},\cdot\right)\nonumber\\
&= \left(e^{-2s\frac{4^{10}}{3^5}\left(\frac{e^{12\lambda\Vert\eta^0\Vert_\infty} - e^{\lambda(10\Vert\eta^0\Vert_\infty+\eta^0(x))}}{T^{10}}\right)}\right) \left(\frac{4^{10}e^{(2p+7)\lambda (10\Vert\eta^0\Vert_\infty+\eta^0(x))}}{3^5T^{10}}\right).
\end{align}
We can choose $s$ sufficiently large such that 
\begin{equation}\label{eq:weight_lower_bound}
\frac{4^{10}}{3^5T^{10}}e^{-\frac{s}{T^{10}}} \leq e^{-2s\alpha}\xi^{2p+7},
\end{equation}
for all $t\in\left[\frac{T}{4},\frac{3T}{4}\right]$. Indeed, choosing
\begin{align}\label{eq:choice_of_s}
s& \geq s_1:=\max\left\{s_0,\left(\frac{3^5(2p+7)\lambda}{4^{10}}\right)\max_{x\in\bar{\Omega}}\left\{\frac{10\Vert\eta^0\Vert_\infty + \eta^0(x)}{e^{12\lambda\Vert\eta^0\Vert_\infty}-e^{\lambda(10\Vert\eta^0\Vert_\infty + \eta^0(x))}}\right\}\right\}
\end{align}
in~\eqref{eq:min_weight_function} will ensure that~\eqref{eq:weight_lower_bound} is verified. Note that we can write $s_1$ as $s_1 = \sigma \left(T^5+T^{10}\right)$, where $\sigma>0$ depends only on $\Omega$ and $\omega_0$. Fixing $s=s_1$ from now on, we deduce from~\eqref{eq:carleman_estimate_condensed} and \eqref{eq:weight_lower_bound} that
\begin{align*}
\iint_{\left(\frac{T}{4},\frac{3T}{4}\right)\times\Omega} \left\Vert \tilde{\psi}\right\Vert_1^2 dxdt \leq CT^{10}e^{C\left(1+1/T^5\right)} \iint_{(0,T)\times\omega_0} e^{-2s_1\alpha} \xi^7\left\Vert \tilde{\psi}\right\Vert_1^2 dx dt
\end{align*}
for every $\lambda\geq C$ and $s\geq s_1$. We claim that
\begin{equation}\label{eq:quarter_T_lower_bound}
\int_\Omega\left\Vert \tilde{\psi}(T/4,\cdot)\right\Vert_1^2dx\leq \frac{C}{T}e^{CT/2}\iint_{\left(\frac{T}{4},\frac{3T}{4}\right)\times\Omega}\left\Vert \tilde{\psi}\right\Vert_1^2dxdt
\end{equation}
and
\begin{equation}\label{eq:time_zero_lower_bound}
\int_\Omega\left\Vert \tilde{\psi}(0,\cdot)\right\Vert_1^2dx\leq e^{CT/4}\int_\Omega\left\Vert \tilde{\psi}(T/4,\cdot)\right\Vert_1dx,
\end{equation}
from which we can deduce~\eqref{eq:observability_inequality_carleman}. Indeed, we can multiply system~\eqref{eq:coupled_adjoint_2} by $\tilde{\psi}$, integrate the resulting equation by parts over $\Omega$ and use the Cauchy-Schwarz and Young's inequalities to obtain
\begin{align*}
-\frac{1}{2}\frac{d}{dt}\int_\Omega\left\Vert \tilde{\psi}\right\Vert_1^2dx+D\int_\Omega\Vert\nabla\tilde{\psi}\Vert_2^2dx
&=-\int_\Omega \left(\partial_t\tilde{\psi}\right)\tilde{\psi} dx+\int_\Omega\text{div}(D\nabla\tilde{\psi})\tilde{\psi} dx\\
&=-\int_\Omega\left(G^*\cdot\nabla\tilde{\psi} \right)\tilde{\psi} dx+\int_\Omega \left(A^*\tilde{\psi}\right)\tilde{\psi} dx\\
&\leq\frac{1}{2}\int_\Omega\left\Vert G^*\cdot\nabla\tilde{\psi}\right\Vert_1^2dx+\left(1+\frac{\Vert A^*\Vert_\infty}{2}\right)\int_\Omega \left\Vert \tilde{\psi}\right\Vert_1^2dx.
\end{align*}
Hence, since~\eqref{eq:coupled_divergence_operator_2} satisfies the uniform ellipticity condition (see~\eqref{eq:ellipticity}), we obtain
\[
-\frac{d}{dt} \int_\Omega \left\Vert \tilde{\psi}\right\Vert_1^2 dx + \int_\Omega\Vert\nabla\tilde{\psi}\Vert_2^2 dx \leq C \int_\Omega \left\Vert \tilde{\psi}\right\Vert_1^2 dx,
\]
from which we deduce
\begin{align}\label{eq:integrating_factor}
\frac{d}{dt} \left(e^{Ct} \int_\Omega \left\Vert \tilde{\psi}\right\Vert_1^2 dx\right) & = e^{Ct} \left( C\int_\Omega \left\Vert \tilde{\psi}\right\Vert_1^2 dx + \frac{d}{dt}\int_\Omega \left\Vert \tilde{\psi}\right\Vert_1^2 dx\right) \geq e^{Ct} \int_\Omega\Vert\nabla\tilde{\psi}\Vert_2^2 dx\geq 0,
\end{align}
for all $t>0$. We integrate~\eqref{eq:integrating_factor} over $\left[\frac{T}{4},t\right]$ to obtain
\begin{align}\label{eq:quarter_T_integral}
\int_{\Omega}\left\Vert \tilde{\psi}\right\Vert_1^2dx&\geq e^{C(T/4-t)}\int_{\Omega}\left\Vert \tilde{\psi}\left(T/4,\cdot\right)\right\Vert_1^2dx
\geq e^{-CT/2}\int_{\Omega}\left\Vert \tilde{\psi}\left(T/4,\cdot\right)\right\Vert_1dx,
\end{align}
for every $t\in\left[\frac{T}{4},\frac{3T}{4}\right]$. Integrating~\eqref{eq:quarter_T_integral} once more over $\left[\frac{T}{4},\frac{3T}{4}\right]$ now yields~\eqref{eq:quarter_T_lower_bound}. Finally, to show~\eqref{eq:time_zero_lower_bound}, we integrate~\eqref{eq:integrating_factor} over $t\in\left[0,\frac{T}{4}\right]$. 
% and conclude at once that
%\[
%\int_{\Omega}\left\Vert \tilde{\psi}(0,\cdot)\right\Vert_1^2dx\leq e^{CT/4}\int_{\Omega}\left\Vert \tilde{\psi}\left(T/4,\cdot\right)\right\Vert_1^2dx.
%\]
\end{proof}
\end{document}

\title{Controllability of coupled parabolic systems with multiple underactuations, part 1: algebraic solvability}
  \author{Drew Steeves\footnotemark[1] \and 
         Bahman Gharesifard\footnotemark[2] \and
         Abdol-Reza Mansouri\footnotemark[2]}
\renewcommand{\thefootnote}{\fnsymbol{footnote}}
  \footnotetext[1]{Department of Mechanical and Aerospace Engineering, University of California, San Diego, EBU1 2101, La Jolla, CA, United States 92093, \texttt{dsteeves@eng.ucsd.edu}}
   \footnotetext[2]{Department of Mathematics and Statistics, Queen's University, Jeffery Hall, University Ave., Kingston, ON, Canada K7L3N6 (\texttt{bahman.gharesifard@queensu.ca}, \texttt{mansouri@mast.queensu.ca})}

\maketitle

\begin{abstract}
This paper is the first of two parts which together study the null controllability of a system of coupled parabolic PDEs. This work specializes to an important subclass of these control problems which are coupled by first and zero-order couplings and are, additionally, underactuated. In this paper, we pose our control problem in a fairly new framework which divides the problem into interconnected components: we refer to the first component as the analytic control problem; we refer to the second component as the algebraic control problem, where we use an algebraic method to ``algebraically invert'' a linear partial differential operator that describes our system; this allows us to recover null controllability by means of internal controls which appear on only a few of the equations. Treatment of the analytic control problem is deferred to the second part of this work~\cite{steeves2017part2}. The conclusion of this two-part work is a null controllability result for the original problem.
\end{abstract}
 
\begin{keywords}
Controllability, Parabolic systems, Algebraic solvability, Fictitious control method.
\end{keywords}

\begin{AMS}
  35K40, % Second-order parabolic systems
  93B05 % Controllability
\end{AMS}

\section{Introduction}\label{sec:intro}
In recent years, problems concerning controllability of coupled parabolic PDEs have received much interest from the mathematical control community, see~\cite{ammar2011recent} and references therein. One classification of these numerous control problems is into problems with zero-order couplings (i.e., the reaction term in a usual parabolic PDE is now replaced with terms which couple the evolution of the solution with the solutions to other PDEs in the system) and problems with first-order couplings (i.e., the advection term is now replaced with terms which couple the evolution of the solution with the gradient of the solutions to other PDEs in the system). The applications of such control problems are ubiquitous: zero-order couplings arise in engineering problems modelled by reaction-diffusion equations, such as~\cite{bothe2003reaction,ennis2007coupling,shigesada1979spatial}, whereas first-order couplings arise in engineering problems modelled by reaction-advection-diffusion equations, such as~\cite{coron2014local,orlov2002discontinuous,pisano2016boundary,valli2005control}.

\subsection{Literature review}
For systems of several coupled parabolic equations, an important problem is to establish their controllability with reduced number of controls; we refer to such systems with reduced controls as underactuated systems of coupled parabolic PDEs. For the case of zero-order couplings and with internal controls, this control problem has been studied extensively in~\cite{ammar2009generalization,ammar2009kalman}. In~\cite{ammar2009kalman}, a necessary and sufficient condition for exact controllability is proved for a system of $m$ equations with constant coupling coefficients, which mimics the Kalman rank condition for finite-dimensional systems. In~\cite{ammar2009generalization}, some results similar to the Silverman-Meadows condition are obtained for time-varying coefficients.

General conditions for controllability of systems with first and zero-order couplings and internal controls have proven to be more elusive. In~\cite{guerrero2007null}, a system of $n+1$ coupled heat equations with constant couplings and with one underactuation is studied, and a sufficient condition for null controllability is given under some restrictions on the controls. In~\cite{benabdallah2014controllability}, a system of three parabolic equations coupled by (time and space) varying coefficients is studied for two underactuations.  The authors were able to recover a null controllability condition under some technical restrictions on the control domain and the coupling terms. In~\cite{duprez2016indirect}, a necessary and sufficient condition for null controllability is given for a system of $m$ equations with one underactuation and constant coupling coefficients; furthermore, the authors study the case of (time and space) varying coupling coefficients and prove a sufficient controllability condition for a system of two equations with one underactuation, under some technical conditions.

\subsection{Statement of contributions}
The first part of this work has one main contribution: it achieves in proving the so-called algebraic solvability of a system of coupled parabolic PDEs under a moderate rank condition, where controls appear on more than half of the equations, and additionally, is large enough (cf. Proposition~\ref{prop:algebraic_solvability} for details). The latter assumption is somewhat restrictive: for example, it limits the application of Proposition~\ref{prop:algebraic_solvability} for systems with two underactuations and in one dimension to systems with at least six equations. However, we address this shortcoming in Example~\ref{example:c_less_than_h}, where we demonstrate that the technique we've employed produces a moderate rank condition for smaller systems under which algebraic solvability is ensured.

Algebraic solvability of an underactuated system, which is referred to as the algebraic control problem, allows one to generate its solution locally, and this solution inherits zero as its initial and final conditions from the particular treatment that is employed. This result is a key component of the fictitious control method, which can be used to prove controllability results for underactuated coupled PDE systems and is employed in Section~\ref{sec:fictitious_control_method}. The technique used to prove our result is adapted from~\cite{coron2014local}.
\section{Preliminaries}\label{sec:prelims}
In this section, we introduce some notational conventions and present some mathematical background that we utilize throughout this work.

\subsection{Notation and conventions}
Throughout this work, we define $\mathbb{N}^*:=\mathbb{N}\setminus\{0\}$, and similarly, $\mathbb{R}^*:=\mathbb{R}\setminus\{0\}$. For $n,k\in\mathbb{N}^*$, we denote the set of $n\times k$ matrices with real-valued entries by $\mathcal{M}_{n\times k}(\mathbb{R})$, and we denote the set of $n\times n$ matrices with real-valued entries by $\mathcal{M}_{n}(\mathbb{R})$. We denote the set of linear maps from a vector space $U$ to a vector space $V$ by $\mathscr{L}(U;V)$. For $(X,\mathcal{T}_X)$ a topological space and $U\subset X$, we denote the closure of $U$ by $\bar{U}$.

\subsection{A system of interest}
In many fields of engineering, equations which describe the conservation of physical quantities are paramount. Among these conservation equations, the general second-order diffusion equation is routinely used to model engineering processes. Let $Q_T := (0,T)\times \Omega$ and $\Sigma_T := (0,T)\times \partial\Omega$ for some $T>0$; consider the second-order PDE
\begin{equation}\label{eq:general_parabolic}
\left\{
  \begin{aligned}
    \partial_t y + \mathcal{L}y&=r, &  &\text{in} \; Q_T, \cr
    y&=0, &&\text{on} \; \Sigma_T, \cr
    y(0,\cdot)&=y^0(\cdot), &&\text{in} \; \Omega,
  \end{aligned}\right.
\end{equation}
where $r:Q_T\rightarrow\mathbb{R}$ and $y^0:\Omega\rightarrow\mathbb{R}$ are known, $y: \bar{Q}_T\rightarrow \mathbb{R}$ is the unknown, and for each $t\in(0,T)$, $\mathcal{L}$ denotes the second-order linear differential operator given by
\begin{equation}\label{eq:normal_elliptic_operator}
\mathcal{L}y = -\sum_{i,j=1}^n \partial_{x_j}\left(d^{ij}(t,x)\partial_{x_i}y\right) + \sum_{i=1}^n g^i(t,x)\partial_{x_i} y + a(t,x)y,
\end{equation}
for given coefficients $d^{ij}, g^{i}, a$, for $i,j\in\{1,\dots,n\}$. Equation~\eqref{eq:general_parabolic} can be used to describe the evolution in time of the distribution of some quantity $y$ (e.g., heat), where the second-order term models diffusion, the first-order term models advection, the zero-order term models linear generation or depletion, and the forcing function accounts for external sources or sinks. We begin with some definitions that help us classify~\eqref{eq:general_parabolic}.
\begin{definition}
Let $\alpha=(\alpha_1,\dots,\alpha_n)$ be a multi-index and denote $\partial_{\alpha_1}\cdots\partial_{\alpha_n}y$ by $\partial_\alpha y$. For $k,\;l\in\mathbb{N}$ and $(d_\alpha)_{\alpha}$ coefficients, where $d_\alpha: Q_T \rightarrow \mathbb{R}$, a linear time-variant differential operator of order $l=2k$ on $\Omega$ given by
\[
\mathcal{L}y = \sum_{|\alpha|\leq l} d_{\alpha}(t,x)\partial_\alpha y 
\]
satisfies the uniform ellipticity condition if there exists $C>0$ such that,
\begin{equation}\label{eq:ellipticity}
\sum_{|\alpha|=l} d_{\alpha}(t,x)\xi^\alpha \geq C|\xi|^{l}, \qquad \forall \; \xi \in\mathbb{R}^n, \forall (t,x)\in Q_T,
\end{equation}
where $\xi^\alpha=\xi_1^{\alpha_1}\cdots \xi_n^{\alpha_n}$.
\end{definition}
\begin{definition}
A partial differential operator $\partial_t+\mathcal{L}$ is (uniformly) parabolic if $\mathcal{L}$ satisfies the uniform ellipticity condition.
\end{definition}

Of greater interest in many areas of engineering is the study a \emph{system of second-order parabolic PDEs} (e.g.,~\cite{pletcher2013computationalfluid},~\cite{vazquez2008control}). We express systems consisting of $m$ equations in vector form as
\begin{equation}\label{eq:system_parabolic_pde}
\left\{
  \begin{aligned}
    \partial_t y + \mathcal{L}y&=r, &  &\text{in} \; Q_T, \cr
    y&=0, &&\text{on} \; \Sigma_T, \cr
    y(0,\cdot)&=y^0(\cdot), &&\text{in} \; \Omega,
  \end{aligned}\right.
\end{equation}
where $y^0:=(y_1,\dots, y_m)$ and $r:=(r_1,\dots,r_m)$ are known, $y:=(y_1,\dots,y_m)$ are the unknowns, and the differential operator $\mathcal{L}$ is now defined as
\[
\mathcal{L}y = \sum_{k=1}^m\left(-\sum_{i,j=1}^n \partial_{x_j}\left(d_k^{ij}(t,x)\partial_{x_i}y_k\right) + \sum_{i=1}^n g_k^i(t,x)\partial_{x_i} y_k + a_k(t,x)y_k\right) \mathbf{e_k},
\]
where $\mathbf{e_k}$ is the $k$-th canonical basis vector in $\mathbb{R}^m$. Yet another very practical extension of this system of second-order PDEs is when the \emph{equations within the system are coupled} (e.g., \cite{ammar2011recent, hillen2009user, shigesada1979spatial}): denoting the $p$-th entry of $\mathcal{L}y$ as $\mathcal{L}_py$ for $p\in\{1,\dots,m\}$, we now have
\begin{equation}\label{eq:coupled_divergence_operator_1}
\mathcal{L}_py = \sum_{k=1}^m\left(-\sum_{i,j=1}^n \partial_{x_j}\left(d_{pk}^{ij}(t,x)\partial_{x_i}y_k\right) + \sum_{i=1}^n g_{pk}^i(t,x)\partial_{x_i} y_k + a_{pk}(t,x)y_k\right).
\end{equation}
When $p\neq k$, we call $d_{pk}^{ij}$ the \emph{second-order coupling coefficients}, $g_{pk}^i$ the \emph{first-order coupling coefficients}, and $a_{pk}$ the \emph{zero-order coupling coefficients}. This work studies a particular case of first and zero-order constant coupling coefficients, where for $\delta_{ij}$ denoting the Kronecker delta function, $d_{pk}^{ij}(t,x)=d_p^{ij}\delta_{pk}\in\mathbb{R}$, $g_{pk}^i(t,x)=-g_{pk}^i\in\mathbb{R}$ and $a_{pk}(t,x)=-a_{pk}\in \mathbb{R}$, for $i,j\in\{1,\dots,n\}$ and $p\in\{1,\dots,m\}$. Additionally, we study the case where $d_p^{ij}=d_p^{ji}$, for $i,j\in\{1,\dots,n\}$ and $p\in\{1,\dots,m\}$. Hence, we can write $\mathcal{L}y$ as
\begin{equation}\label{eq:coupled_divergence_operator_2}
\mathcal{L}y = \sum_{p=1}^m \left(-\text{div}(d_p\nabla y_p) -\sum_{k=1}^m g_{pk} \cdot \nabla y_k -\sum_{k=1}^m a_{pk} y_k\right)\mathbf{e_p},
\end{equation}
where $g_{pk}:=(g_{pk}^1,\dots,g_{pk}^n) \in \mathbb{R}^n$, $d_p\in\mathcal{M}_n(\mathbb{R})$ is symmetric and $\mathbf{e_p}$ is the  $p$-th canonical basis vector in $\mathbb{R}^m$, for $p\in\{1,\dots,m\}$. With these choices of coefficients, system~\eqref{eq:system_parabolic_pde} becomes
\begin{equation}\label{eq:coupled_system}
\left\{\begin{aligned}
    \partial_t y &=\text{div}(D\nabla y)+G\cdot \nabla y + Ay +r, &&\text{in} \; Q_T, \cr
    y&=0, &&\text{on} \; \Sigma_T, \cr
    y(0,\cdot)&=y^0(\cdot), &&\text{in} \; \Omega,
  \end{aligned}\right.
\end{equation}
where $D:=\text{diag}(d_1,\dots,d_m)$, $G:=(g_{pk})_{1\leq p,k\leq m} \in \mathcal{M}_m(\mathbb{R}^n)$ and $A:=(a_{pk})_{1\leq p,k\leq m} \in \mathcal{M}_m(\mathbb{R})$.
\subsection{The solution of coupled parabolic systems}\label{subsection:well_posedness}
To adapt classical existence and uniqueness results to a system of \emph{coupled} parabolic PDEs such as in system~\eqref{eq:coupled_system}, one can follow the treatment, for example, in~\cite[Section 7]{evans1988partial}, but write all intermediary results for a system of solutions. From now on, we assume that $\mathcal{L}$ satisfies~\eqref{eq:ellipticity}. Suppose $r\in L^2(Q_T)^m$, $y^0\in L^2(\Omega)^m$. For $u,v\in H_0^1(\Omega)^m$, we define the bilinear form
\[
B[u,v]:=\int_{\Omega} \sum_{p,k=1}^m \left(\sum_{i,j=1}^n d_p^{ij} (\partial_{x_i} u_p) (\partial_{x_j} v_p) - \sum_{i=1}^n g_{pk}^i (\partial_{x_i} u_k) v_p - a_{pk}u_kv_p\right)\mathbf{e_p} dx.
\]
One has the following definition.
\begin{definition}\label{def:weak_solution}
Suppose $r\in L^2(Q_T)^m$, $y^0\in L^2(\Omega)^m$. A function
\[
\mathbf{y} \in L^2((0,T);H_0^1(\Omega))^m\cap H^1((0,T);H^{-1}(\Omega))^m
\]
is said to be a weak solution of system~\eqref{eq:coupled_system} provided that for every $v\in H_0^1(\Omega)^m$ and almost every $t\in [0,T]$
\begin{enumerate}
\item $\langle\frac{d}{dt}\mathbf{y}, v\rangle + B[\mathbf{y},v]=\int_{\Omega} \mathbf{r}^T v dx$, and;
\item $\mathbf{y}(0)=y^0$,
\end{enumerate}
where $\langle\cdot,\cdot\rangle$ denotes the appropriate duality pairing.
%the second equality makes sense thanks to Theorem~\ref{thm:sobolev_regularity_time}.
\end{definition}

From now on, we mean by ``solution to a coupled parabolic system" the weak solution in the sense of Definition~\ref{def:weak_solution}.
\subsection{A parabolic regularity result}
We state a regularity result for the solution of system~\eqref{eq:coupled_system} which is essential in the work to follow.
%\begin{theorem}\cite[Theorem 5, Subsection 7.1.3]{evans1988partial}\label{thm:regularity_1}
%Assume $y^0\in H_0^1(\Omega)^m$ and $\mathbf{r} \in L^2(Q_T)^m$; suppose that $\mathbf{y} \in L^2((0,T);H_0^1(\Omega))^m\cap H^1((0,T);H^{-1}(\Omega))^m$ is the weak solution of system~\eqref{eq:coupled_system}. Then in fact
%\[
%\mathbf{y} \in L^2((0,T);H^2(\Omega))^m\cap L^\infty((0,T);H_0^1(\Omega))^m\cap H^1((0,T);L^2(\Omega))^m,
%\]
%and we have the estimate
%\begin{align*}
%&\left\Vert\mathbf{y}\right\Vert_{ L^2((0,T);H^2(\Omega))^m\cap L^\infty((0,T);H_0^1(\Omega))^m\cap H^1((0,T);L^2(\Omega))^m}\\
%&\leq C\left(||\mathbf{r}||_{L^2((0,T);L^2(\Omega))^m} + ||y^0||_{H_0^1(\Omega)^m} \right),
%\end{align*}
%where $C:=C(\Omega,T)$. If, in addition, $y^0 \in H^2(\Omega)^m$ and $\mathbf{r} \in H^1((0,T);L^2(\Omega))^m$, then
%\[
%\mathbf{y} \in L^\infty((0,T);H^2(\Omega))^m \cap H^2((0,T);H^{-1}(\Omega))^m
%\]
%and
%\[
%\frac{d}{dt} \mathbf{y} \in L^\infty((0,T);L^2(\Omega))^m \cap L^2((0,T); H_0^1(\Omega))^m,
%\]
%with the estimate
%\begin{align*}
%\esssup_{0\leq t \leq T} \left( ||\mathbf{y}||_{H^2(\Omega)^m} + \left\Vert\frac{d}{dt}\mathbf{y}\right\Vert_{L^2(\Omega)^m}\right) &+ ||\mathbf{y}||_{H^1((0,T); H_0^1(\Omega))^m \cap H^2((0,T);H^{-1}(\Omega))^m}\\
%& \leq C\left(||\mathbf{r}||_{L^2((0,T);L^2(\Omega))^m} + ||y^0||_{H^2(\Omega)^m} \right).
%\end{align*}
%\end{theorem}
%Under certain conditions, one can expect an even higher regularity for the weak solutions of system~\eqref{eq:coupled_system}. We have the following regularity result.
\begin{theorem}\cite[Theorem 6, Subsection 7.1.3]{evans1988partial}\label{thm:smoothing_theorem}
For $d\in \mathbb{N}$, assume $y^0 \in H^{2d+1}(\Omega)^m$, $\mathbf{r} \in L^2((0,T);H^{2d}(\Omega))^m\cap H^d((0,T);L^2(\Omega))^m$, and assume that $\mathbf{y} \in L^2((0,T);H_0^1(\Omega))^m\cap H^1((0,T);H^{-1}(\Omega))^m$ is the solution of system~\eqref{eq:coupled_system}. Suppose also that the following compatibility conditions hold:
\[
\begin{cases}
g^0 := y^0 \in H_0^1(\Omega)^m; \\
g^1 := \mathbf{r}(0)-\mathcal{L}g^0 \in H_0^1(\Omega)^m; \\
\vdots \\
g^d := \frac{d^{d-1}\mathbf{r}}{dt^{d-1}}(0)-\mathcal{L}g^{d-1} \in H_0^1(\Omega)^m.
\end{cases}
\]
Then $\mathbf{y} \in L^2((0,T);H^{2d+2}(\Omega))^m\cap H^{d+1}((0,T);L^2(\Omega))^m$ and we have the estimate
\begin{align}\label{eq:smoothing_estimate}
||\mathbf{y}||_{L^2((0,T);H^{2d+2}(\Omega))^m\cap H^{d+1}((0,T);L^2(\Omega))^m} &\leq C\left( ||\mathbf{r}||_{L^2((0,T);H^{2d}(\Omega))^m\cap H^d((0,T);L^2(\Omega))^m}\right.\nonumber \\
&\left.+ ||y^0||_{H^{2d+1}(\Omega)^m}\right).
\end{align}
\end{theorem}

\subsection{Some sparse matrix theory preliminaries}\label{sec:sparse_matrices}
When studying the invertibility of certain linear operators of interest, we are faced with studying non-singularity conditions for matrices associated to coupled parabolic PDEs of interest (cf. Subsection~\ref{sec:algebraic_solvability}). By nature of their construction, these matrices are \emph{sparse}. We describe an algorithm that can be used to decompose a sparse matrix into block triangular form. Importantly, this algorithm can be applied to matrices with symbolic entries as it only makes use of the placement of zero entries in the matrix.

Given a matrix $P\in\mathcal{M}_{q\times r}(\mathbb{R})$, consider the bipartite graph associated to $P$ given by the triple $G(P):=(R,C,E)$, where $R:=\{r_1,\dots,r_q\}$ is the set of row vertices associated to $P$, $C:=\{c_1,\dots,c_r\}$ is the set of column vertices associated to $P$, and $E$ denotes the set the edges $(r_i,c_j)$ associated to every nonzero entry $p_{ij}$ of $P$, for $i\in\{1,\dots,q\}$ and $j\in\{1,\dots,r\}$. As in~\cite{bondy1976graph}, we have the following definitions.
\begin{definition}
A matching $M\subset E$ in $G(P)$ is such that the edges in $M$ have no common vertices. We define the cardinality of $M$ as the number of edges in $M$. A maximum matching is a matching with maximum cardinality. Furthermore, $M$ is said to be column-perfect if every column vertex in $C$ is matched; it is said to be row-perfect if every row vertex in $R$ is matched; and it is said to be perfect if it is both column-perfect and row-perfect. A vertex $v_i$ is said to be matched with respect to $M$ if there exists $(v_i,v_j)\in M$ for appropriate indices $i,j$.
\end{definition}
\begin{definition}
The structural rank of a matrix $P\in\mathcal{M}_{q\times r}(\mathbb{R})$ is the cardinality of a maximum matching $M\subset E$ in $G(P)$.
\end{definition}
\begin{definition}
For an appropriate index $i$, let either $v_i=r_i$ or $v_i=c_i$. Fix a maximum matching $M$ in $G(P)$. For $k\in\mathbb{N}^*$, a walk is a sequence of (possibly repeated) vertices $(v_i)_{i=0}^k$ such that $(v_i,v_{i+1})$ is an edge for $i\in\{1,\dots,k-1\}$. An alternating walk is a walk with every second edge belonging to $M$. An alternating path is an alternating walk with no repeated vertices.
\end{definition}

Next, we define some important subsets of $R$ and $C$.
\begin{definition}
Let $M$ be a maximum matching in $G(P)$ with row set $R$ and column set $C$. We define the following sets of vertices with respect to $M$:
\begin{enumerate}
\item\resizebox{0.89\hsize}{!}{$VR:=\{\text{row vertices reachable by alternating paths from an unmatched row}\}$};
\item\resizebox{0.89\hsize}{!}{$HR:=\{\text{row vertices reachable by alternating paths from an unmatched col.}\}$};
\item\resizebox{0.89\hsize}{!}{$VC:=\{\text{col. vertices reachable by alternating paths from an unmatched row}\}$};
\item\resizebox{0.89\hsize}{!}{$HC:=\{\text{col. vertices reachable by alternating paths from an unmatched col.}\}$};
\item $SR:=R\setminus(VR\cup HR)$, and;
\item $SC:=C\setminus(VC\cup HC)$.
\end{enumerate}
\end{definition}
It was proven in~\cite{dulmage1958coverings} that $VR$, $HR$ and $SR$ are pairwise disjoint, and also that $VC$, $HC$ and $SC$ are pairwise disjoint. We demonstrate these definitions on an example.
\begin{example}
Consider the matrix $P\in\mathcal{M}_{4\times3}(\mathbb{R})$ and its bipartite graph $G(P)$ given by\\
\begin{minipage}{.5\textwidth}
\[P=
\left(\begin{array}{c c c}
a_{11}&a_{12}&0\\
0&0&a_{23}\\
0&a_{32}&0\\
a_{41}&0&a_{43}
\end{array}\right)
\]
\end{minipage}
\begin{minipage}{.4\textwidth}
\resizebox{1.7in}{1.7in}{
\begin{tikzpicture}[thick,
  every node/.style={draw,circle},
  fsnode/.style={fill=black},
  ssnode/.style={fill=black},
  every fit/.style={ellipse,draw,inner sep=-2pt,text width=2cm},
  <->,shorten >= 3pt,shorten <= 3pt]
\begin{scope}[start chain=going below,node distance=7mm]
\foreach \i in {1,2,3,4}
  \node[fsnode,on chain] (f\i){};
\end{scope}
\begin{scope}[xshift=4cm,yshift=-0.5cm,start chain=going below,node distance=7mm]
\foreach \i in {5,6,7}
  \node[ssnode,on chain] (s\i){};
\end{scope}
\node [black,fit=(f1) (f4),label=above:$R$] {};
\node [black,fit=(s5) (s7),label=above:$C$] {};
\draw (f1) -- (s5);
\draw (s6) -- (f1);
\draw (f2) -- (s7);
\draw (f3) -- (s6);
\draw (f4) -- (s5);
\draw (s7) -- (f4);
\end{tikzpicture}}
\end{minipage}\\\\
Consider two maximum matchings
\[
M_1:=\{(r_1,c_1),(r_2,c_3),(r_3,c_2)\}\quad \text{and} \quad M_2:=\{(r_1,c_1),(r_3,c_2),(r_4,c_3)\}
\]
in $G(P)$.
Note that $M_1$ and $M_2$ are column-perfect and the structural rank of $A$ is 3.
%Note that the structural rank is independent of every nonzero entry of $P$ (i.e., the structural rank is independent of the weights of the edges in $E$). 
For matching $M_1$, an alternating path is given by the sequence $r_4, c_1, r_1, c_2, r_3$. Furthermore, for matching $M_1$, we have $VR:=\{r_1,r_2,r_3,r_4\}$ and $VC:=\{c_1,c_2,c_3\}$.
\end{example}

In the above example, the structural rank of $P$ is equal to the rank of $P$ in general. It is easily deduced that the structural rank of a matrix in $\mathcal{M}_{q,r}(\mathbb{R})$ is an upper bound on the rank of that matrix, and is never greater than $\min\{q,r\}$. We arrive at the following important result, which is identified in literature as the \emph{Dulmage-Mendelsohn decomposition} and can be deduced from~\cite{dulmage1958coverings,pothen1990computing}.
\begin{theorem}\label{thm:dmperm_decomposition}
Let $P\in\mathcal{M}_{q\times r}(\mathbb{R})$, and let $M$ be a maximum matching in $G(P)$. Then, one can permute the rows and columns of $P$ to obtain the following block-triangular form (which we refer to as coarse decomposition):
\[
\left(\begin{array}{c c c c}
P_{11} & P_{12} & P_{13} & P_{14}\\
0 & 0 & P_{23} & P_{24}\\
0 & 0 & 0 & P_{34}\\
0 & 0 & 0 & P_{44}
\end{array}\right),
\]
where
\begin{enumerate}
\item $(P_{11}, P_{12})$ is the underdetermined part of the matrix (i.e., more rows than columns), is generated by $(r_i,c_i)\in HR\times HC$, and has row-perfect matching;
\item $\Bigg(\begin{array}{c}P_{34}\\P_{44}\end{array}\Bigg)$ is the overdetermined part of the matrix (i.e., more columns than rows), is generated by $(r_i,c_i)\in VR\times VC$, and has column-perfect matching;
\item $P_{23}$ is generated by $(r_i,c_i)\in SR\times SC$, and;
\item $P_{12}, P_{23}, P_{34}$ are square matrices with nonzero diagonal, and hence have perfect matchings (i.e., they are of maximal structural rank).
\end{enumerate}
Moreover, $P_{12}, P_{23}, P_{34}$ can be further decomposed into block-triangular form with nonzero diagonal (which we refer to as fine decomposition). The structural rank of $P$ is given by the sum of the structural ranks of $P_{12}, P_{23}, P_{34}$.
\end{theorem}
\begin{remark}
If $P$ is overdetermined, then $(P_{11}, P_{12})$ will be present only if $P$ does not have a column-perfect matching. Similarly, if $P$ is underdetermined, then $(P_{34}, P_{44})$ will appear only if $P$ does not have a row-perfect matching. In both of these cases, the presence of $P_{23}$ depends on the nonzero structure of $P$. If $P$ is square, non-symmetric and has a perfect maximum matching, then its coarse decomposition will consist only of $P_{23}$.
\end{remark}
\begin{remark}
It was proven in~\cite{dulmage1958coverings} that the Dulmage-Mendelsohn decomposition is independent of the choice of maximum matching in $G(P)$.
\end{remark}

We are now ready to study system~\eqref{eq:coupled_system} under the framework of control systems, in the sense that we ``select" the forcing term $r$ to drive the system to a desired final state in some time $T\in\mathbb{R}^*$.

\section{Problem statement}\label{sec:problem}
%For $Q_T:=(0,T)\times\Omega$ and $\Sigma_T:=(0,T)\times\partial\Omega$,
We revisit the system consisting of $m$ coupled second-order parabolic PDEs given by system~\eqref{eq:coupled_system}, where it can be deduced, for example, from~\cite[Theorems 3 and 4, Section 7.1.2]{evans1988partial}, that for any initial condition $y^0\in L^2(\Omega)^m$ and $r\in L^2(Q_T)^m$, system~\eqref{eq:coupled_system} admits a unique solution
\[
y\in L^2((0,T);H_0^1(\Omega))^m\cap H^1((0,T);H^{-1}(\Omega))^m.
\]
We now introduce the problem of interest.

\subsection{The control problem}
We recast system~\eqref{eq:coupled_system} as a \emph{control system}, where $r=Bu$ with $u\in L^2(Q_T)^c$ being control inputs to be chosen, and $B\in\mathcal{M}_{m\times c}(\mathbb{R})$, with $0< c\leq m$, yielding
\begin{equation}\label{eq:coupled_control_system}
\left\{\begin{aligned}
    \partial_t y &=\text{div}(D\nabla y)+G\cdot \nabla y + Ay +Bu, &&\text{in} \; Q_T, \cr
    y&=0, &&\text{on} \; \Sigma_T, \cr
    y(0,\cdot)&=y^0(\cdot), &&\text{in} \; \Omega.
  \end{aligned}\right.
\end{equation}
%We associate to the control system~\eqref{eq:coupled_control_system} the operator
%\begin{align}\label{eq:elliptic_operator}
%\mathcal{L}(y):=-\left(\text{div}(D\nabla)+G\cdot\nabla+A\right) (y).
%\end{align}

Let us now introduce our objectives that we aim to achieve by selecting appropriate control inputs. We have the following notions of \emph{controllability} for system~\eqref{eq:coupled_control_system}.
\begin{definition}
We say that system~\eqref{eq:coupled_control_system} is null controllable in time $T$ if for every initial condition $y^0\in L^2(\Omega)^m$, there exists a control $u\in L^2(Q_T)^{c}$ such that the solution $y \in L^2((0,T);H_0^1(\Omega))^m\cap H^1((0,T);H^{-1}(\Omega))^m$ to~\eqref{eq:coupled_control_system} satisfies
\[
y(T)=0 \quad \text{in} \quad \Omega.
\]
\end{definition}

\begin{definition}
We say that system~\eqref{eq:coupled_control_system} is approximately controllable in time $T$ if for every $\epsilon>0$, for every initial condition $y^0\in L^2(\Omega)^m$ and for every $y_T\in L^2(\Omega)^m$, there exists a control $u\in L^2(Q_T)^{c}$ such that the solution $y \in L^2((0,T);H_0^1(\Omega))^m\cap H^1((0,T);H^{-1}(\Omega))^m$ to~\eqref{eq:coupled_control_system} satisfies
\[
||y(T)-y_T||_{L^2(\Omega)^m}^2\leq \epsilon.
\]
\end{definition}
%In the theory of control of PDEs, one encounters two prominent types of control: internal (or distributed) control; and boundary control. As the names suggest, distributed control refers to the case where supp$(u)\subset\Omega$ which is nonempty, whereas boundary control refers to the case where supp$(u)\subset\partial\Omega$, and hence the controls are defined via the boundary conditions of the PDE.

This work specializes to the case of internal (or distributed) control: that is, for $\omega\subset \Omega$ nonempty and open, we study the case where $r=\mathbbm{1}_{\omega}Bu$, and henceforth, we denote by $q_T$ the set $(0,T)\times\omega$. 

An interesting control problem that arises in many engineering applications is underactuation, that is, when $c<m$. Our work will further specialize to this case, where there are currently few results for \emph{first and zero-order couplings}, for arbitrary $m$ and $c<m-1$ (even for the case of constant coefficients).

Since we treat the particular case of a system of linear parabolic PDEs with constant coefficients (constant in space and \emph{time}), we are easily able to ascertain approximate controllability of system~\eqref{eq:coupled_control_system} from its null controllability.
\begin{theorem}\cite[Theorem 2.45]{coron2007control}
Assume that for every $T>0$, the control system~\eqref{eq:coupled_control_system} is null controllable in time $T$. Then, for every $T>0$, system~\eqref{eq:coupled_control_system} is approximately controllable in time $T$.
\end{theorem}

We now outline the treatment that we use throughout this work.
\section{Fictitious control method}\label{sec:fictitious_control_method}
This section presents a technique that can be used to prove the null controllability of the coupled system~\eqref{eq:coupled_control_system} with possibly multiple underactuations (i.e., when $c\leq m-1$). We first introduce the so-called \emph{fictitious control method}, developed in~\cite{coron2014local}, which allows one to bifurcate the null controllability problem into interconnected components: an analytic control problem, where \emph{fictitious} controls act on every equation in the coupled system~\eqref{eq:coupled_control_system}; and an algebraic control problem, where there are possibly many underactuations. For the analytic problem, one can prove a so-called \emph{weighted observability inequality} which will help deduce null controllability of the analytic system. For the algebraic problem, one can pose this underactuated control problem as an underdetermined system involving differential operators, and, under some conditions, ``invert" one of these operator algebraically. The first part of this two-part work focuses on the latter treatment.

\subsection{Definitions}\label{sec:algebraic_solvability_motivation}
Recall that we denote our control domain by $q_T:=(0,T)\times\omega$. We begin with some definitions.
\begin{definition}
For $n\in\mathbb{N}^*$, let $\alpha$ be a multi-index of length $n+1$. For $k,l\in\mathbb{N^*}$, a linear map $\mathcal{B}:C^\infty(q_T)^k\rightarrow C^\infty(q_T)^l$ is called a linear partial differential operator of order $m\in\mathbb{N}$ in $q_T$ if for every $\alpha$ verifying $|\alpha| \leq m$, there exists $A_\alpha \in C^\infty(q_T;\mathscr{L}(\mathbb{R}^k;\mathbb{R}^l))$ such that for all $\phi\in C^\infty(q_T)^k$ and $(t,x)\in q_T$,
\begin{equation}\label{eq:linear_partial_differential_operator}
(\mathcal{B}\phi)(t,x) = \sum_{|\alpha|\leq m} A_\alpha(t,x) \partial_\alpha \phi(t,x).
\end{equation}
%If $n=0$ and the above is satisfied, then $\mathcal{B}$ is called a linear ordinary differential operator. Note that for a linear ordinary differential operator $\mathcal{B}$, one may not be able to associate it to a linear ODE: indeed, for such an association, one may have to additionally impose an invertibility condition on one of the $A_\alpha$'s.
\end{definition}
Let $c, m, k\in\mathbb{N}$ and consider the linear partial differential operators
\[
\begin{cases}
\mathcal{L}:C^\infty(q_T)^{m+c} \rightarrow C^\infty(q_T)^{m},\\
\mathcal{N}: C^\infty(q_T)^k \rightarrow C^\infty(q_T)^{m}.
\end{cases}
\]
Suppose that for $(\hat{y} \quad \hat{u})^T \in C^\infty(q_T)^{m+c}$ and $\tilde{u} \in C^\infty(q_T)^k$, the linear equation
\begin{equation}\label{eq:linear_operator_equation}
\mathcal{L}\left((\hat{y} \quad \hat{u})^T\right)=\mathcal{N}(\tilde{u})
\end{equation}
is of interest, where $\tilde{u}$ is given and $(\hat{y} \quad \hat{u})^T$ are the unknowns. We characterize the solvability of~\eqref{eq:linear_operator_equation}.
\begin{definition}\label{def:algebraically_solvable}
We say that the linear equation~\eqref{eq:linear_operator_equation} is algebraically solvable in $q_T$ if there exists a linear partial differential operator $\mathcal{B}:C^\infty(q_T)^k\rightarrow C^\infty(q_T)^{m+c}$ such that
\begin{equation}\label{eq:algebraically_solvable}
\mathcal{L}\circ\mathcal{B}=\mathcal{N},
\end{equation}
that is, $\mathcal{B}(\tilde{u})$ is a solution to~\eqref{eq:linear_operator_equation} for every $\tilde{u}\in C^\infty(q_T)^k$. If $k=m$ and $\mathcal{N}=\text{Id}_{C^\infty(q_T)^m}$, then we call $\mathcal{B}$ the right inverse of $\mathcal{L}$.
%For $\mathcal{L}$ and $\mathcal{N}$ linear ordinary differential operators, if there exists a linear ordinary differential operator $\mathcal{B}$ satisfying~\eqref{eq:algebraically_solvable}, then we say that~\eqref{eq:linear_operator_equation} is algebraically solvable.
\end{definition}

In other words, we wish to find $\mathcal{B}$ such that the following diagram is commutative:
\[
\xymatrix
   {&C^\infty(q_T)^{m+c}\ar[r]^{\mathcal{L}}&C^\infty(q_T)^{m}\\
     &C^\infty(q_T)^{k}\ar[ur]_{\mathcal{N}}\ar@{-->}[u]_{\mathcal{B}}&
   }.
\]
\subsection{The fictitious control method}
Our goal is to prove null controllability in time $T$ for the control system~\eqref{eq:coupled_control_system}, where there are $m$ coupled parabolic equations and less than $m$ controls. To accomplish this for an arbitrary number of controls $c\leq m-1$, our strategy is to divide this control problem into two separate parts as was done in~\cite{coron2014local,duprez2016indirect}.

\subsubsection{Analytic control problem}
We first consider following control problem: for any $\tilde{y}^0\in L^2(\Omega)^m$, prove the existence of $(\tilde{y},\tilde{u})$ a solution of
\begin{equation}\label{eq:analytic_problem}
\left\{\begin{aligned}
    \partial_t \tilde{y} &=\text{div}(D\nabla \tilde{y})+G\cdot \nabla \tilde{y} + A\tilde{y} +\mathcal{N}\left(\mathbbm{1}_{\omega} \tilde{u}\right), &&\text{in} \; Q_T, \cr
    \tilde{y}&=0, &&\text{on} \; \Sigma_T, \cr
    \tilde{y}(0,\cdot)&=\tilde{y}^0(\cdot), &&\text{in} \; \Omega,
  \end{aligned}\right.
\end{equation}
such that $\tilde{y}(T,\cdot)=0$, where $\mathcal{N}$ is a differential operator that is to be determined (cf. Section~\ref{sec:algebraic_solvability}), $\tilde{u}$ \emph{acts on all equations} in~\eqref{eq:analytic_problem}, and we denote by $\mathbbm{1}_{\omega}$ a \emph{smooth version} of the indicator function which will be constructed in~\cite{steeves2017part2}. Note that $(\tilde{y},\tilde{u})$ has to be in a suitable space: in particular, depending on our choice of differential operator $\mathcal{N}$, $\tilde{u}$ has to be regular enough to withstand the derivatives applied by $\mathcal{N}$.
%While these restrictions make solving control system~\eqref{eq:analytic_problem} slightly non-classical (we need to use a \emph{weighted} observability inequality), there is hope to finding such a solution since controls appear in every equation in~\eqref{eq:analytic_problem}.
%We elaborate on our technique for constructing very regular controls in part two of this work.

\subsubsection{Algebraic control problem}
We next consider a different control problem: prove the existence of a solution $(\hat{y},\hat{u})$ of
\begin{equation}\label{eq:algebraic_problem}
\left\{\begin{aligned}
    \partial_t \hat{y} &=\text{div}(D\nabla \hat{y})+G\cdot \nabla \hat{y} + A\hat{y} +B\hat{u} + \mathcal{N}\left(\mathbbm{1}_{\omega} \tilde{u}\right), &&\text{in} \; Q_T, \cr
    \hat{y}&=0, &&\text{on} \; \Sigma_T, \cr
    \hat{y}(0,\cdot)&=\hat{y}(T,\cdot)=0, &&\text{in} \; \Omega,
  \end{aligned}\right.
\end{equation}
\emph{where} $\hat{u}$ \emph{acts only on the first} $c$ \emph{equations} and $B= \left(\text{Id}_{c}\quad 0_{c\times (m-c)}\right)^T\in\mathcal{M}_{m\times c}(\mathbb{R})$.  The notions of algebraic solvability, as described in Section~\ref{sec:algebraic_solvability_motivation}, will be used to resolve this control problem in the next subsection. The analytic and algebraic control problems differ in the following ways: in the analytic problem, the controls are $\mathcal{N}(\mathbbm{1}_\omega\tilde{u})$, whereas in the algebraic problem, the controls are $\hat{u}$, and furthermore, $\mathcal{N}(\mathbbm{1}_\omega\tilde{u})$ appears but is considered to be a source term; and in the analytic problem, one has to prove that $\tilde{y}(T,\cdot)=0$ (we will accomplish this in~\cite{steeves2017part2} by means of an observability inequality), whereas in the algebraic problem, $\hat{y}(T,\cdot)=0$ is inherited from the construction of the solution $(\hat{y},\hat{u})$ (cf. Remark~\ref{remark:regular_enough_controls}).

Solving both the analytic and algebraic problems will prove the null controllability of system~\eqref{eq:coupled_control_system}. Indeed, defining
\[
(y,u):= (\tilde{y}-\hat{y},-\hat{u}),
\]
one notices that $(y,u)$ is the solution to~\eqref{eq:coupled_control_system} in a suitable space with $y(T,\cdot)=0$. Note that the controls in the analytic system, $\mathcal{N}(\mathbbm{1}_\omega\tilde{u})$, are eliminated via the subtraction $\tilde{y}-\hat{y}$; this gives meaning to the name of the method we employ.
\subsection{Algebraic solvability}\label{sec:algebraic_solvability}
In this section, we study the algebraic solvability of differential operators corresponding system~\eqref{eq:algebraic_problem} which contains $m$ equations and $c$ controls, for $c\in\{1,\dots,m-1\}$. To this end, we consider the linear partial differential operator defined by
\begin{equation}\label{eq:underdetermined_operator}
\mathcal{L}\left((\hat{y}\quad\hat{u})^T\right) := \partial_t\hat{y} - \text{div}(D\nabla \hat{y}) - G\cdot \nabla \hat{y} - A\hat{y} - B\hat{u},
\end{equation}
which is an underdetermined operator, and we consider $\mathcal{N}\left(\mathbbm{1}_{\omega} \tilde{u}\right)$ as a source term, where $\mathcal{N}$ is to be chosen later. One can write system~\eqref{eq:algebraic_problem} as
\begin{equation}\label{eq:algebraic_underdetermined_system}
\mathcal{L}\left((\hat{y}\quad\hat{u})^T\right) = \mathcal{N}\left(\mathbbm{1}_{\omega} \tilde{u}\right);
\end{equation}
we study the algebraic solvability of~\eqref{eq:algebraic_underdetermined_system} in $q_T$. Recall from Definition~\ref{def:algebraically_solvable} that this is equivalent to proving the existence of a linear partial differential operator $\mathcal{B}:C^\infty(q_T)^k\rightarrow C^\infty(q_T)^m$ such that $(\hat{y}\quad\hat{u})^T = \mathcal{B}\left(\mathbbm{1}_{\omega} \tilde{u}\right)$ for any $\mathbbm{1}_{\omega}\tilde{u}\in C^\infty(q_T)^m$, and hence by reason of $\mathcal{B}	$ being a differential operator, $(\hat{y},\hat{u})$ will have support in $q_T$. With a slight abuse of notation, from now on we denote the extension by zero of $(\hat{y},\hat{u})$ to $Q_T$ also by $(\hat{y},\hat{u})$, so that $\hat{y}=0$ on $\Sigma_T$ and $\hat{y}(0,\cdot)=\hat{y}(T,\cdot)=0 \; \text{in} \; \Omega$.
\begin{remark}\label{remark:regular_enough_controls}
For simplicity, we formulated the notion of algebraic solvability for controls in the analytic problem $\mathbbm{1}_{\omega} \tilde{u} \in C^\infty(q_T)$, which dictates the regularity of $\left(\hat{y},\hat{u}\right)$; however, in~\cite{steeves2017part2} we will need to expand the space of controls that we may access to recover null controllability results for system~\eqref{eq:analytic_problem}. For controls with weaker regularity, we must additionally show that these controls vanish at times $t=0$ and $t=T$. This treatment is deferred to~\cite{steeves2017part2}. For the time being, assume $(\hat{y},\hat{u})$ are regular enough such that $\mathcal{L}\left((\hat{y}\quad\hat{u})^T\right)$ is well-defined.\\
\end{remark}
%As we will see, for our choice of $k$ it is easier to solve the adjoint equation of~\eqref{eq:algebraic_underdetermined_system}. To this end,

We study the adjoint system associated to system~\eqref{eq:algebraic_problem}:
\begin{equation}\label{eq:algebraic_problem_adjoint}
\left\{\begin{aligned}
    -\partial_t \hat{\psi} &=\text{div}(D\nabla\hat{\psi})-G^*\cdot\nabla\hat{\psi}+A^*\hat{\psi}, &&\text{in} \; Q_T, \cr
    \hat{\psi}&=0, &&\text{on} \; \Sigma_T, \cr
    \hat{\psi}(T,\cdot)&=\hat{\psi}^0(\cdot), &&\text{in} \; \Omega,
  \end{aligned}\right.
\end{equation}
for $\hat{\psi}^0\in L^2(\Omega)^m$.

\subsubsection{One underactuation}\label{sub:one_underactuation}
This section follows the treatment in~\cite[Subsection 2.1]{duprez2016indirect} and is presented here to contrast the existing technique to treat the null controllability of system~\eqref{eq:algebraic_problem} with one underactuation and the proposed technique in Subsection~\ref{sub:many_underactuations}, which treats the case of multiple underactuations. The method presented here succeeds in algebraically solving~\eqref{eq:algebraic_underdetermined_system} by utilizing the first and zero-order couplings to isolate for the unknown, and is henceforth referred to as the \emph{direct isolation technique}.

Choose $k=m$; we wish to find a linear partial differential operator $\mathcal{B}$ such that
\begin{equation}\label{eq:algebraically_solvable_2}
\mathcal{L}\circ \mathcal{B} = \mathcal{N},
\end{equation}
where $\mathcal{L}$ is given in~\eqref{eq:underdetermined_operator} and $\mathcal{N}$ is to be chosen. Note that this is equivalent to solving the adjoint problem: that is, finding a linear partial differential operator $\mathcal{B}^*$ such that 
\begin{equation}\label{eq:algebraically_solvable_adjoint}
\mathcal{B}^*\circ \mathcal{L}^* = \mathcal{N}^*.
\end{equation}
We calculate the (formal) adjoint of differential operator $\mathcal{L}$: for $\hat{\psi} \in L^2(Q_T)^m$, we have
\begin{align*}
&\left(\mathcal{L}\left((\hat{y}\quad\hat{u})^T\right),\hat{\psi}\right)\\
& = \left(\iint_{Q_T} \sum_{k=1}^m \left(\partial_t \hat{y}\hat{\psi}_k - \text{div}(d_k \nabla \hat{y}_k)\hat{\psi}_k - \sum_{i=1}^m (g_{ki} \cdot\nabla\hat{y}_k + a_{ki}\hat{y}_k) \hat{\psi}_k \right)\right.\\
&\left.+ \sum_{l=1}^c \hat{u}_l \hat{\psi}_l dx dt \right); 
\end{align*}
equating this to
\begin{align*}
&  \iint_{Q_T}\sum_{k=1}^m \hat{y}_k \mathcal{L}^*_k\hat{\psi} + \sum_{l=1}^c \hat{u}_l \mathcal{L}^*_{m+l}\hat{\psi} \\
& = \left((\hat{y}\quad\hat{u})^T,\mathcal{L}^*\hat{\psi}\right),
\end{align*}
yields
\begin{equation}\label{eq:L_adjoint_calculation}
\mathcal{L}^*\hat{\psi}=
\left(\begin{array}{c}
-\left(\partial_t+\text{div}(d_1\nabla)\right)\hat{\psi}_1 +\sum_{j=1}^m \left(g_{j1}\cdot \nabla -a_{j1}\right)\hat{\psi}_j\\
-\left(\partial_t+\text{div}(d_2\nabla)\right)\hat{\psi}_2 +\sum_{j=1}^m \left(g_{j2}\cdot \nabla -a_{j2}\right)\hat{\psi}_j\\
\vdots \\
-\left(\partial_t+\text{div}(d_m\nabla)\right)\hat{\psi}_m +\sum_{j=1}^m \left(g_{jm}\cdot \nabla -a_{jm}\right)\hat{\psi}_j\\
\hat{\psi}_1\\
\vdots\\
\hat{\psi}_{m-1}
\end{array}\right).
\end{equation}
We state the following lemma, which is a reformulation of~\cite[Theorem 1]{duprez2016indirect}.
\begin{lemma}\label{thm:suff_necc_one_underactuation}
The linear partial differential equation~\eqref{eq:algebraically_solvable_adjoint} is algebraically solvable in $q_T$ if there exists an index $i_0 \in \{1,\dots,m-1\}$ such that
\begin{equation}\label{eq:suff_necc_coupling}
g_{mi_0} \neq 0 \quad \text{or} \quad a_{mi_0}\neq 0.
\end{equation}
\end{lemma}
\begin{proof}
One need only look at the $i_0$-th entry of $\mathcal{L}^*$ to verify this assertion: 
\begin{align*}
\mathcal{L}^*_{i_0}\hat{\psi} &= -\left(\partial_t+\text{div}(d_{i_0}\nabla)\right)\hat{\psi}_{i_0} +\sum_{j=1}^m \left(g_{ji_0}\cdot \nabla -a_{ji_0}\right)\hat{\psi}_j\\
& = -\left(\partial_t+\text{div}(d_{i_0}\nabla)\right) \mathcal{L}^*_{m+i_0}\hat{\psi} + \sum_{j=1}^{m-1} \left(g_{ji_0}\cdot \nabla -a_{ji_0}\right)\mathcal{L}^*_{m+j}\hat{\psi}\\
& + \left(g_{mi_0} \cdot \nabla -a_{mi_0}\right)\hat{\psi}_m,
\end{align*}
which one can use to isolate for the unknown $\hat{\psi}_m$ and its spatial derivative:
\begin{equation}\label{eq:one_underactuation_coupling_equality}
\scalemath{0.97}{\left(g_{mi_0} \cdot \nabla -a_{mi_0}\right)\hat{\psi}_m = \mathcal{L}^*_{i_0}\hat{\psi} + \left(\partial_t+\text{div}(d_{i_0}\nabla)\right) \mathcal{L}^*_{m+i_0}\hat{\psi} - \sum_{j=1}^{m-1} \left(g_{ji_0}\cdot \nabla -a_{ji_0}\right)\mathcal{L}^*_{m+j}\hat{\psi}.}
\end{equation}
Hence, a careful choice of $\mathcal{N}^*$ yields the desired result: choosing
\[
\mathcal{N}^*\hat{\psi} :=
\left(\begin{array}{c}
\hat{\psi}_1\\
\hat{\psi}_2\\
\vdots\\
\hat{\psi}_{m-1}\\
\left(g_{mi_0} \cdot \nabla -a_{mi_0}\right)\hat{\psi}_m
\end{array}\right),
\]
one can define for $\phi \in C^\infty(Q_T)^{2m-1}$
\[
\mathcal{B}^*\phi :=
\left(\begin{array}{c}
\phi_{m+1}\\
\phi_{m+2}\\
\vdots\\
\phi_{2m-1}\\
\phi_{i_0} + \left(\partial_t+\text{div}(d_{i_0}\nabla)\right) \phi_{m+i_0} - \sum_{j=1}^{m-1} \left(g_{ji_0}\cdot \nabla -a_{ji_0}\right)\phi_{m+j}
\end{array}\right),
\]
so that
\[
\left(\mathcal{B}^*\circ \mathcal{L}^*\right)\hat{\psi} = \mathcal{N}^*\hat{\psi}
\]
is verified for every $\hat{\psi}\in C^\infty(Q_T)^m$.
\end{proof}
\begin{remark}
One notices that condition~\eqref{eq:suff_necc_coupling} is also necessary for algebraic solvability of~\eqref{eq:algebraically_solvable_adjoint}.
\end{remark}
\subsubsection{Multiple underactuations}\label{sub:many_underactuations}
We specialize to the case where system \eqref{eq:algebraic_problem} has more than one underactuation (i.e., when $c<m-1$).

\emph{Direct isolation technique:} we begin by employing the technique presented in Subsection~\ref{sub:one_underactuation}.
%to reveal the obstructions that limit our ability to deduce algebraic solvability of~\eqref{eq:algebraic_problem}.
For the moment, we focus on the simplest case, when $c=m-2$. We have
\[
\mathcal{L}^*\hat{\psi}=
\left(\begin{array}{c}
-\left(\partial_t+\text{div}(d_1\nabla)\right)\hat{\psi}_1 +\sum_{j=1}^m \left(g_{j1}\cdot \nabla -a_{j1}\right)\hat{\psi}_j\\
-\left(\partial_t+\text{div}(d_2\nabla)\right)\hat{\psi}_2 +\sum_{j=1}^m \left(g_{j2}\cdot \nabla -a_{j2}\right)\hat{\psi}_j\\
\vdots \\
-\left(\partial_t+\text{div}(d_m\nabla)\right)\hat{\psi}_m +\sum_{j=1}^m \left(g_{jm}\cdot \nabla -a_{jm}\right)\hat{\psi}_j\\
\hat{\psi}_1\\
\vdots\\
\hat{\psi}_{m-2}
\end{array}\right).
\]
A natural necessary condition for algebraic solvability of~\eqref{eq:algebraic_problem} as in Lemma~\eqref{thm:suff_necc_one_underactuation} is the following: without loss of generality, suppose there exists indices $i_0, i_1\in \{1,\dots,m-2\}$ such that
\[
\left\{\begin{aligned}
&g_{(m-1)i_0}\neq 0 \quad &&\text{or} \quad a_{(m-1)i_0}\neq 0,\\
&g_{mi_1}\neq 0 \quad &&\text{or} \quad a_{mi_1}\neq 0.
\end{aligned}\right.
\]
One immediately encounters the issue that none of the entries of $\mathcal{L}^*$ can be used to isolate for the individual unknowns $\hat{\psi}_{m-1}$ and $\hat{\psi}_m$ (and their spatial derivatives). Instead, we recover the system of equations
\begin{equation}\label{eq:system_equations_m-2}
\left\{\begin{aligned}
\left(g_{(m-1)i_0} \cdot \nabla -a_{(m-1)i_0}\right)\hat{\psi}_{m-1}&+ \left(g_{mi_0} \cdot \nabla -a_{mi_0}\right)\hat{\psi}_{m}\\
&= \mathcal{L}^*_{i_0}\hat{\psi} + \left(\partial_t+\text{div}(d_{i_0}\nabla)\right) \mathcal{L}^*_{m+i_0}\hat{\psi} \\
&- \sum_{j=1}^{m-2} \left(g_{ji_0}\cdot \nabla -a_{ji_0}\right)\mathcal{L}^*_{m+j}\hat{\psi},\\
\left(g_{(m-1)i_1} \cdot \nabla -a_{(m-1)i_1}\right)\hat{\psi}_{m-1}&+ \left(g_{mi_1} \cdot \nabla -a_{mi_1}\right)\hat{\psi}_{m}\\
&= \mathcal{L}^*_{i_1}\hat{\psi} + \left(\partial_t+\text{div}(d_{i_1}\nabla)\right) \mathcal{L}^*_{m+i_1}\hat{\psi}\\
&- \sum_{j=1}^{m-2} \left(g_{ji_1}\cdot \nabla -a_{ji_1}\right)\mathcal{L}^*_{m+j}\hat{\psi}.
\end{aligned}\right.
\end{equation}

While one can define an appropriate $\mathcal{N}^*$ using~\eqref{eq:system_equations_m-2} such that~\eqref{eq:algebraic_problem} is algebraically solvable,  in general this $\mathcal{N}^*$ will have entries involving both $\hat{\psi}_{m-1}$ and $\hat{\psi}_m$ (and their spatial derivatives). Such an $\mathcal{N}^*$ introduces an unresolvable issue when attempting to solve the analytic control problem. Alas, we are not aware of a procedure through which one can hope to recover a general necessary and sufficient condition for algebraic solvability of~\eqref{eq:algebraic_problem} using this technique.

\emph{Prolongation technique:} inspired by~\cite[Section 3]{coron2014local}, we present a method to prove the algebraic solvability of~\eqref{eq:algebraically_solvable_adjoint} by means of \emph{prolongation}: that is, since $\mathcal{L}^*\hat{\psi} = \mathcal{N}^*\hat{\psi}$ is an overdetermined system (i.e., there are $m+c$ equations and only $m$ unknowns), we can expect to differentiate each equation a sufficient amount of times with respect to all of the spatial variables in order to gain more equations than ``algebraic unknowns", which we make more precise in what follows. An inversion technique, which is inspired by the results in~\cite[Section 2.3.8]{gromov2013partial}, is then used to recover the unknowns from the overdetermined system.

We consider system~\eqref{eq:algebraic_problem} for an arbitrary $c\in\{1, \dots, m-2\}$ and define the linear partial differential operator
\[
\mathcal{N}\zeta :=
\left(\begin{array}{c}
\zeta_1\\
\zeta_2\\
\vdots\\
\zeta_m
\end{array}\right),
\]
for $\zeta\in C^\infty(Q_T)^m$. With this choice of $\mathcal{N}$, it suffices to consider differential operators $\bar{\mathcal{L}}: C^\infty(Q_T)^m \rightarrow C^\infty(Q_T)^{m-c}$ and $\bar{\mathcal{N}}: C^\infty(Q_T)^{m-c} \rightarrow C^\infty(Q_T)^{m-c}$ defined by
\[
\bar{\mathcal{L}}\zeta:=
\left(\begin{array}{c}
\left(\partial_t-\text{div}(d_{c+1}\nabla)\right)\zeta_{c+1} -\sum_{i=1}^m \left(g_{(c+1)i}\cdot \nabla+a_{(c+1)i}\right)\zeta_i\\
\left(\partial_t-\text{div}(d_{c+2}\nabla)\right)\zeta_{c+2} -\sum_{i=1}^m \left(g_{(c+2)i}\cdot \nabla+a_{(c+2)i}\right)\zeta_i\\
\vdots \\
\left(\partial_t-\text{div}(d_{m}\nabla)\right)\zeta_{m} -\sum_{i=1}^m \left(g_{mi}\cdot \nabla+a_{mi}\right)\zeta_i
\end{array}\right)
\]
and
\[
\bar{\mathcal{N}}\zeta :=
\left(\begin{array}{c}
\zeta_{c+1}\\
\vdots\\
\zeta_m
\end{array}\right)
\]
to prove algebraic solvability of~\eqref{eq:algebraically_solvable_2}.
Indeed, with our choice of $\mathcal{N}$ we can write system~\eqref{eq:algebraic_problem} as
\begin{equation}\label{eq:multiple_underactuations_equation1}
\mathcal{L}(\hat{y},\hat{u})=\mathbbm{1}_{\omega}\tilde{u},
\end{equation}
where $\hat{u}$ acts on the first $c$ equations; also, finding a linear partial differential operator $\mathcal{B}$ satisfying~\eqref{eq:algebraically_solvable_2} is equivalent to finding $\mathcal{B}$ such that
\begin{equation}\label{eq:multiple_underactuations_equation2}
\left\{\begin{aligned}
\hat{y}_1 &= \mathcal{B}_1(\mathbbm{1}_{\omega}\tilde{u}), \\
&\;\;\vdots\\
\hat{y}_m &= \mathcal{B}_m(\mathbbm{1}_{\omega}\tilde{u}), \\
\hat{u}_1 & = \mathcal{B}_{m+1}(\mathbbm{1}_{\omega}\tilde{u}), \\
&\;\;\vdots\\
\hat{u}_c & = \mathcal{B}_{m+c}(\mathbbm{1}_{\omega}\tilde{u}). \\
\end{aligned}\right.
\end{equation}
Hence, from our choice of $B$ in~\eqref{eq:algebraic_problem},~\eqref{eq:underdetermined_operator},~\eqref{eq:multiple_underactuations_equation1} and~\eqref{eq:multiple_underactuations_equation2}, we have for $l\in\{1,\dots,c\}$ that the last $c$ entries of $\mathcal{B}$ must satisfy
\begin{align*}
\mathcal{B}_{m+l}(\mathbbm{1}_{\omega}\tilde{u}) &= \left(\partial_t - \text{div}(d_l\nabla)\right)\hat{y}_l - \sum_{i=1}^m \left(g_{li}\cdot\nabla+a_{li}\right)\hat{y}_i - \mathbbm{1}_{\omega}\tilde{u}_l\\
& = \left(\partial_t - \text{div}(d_l\nabla)\right)\mathcal{B}_{l}(\mathbbm{1}_{\omega}\tilde{u}) - \sum_{i=1}^m \left(g_{li}\cdot\nabla+a_{li}\right)\mathcal{B}_{i}(\mathbbm{1}_{\omega}\tilde{u}) - \mathbbm{1}_{\omega}\tilde{u}_l,
\end{align*}
if~\eqref{eq:algebraically_solvable_2} is to be verified. Thus, one need only to find a $\bar{\mathcal{B}}:C^\infty(Q_T)^{m-c} \rightarrow C^\infty(Q_T)^m$ to satisfy the first $m$ lines of~\eqref{eq:multiple_underactuations_equation2}, as the last $c$ lines of~\eqref{eq:multiple_underactuations_equation2} are completely determined by its first $m$ lines and the respective entry of $\tilde{u}$; consequentially, for our choice of $\mathcal{N}$, the algebraic solvability of~\eqref{eq:algebraically_solvable_2} is equivalent to the algebraic solvability of
\begin{equation}\label{eq:reduced_diff_operator_system}
\bar{\mathcal{L}} \circ \bar{\mathcal{B}} = \bar{\mathcal{N}}.
\end{equation}
We study the adjoint equation of~\eqref{eq:reduced_diff_operator_system},
\begin{equation}\label{eq:reduced_diff_operator_system_adjoint}
\bar{\mathcal{B}}^* \circ \bar{\mathcal{L}}^* = \bar{\mathcal{N}}^*,
\end{equation}
and we call $\bar{\mathcal{B}}^*$ the \emph{left inverse} of $\bar{\mathcal{L}}^*$. Similar to~\eqref{eq:L_adjoint_calculation}, we have for $\hat{\psi}\in C^\infty(Q_T)^{m-c}$ that
\[
\bar{\mathcal{L}}^*\hat{\psi}=
\left(\begin{array}{c}
\sum_{j=c+1}^m \left(g_{j1}\cdot \nabla -a_{j1}\right)\hat{\psi}_j\\
\vdots \\
\sum_{j=c+1}^m \left(g_{jc}\cdot \nabla -a_{jc}\right)\hat{\psi}_j\\
\left(-\partial_t-\text{div}(d_{c+1}\nabla)\right)\hat{\psi}_{c+1} + \sum_{j=c+1}^m \left(g_{j(c+1)}\cdot \nabla -a_{j(c+1)}\right)\hat{\psi}_j\\
\vdots\\
\left(-\partial_t-\text{div}(d_{m}\nabla)\right)\hat{\psi}_{m} + \sum_{j=c+1}^m \left(g_{jm}\cdot \nabla -a_{jm}\right)\hat{\psi}_j\\
\end{array}\right)
\]
and
\[
\bar{\mathcal{N}}^*\hat{\psi} = \left(\begin{array}{c}
\hat{\psi}_{c+1}\\
\vdots \\
\hat{\psi}_m
\end{array}\right).
\]

Hence, the algebraic solvability of~\eqref{eq:reduced_diff_operator_system} is equivalent to proving the existence of a differential operator $\bar{\mathcal{B}}^*:C^\infty(Q_T)^m \rightarrow C^\infty(Q_T)^{m-c}$ such that for every $\phi\in C^\infty(Q_T)^{m}$, if $\hat{\psi}\in C^\infty(Q_T)^{m-c}$ is a solution of
\begin{equation}\label{eq:reduced_system_of_equations}
\left\{\begin{aligned}
&\sum_{j=c+1}^m \left(g_{j1}\cdot \nabla -a_{j1}\right)\hat{\psi}_j\  = \phi_1,\\
&\qquad \vdots \\
&\sum_{j=c+1}^m \left(g_{jc}\cdot \nabla -a_{jc}\right)\hat{\psi}_j = \phi_c,\\
&\left(-\partial_t-\text{div}(d_{c+1}\nabla)\right)\hat{\psi}_{c+1} + \sum_{j=c+1}^m \left(g_{j(c+1)}\cdot \nabla -a_{j(c+1)}\right)\hat{\psi}_j = \phi_{c+1},\\
&\qquad \vdots\\
&\left(-\partial_t-\text{div}(d_{m}\nabla)\right)\hat{\psi}_{m} + \sum_{j=c+1}^m \left(g_{jm}\cdot \nabla -a_{jm}\right)\hat{\psi}_j = \phi_m,\\
\end{aligned}\right.
\end{equation}
then
\begin{equation}\label{eq:algebraic_solvability_reformulation}
\bar{\mathcal{B}}^*\phi = \left(\begin{array}{c}
\hat{\psi}_{c+1}\\
\vdots \\
\hat{\psi}_m
\end{array}\right).
\end{equation}

We encode systems of equations related to system~\eqref{eq:reduced_system_of_equations} using matrices: we utilize a matrix containing the coefficients of $D$, $G$, $A$ and $-1$ (to account for the time derivative terms) as entries to describe system~\eqref{eq:reduced_system_of_equations}; this encoding is made precise in the work to follow. Throughout this work, we make the following assumption.
\begin{assumption}\label{assumption:distinct_equations}
Assume that the equations in system~\eqref{eq:reduced_system_of_equations} are distinct, i.e., that the matrix associated to system~\eqref{eq:reduced_system_of_equations} is of full rank.
\end{assumption}

An examination of~\eqref{eq:reduced_system_of_equations} reveals that there are $m$ distinct equations and only $m-c$ unknowns, them being $\hat{\psi}_{c+1},\dots,\hat{\psi}_m$. Let us call $\hat{\psi}_{c+1},\dots,\hat{\psi}_m$ the \emph{analytic unknowns}. If we view~\eqref{eq:reduced_system_of_equations} as a linear algebraic system by treating every (time and spatial) derivative of $\hat{\psi}_l$ as an \emph{independent algebraic unknown}, for $l\in\{c+1,\dots,m\}$, then there are many more algebraic unknowns than distinct equations. Under this algebraic viewpoint, one can hope to prolong (or differentiate with respect to every spatial variable) each equation of~\eqref{eq:reduced_system_of_equations} to introduce many new equations and a few new algebraic unknowns (owing to the symmetry property of mixed partial derivatives). Repeating this process a sufficient amount of times, one can hope that the linear algebraic system eventually becomes \emph{overdetermined}, that is, the number of distinct equations eventually exceeds the number of algebraic unknowns. Proceeding this way, we begin by counting the number of derivatives up to the highest order contained in a prolonged version of system~\eqref{eq:reduced_system_of_equations}, which is an adaptation of the method used in~\cite[Subection 3.2.2]{coron2014local}.
\begin{lemma}\label{lemma:counting_derivatives}
Let $p\in\mathbb{N}$ denote the number of prolongations of~\eqref{eq:reduced_system_of_equations}, and let $F(p)$ denote the distinct number of derivatives of order less than or equal to $p$ for \emph{smooth enough} functions having $n$ variables. Then
\begin{equation}\label{eq:count_distinct_derivatives}
F(p) = {p+n\choose n}.
\end{equation}
Furthermore, denoting by $U(p)$ and by $E(p)$ the number of algebraic unknowns and the number of equation contained in the prolonged version of system~\eqref{eq:reduced_system_of_equations}, respectively, we have
\begin{equation}\label{eq:unknowns_counting}
U(p)=(m-c)\left(F(p+2)+F(p)\right),
\end{equation}
and
\begin{equation}\label{eq:equations_counting}
E(p) = mF(p).
\end{equation}
\end{lemma}
\begin{proof}
Let $\alpha$ be a multi-index of length $n$ such that $|\alpha|\leq p$: that is, $\alpha = (\alpha_1,\dots,\alpha_{n})\in\mathbb{N}^{n}$, where $\sum_{i=1}^n \alpha_i \leq p$. Note that
\[
(\alpha_1,\dots,\alpha_{n})\mapsto\left\{\alpha_1+1,\alpha_1+\alpha_2+2,\alpha_1+\alpha_2+\alpha_3+3,\dots, \sum_{i=1}^n\alpha_i+n\right\}
\]
defines a bijection between the set of tuples $(\alpha_1,\dots,\alpha_{n})\in\mathbb{N}^{n}$ such that $|\alpha|\leq p$ and the set of subsets of $\{1,2,\dots,p+n\}$ having $n$ elements. Furthermore, attributing the multi-index $\alpha$ to the partial derivative operator $\partial_\alpha= \partial_{\alpha_1}\cdots\partial_{\alpha_{n}}$ takes into account the symmetry of mixed partial derivatives, and thus only counts the \emph{distinct} number of derivatives of order less than or equal to $p$. Since the cardinality of the set of subsets of $\{1,2,\dots,p+n\}$ having $n$ elements is ${p+n\choose n}$, we have~\eqref{eq:count_distinct_derivatives}.

Since each analytic unknown contained in system~\eqref{eq:reduced_system_of_equations} has corresponding algebraic unknowns of order up to two in space and one time derivative unknown, and there are $m-c$ analytic unknowns,~\eqref{eq:unknowns_counting} follows.

Since there are $m$ equations, each of which is prolonged $p$ times, and $F(p)$ can be used to represent the number of distinct equations differentiated with respect to the multi-index $\alpha$,~\eqref{eq:equations_counting} follows.
\end{proof}

Concerning our system~\eqref{eq:reduced_system_of_equations}, we have the following lemma.
\begin{lemma}\label{lemma:prolongation_always_works}
For all $m\in\mathbb{N}_{>1}$, $n\in\mathbb{N}^*$ and $c\in \{1,\dots,m-2\}$ such that $c>\frac{m}{2}$, there exists $p\in\mathbb{N}^*$ such that 
\[
E(p) > U(p).
\]
\end{lemma}
\begin{proof}
We claim that there exists $ p\in\mathbb{N}^*$ such that
\[
c{p+n\choose n} > (m-c){p+n+2\choose n}.
\]
Indeed, we have
\[
(m-c){p+n+2\choose n}=(m-c)\frac{(p+n+2)(p+n+1)}{(p+2)(p+1)}\frac{(p+n)!}{p!n!}
\]
and
\[
c{p+n\choose n} = c\frac{(p+n)!}{p!n!}.
\]
First, we show that for fixed $m$ and $n$, there exist $p$ and $c$ such that
\begin{equation}\label{eq:prolongation_inequality}
\frac{(p+n+2)(p+n+1)}{(p+2)(p+1)} < \frac{c}{(m-c)};
\end{equation}
Indeed, since $m\in\mathbb{N}_{>1}$, we can choose $c>\frac{m}{\frac{(p+2)(p+1)}{(p+n+2)(p+n+1)}+1}$ to verify~\eqref{eq:prolongation_inequality}. Note that $\frac{(p+2)(p+1)}{(p+n+2)(p+n+1)}\rightarrow 1$ from below as $p\rightarrow \infty$, and thus $c>\frac{m}{2}$ is necessary for $E(p)>U(p)$. Since $m\in\mathbb{N}^*$ and $c\in\{1,\dots,m-2\}$, $c>\frac{m}{2}$ is also sufficient since one can always choose $p\in\mathbb{N}$ large enough to verify~\eqref{eq:prolongation_inequality} when $c=\left\lfloor\frac{m}{2}\right\rfloor+1$.
\end{proof}
\begin{remark}\label{remark:prolongation_regularity}
Lemma~\ref{lemma:prolongation_always_works} shows that for a sufficiently regular solution $\hat{\psi}$ to system~\eqref{eq:algebraic_problem_adjoint}, if $c\geq\left\lfloor\frac{m}{2}\right\rfloor+1$, then there exists $p\in\mathbb{N}^{\*}$ such that we can prolong system~\eqref{eq:reduced_system_of_equations} $p$ times and study the resulting overdetermined linear algebraic system. One can argue the appropriate regularity of $\hat{\psi}$ as follows: without loss of generality, we can take $\hat{\psi}^0\in H^{p+1}(\Omega)^m$ by a classical density argument; then, ones applies Theorem~\ref{thm:smoothing_theorem}. As we will see, under certain conditions, one may hope to extract the analytic unknowns $\hat{\psi}_{c+1},\dots,\hat{\psi}_m$ from the overdetermined algebraic system. Hence, one can expect the left inverse of the differential operator associated to the prolonged version of system~\eqref{eq:reduced_system_of_equations} to be of maximum differential order $p+2$ in space and $1$ in time. Thus, by~\eqref{eq:multiple_underactuations_equation2} we require the analytic system's controls, $\mathbbm{1}_\omega \tilde{u}$, to accommodate $p+2$ spatial differentiations. These highly regular $\mathbbm{1}_\omega \tilde{u}$ are constructed in~\cite{steeves2017part2}.\\
\end{remark}

We finish this work by proving the main result.
\begin{proposition}\label{prop:algebraic_solvability}
Given $m, n$ and $c$ in $\mathbb{N}^*$ with $\lfloor\frac{m}{2}\rfloor+1\leq c\leq m$, if 
\begin{enumerate}
\item $c\geq h$, where $h:=(m-c)(n+1)$, and;
\item the matrix $C\in\mathcal{M}_{h}(\mathbb{R})$ given by 
\begin{equation}\label{eq:C_def}
\scalemath{0.9}{
C:=\left(\begin{array}{c c c c c c c c c c}
a_{(c+1)\alpha_1}&\dots&a_{m\alpha_1}&g_{(c+1)\alpha_1}^1&\dots&g_{m\alpha_1}^1&\dots&g_{(c+1)\alpha_1}^n&\dots&g_{m\alpha_1}^n\\
a_{(c+1)\alpha_2}&\dots&a_{m\alpha_2}&g_{(c+1)\alpha_2}^1&\dots&g_{m\alpha_2}^1&\dots&g_{(c+1)\alpha_2}^n&\dots&g_{m\alpha_2}^n\\
\vdots&&\vdots&\vdots&&\vdots&&\vdots&&\vdots\\
a_{(c+1)\alpha_{h}}&\dots&a_{m\alpha_{h}}&g_{(c+1)\alpha_{h}}^1&\dots&g_{m\alpha_{h}}^1&\dots&g_{(c+1)\alpha_{h}}^n&\dots&g_{m\alpha_{h}}^n
\end{array}\right)}
\end{equation}
is non-singular for any $\{\alpha_1,\dots,\alpha_{h}\}\subseteq\{1,\dots,c\}$ with $\alpha_1\neq\dots\neq\alpha_{h}$, where $g_{ij}^k$ is the $k$-th component of $g_{ij}$, for $k\in\{1,\dots,n\}$ and for $i,j\in\{1,\dots,m\}$,
\end{enumerate}
then~\eqref{eq:algebraically_solvable_2} is algebraically solvable in $q_T$.
\end{proposition}
\begin{proof}
Without loss of generality, for a given $m$, $n$ and $c$, we fix a $p$ large enough such that $E(p)>U(p)$. Consider the overdetermined matrix $\bar{L}^*\in\mathcal{M}_{E(p)\times U(p)}(\mathbb{R})$ with entries equal to the coefficients multiplying the algebraic unknowns generated by prolonging system~\eqref{eq:reduced_system_of_equations} $p$ times. We denote the vector containing the $p$-times prolonged unknowns by $\hat{z}\in\mathcal{M}_{U(p)\times 1}(L^2(Q_T))$, where the necessary regularity of $\hat{\psi}$ is discussed in Remark~\ref{remark:prolongation_regularity}. Similarly, we denote the $p$-times prolonged version of $\phi$ by $\Phi\in\mathcal{M}_{E(p)\times 1}(C^\infty(Q_T))$. Hence, we can write the \emph{prolonged algebraic version} of the system~\eqref{eq:reduced_system_of_equations} as
\begin{equation}\label{eq:reduced_system_of_equations_algebraic_version}
\bar{L}^*\hat{z}=\Phi.
\end{equation}
The counterpart of solving~\eqref{eq:reduced_system_of_equations} and~\eqref{eq:algebraic_solvability_reformulation} simultaneously for~\eqref{eq:reduced_system_of_equations_algebraic_version} is to find a $P\in\mathcal{M}_{(m-c)\times E(p)}$ such that
\begin{equation}\label{eq:algebraic_solvability_algebraic_version}
P\bar{L}^*\hat{z}=\left(\begin{array}{c} \hat{\psi}_{c+1}\\\vdots\\\hat{\psi}_m\end{array}\right),
\end{equation}
with $P$ being the algebraic version of $\bar{\mathcal{B}}^*$. We apply Theorem~\ref{thm:dmperm_decomposition} to $\bar{L}^*$ so that for $S_{\tilde{\sigma}}$ and $S_\sigma$ the left and right permutation matrices generated by the Dulmage-Mendelsohn decomposition, respectively, we have
\begin{equation}\label{eq:dmperm_L}
S_{\tilde{\sigma}}\bar{L}^*S_\sigma= \left(\begin{array}{c c c c}
\bar{L}^*_{11} & \bar{L}^*_{12} & \bar{L}^*_{13} & \bar{L}^*_{14}\\
0 & 0 & \bar{L}^*_{23} & \bar{L}^*_{24}\\
0 & 0 & 0 & \bar{L}^*_{34}\\
0 & 0 & 0 & \bar{L}^*_{44}
\end{array}\right),
\end{equation}
where $\bar{L}^*_{34}$ is square and perfectly matched (i.e., it is of maximal structural rank). We permute $\hat{z}$ accordingly by $S_\sigma^{-1}$. 

Our next steps are as follows. First, we study the structure of $\bar{L}^*$ and argue that under $S_{\tilde{\sigma}}$ and $S_\sigma$, every row of $C$ (which appear in $\bar{L}^*$) is permuted to block $\bar{L}^*_{34}$ (possibly with some zero entries to the right). Then, we argue that the unknowns $\hat{\psi}_{c+1},\dots,\hat{\psi}_m$ contained in $\hat{z}$ are being multiplied by the block $\bar{L}^*_{34}$ (and in particular, the rows of $C$). Finally, we deduce from the full rank of $C$ that $\bar{L}^*_{34}$ is non-singular (possibly after some row permutations of $\left(\bar{L}^*_{34}\quad \bar{L}^*_{44}\right)^T$), which yields a $P$ satisfying~\eqref{eq:algebraic_solvability_algebraic_version}. Immediately following the end of this proof, we supplement our explanations with Example~\ref{eq:L_algebraic_example}, which demonstrates how this proof is carried out on a symbolic matrix, identifies how the proof fails for $c<h$, and provides a similar rank condition to~\eqref{eq:C_def} which ensures algebraic solvability for the scenario $c<h$.

%Furthermore, in Remark~\ref{rk:c_less_than_h} we give some insight into the null controllability condition for the case where $c<h$ (which is possibly non-generic, i.e., it is the non-singularity of a sparse matrix).
By construction of $\bar{L}^*$, we have that the columns of $\bar{L}^*$ corresponding to any algebraic unknown involving a time derivative are very sparse. Indeed, each of these columns has only one nonzero entry (which is $-1$). This occurs since we do not prolong system~\eqref{eq:reduced_system_of_equations} with respect to time, and hence each time derivative term appears in one (and only one) equation within the prolonged version of system~\eqref{eq:reduced_system_of_equations}. Furthermore, the row associated to any one of these $-1$ column entries must correspond to the $j$-th equation (or its prolonged version) in system~\eqref{eq:reduced_system_of_equations}, for $j\in\{c+1,\dots,m\}$. Hence, the coefficients corresponding to the $j$-th equation (or its prolonged version) in system~\eqref{eq:reduced_system_of_equations} lie in this row, for $j\in\{c+1,\dots,m\}$.

We claim that there exists a maximum matching $M$ in $G(\bar{L}^*)$ that contains all of the edges $(r_i,c_i)$ corresponding to these $-1$ entries. Indeed, for any matrix $P$, a \emph{matching} in $G(P)$ is a subset of nonzero entries of $P$ such that no two of which belong to the same row or column. Hence, since the columns of $\bar{L}^*$ corresponding to any algebraic unknown involving a time derivative contain only one nonzero entry, it is easy to deduce that there exists a maximum matching $M$ in $G(\bar{L}^*)$ chosen to include all of these $-1$ entries. Importantly, this choice will omit any other edges associated to coefficients corresponding to the $j$-th equation (or its prolonged version) in system~\eqref{eq:reduced_system_of_equations}, for $j\in\{c+1,\dots,m\}$, from $M$, and the rows containing these coefficients will be matched (see Example~\ref{example:c_less_than_h}). Furthermore, we can choose at random enough edges which make $M$ a maximum matching; it follows that all of these edges will correspond to coupling coefficients of the $j$-th equation (or its prolonged version) in system~\eqref{eq:reduced_system_of_equations}, for $j\in\{1,\dots,c\}$. Without loss of generality, we associate $S_{\tilde{\sigma}}$ and $S_\sigma$ to this choice of maximum matching.

With our choice of $M$, we now study vertex sets $VR$ and $VC$. Recall from Section~\ref{sec:sparse_matrices} that
\begin{align*}
VR&:=\{\text{row vertices reachable by alternating paths from some unmatched row}\},\\
VC&:=\{\text{column vertices reachable by alternating paths from some unmatched row}\},
\end{align*}
where an alternating path is a sequence of (row or column) vertices $(v_i)_{i=0}^k$ such that $(v_{2i},v_{2i+1})\in E$ and, additionally, $(v_{2i+1},v_{2(i+1)})\in M$ and no vertices are repeated, for $k\in\mathbb{N}^*$. By our choice of $M$ and since $\bar{L}^*$ is overdetermined, there exists unmatched rows, and any unmatched row must correspond to the $j$-th equation (or its prolonged version) in system~\eqref{eq:reduced_system_of_equations}, for $j\in\{1,\dots,c\}$. One deduces from the structure of $\bar{L}^*$ that these unmatched rows have nonzero entries which lie in matched columns, and hence $VR$ and $VC$ are not empty. Furthermore, these matched columns cannot be those corresponding to algebraic unknowns involving a time derivative. By the structure of $\bar{L}^*$, all row vertices corresponding to the $j$-th equation in system~\eqref{eq:reduced_system_of_equations} are reachable by an alternating path from some unmatched row, for all $j\in\{1,\dots,c\}$. This is a consequence of equations in system~\eqref{eq:reduced_system_of_equations} having \emph{first and zero-order coupling coefficients} and since $\bar{L}^*$ is generated by prolongations with respect to spatial variables only. Hence, rows corresponding to the $j$-th equation (or its prolonged version) in system~\eqref{eq:reduced_system_of_equations} have corresponding row vertices contained in $VR$, for $j\in\{1,\dots,c\}$. It follows that columns containing coupling coefficients have corresponding column vertices contained in $VC$ (the same alternating paths yield the column vertices in $VC$). Hence by Theorem~\ref{thm:dmperm_decomposition}, only the coefficients that appear in the $j$-th equation (or its prolonged version) in system~\eqref{eq:reduced_system_of_equations} are permuted to the blocks $\bar{L}^*_{34}$ and $\bar{L}^*_{44}$, for $j\in\{1,\dots,c\}$.

By examining system~\eqref{eq:reduced_system_of_equations}, one easily deduces that the unknowns $\hat{\psi}_{c+1},\dots,\hat{\psi}_m$ are being multiplied by $\bar{L}^*_{34}$ and $\bar{L}^*_{44}$. We permute the rows contained in $C$ (the ones from the original -- and not a prolonged -- system~\eqref{eq:reduced_system_of_equations}, and hence have the same number of zero entries appearing only to their right) to the top of $\bar{L}^*_{34}$; we deduce that $\hat{\psi}_{c+1},\dots,\hat{\psi}_m$ are multiplied by $\bar{L}^*_{34}$. We denote this row permutation on $S_{\tilde{\sigma}}\bar{L}^*S_{\sigma}$ by $S_{\tilde{\sigma}^0}$.
Finally, with a slight abuse of notation, we denote by $I$ various identity matrices with appropriate dimensions; using the row permutation
\[
S_{\tilde{\sigma}^1}:=\left(\begin{array}{c c c c} 0&0&I&0\\0&0&0&I\\I&0&0&0\\0&I&0&0\end{array}\right),
\]
the column permutation
\[
S_{\sigma^1}:=\left(\begin{array}{c c c c} 0&I&0&0\\0&0&I&0\\0&0&0&I\\I&0&0&0\end{array}\right),
\]
we permute $S_{\tilde{\sigma}^0}S_{\tilde{\sigma}}\bar{L}S_\sigma$ into lower-block triangular form with the row--permuted version of $\bar{L}^*_{34}$ being the top leftmost block, and we define
\[
P:=\left(\text{Id}_{m-c}\quad 0_{(m-c)\times(h-m+c)}\right)\left((\bar{L}^*_{34})^{-1}\quad 0_{h\times(E(p)-h)}\right)S_{\tilde{\sigma}^1}S_{\tilde{\sigma}^0}S_{\tilde{\sigma}},
\]
which verifies~\eqref{eq:algebraic_solvability_algebraic_version}. Hence, by the non-singularity of the row-permuted version of $\bar{L}^*_{34}$, there exists a linear combination of differentiated lines of $\bar{\mathcal{L}}^*$ that allow us to recover the analytic unknowns $\hat{\psi}_{c+1},\dots,\hat{\psi}_m$. We denote by $\mathcal{P}$ the differential operator associated to matrix $P$; it follows that $\bar{\mathcal{B}}^*:=\mathcal{P}$ verifies~\eqref{eq:reduced_diff_operator_system_adjoint}, and hence $\bar{\mathcal{B}}=\mathcal{P}^*$ verifies~\eqref{eq:reduced_diff_operator_system}.
\end{proof}
\begin{example}\label{example:c_less_than_h}
In this example, we consider the algebraic control system given by~\eqref{eq:algebraic_problem}, where we choose $m=5$, $c=3$, and $n=1$. Importantly, note that the hypothesis in Proposition~\ref{prop:algebraic_solvability} is not satisfied since $c<h=4$. Consider defining $C$ in the same way as above but with only $c$ rows, so that $C\in\mathcal{M}_{c\times h}(\mathbb{R})$. In this example we illustrate how the full--rank condition of $C$ fails to ensure algebraic solvability of~\eqref{eq:algebraically_solvable_2}; furthermore, we show that imposing a full--rank condition on a matrix $C^\dagger\in\mathcal{M}_{2c\times 2c}(\mathbb{R})$, which is a repeated version in almost--block--diagonal form of $C$ (up to some column negation), is sufficient for algebraic solvability of~\eqref{eq:algebraically_solvable_2}. This extension for scenarios where $c<h$ is useful for smaller systems in low dimension; under these conditions, we will demonstrate that $C^\dagger$ is not very sparse and hence imposing a rank condition is not too restrictive.

In solving the algebraic version of~\eqref{eq:reduced_diff_operator_system_adjoint}, which is given by~\eqref{eq:algebraic_solvability_algebraic_version}, we study the linear algebraic operator obtained by prolonging system~\eqref{eq:reduced_system_of_equations} 3 times given by
\begin{equation}\label{eq:L_algebraic_example}
\bar{L}^*=\left(\scalemath{0.52}{\begin{array}{c c c c c c c c c c c c c c c c c c c c}
\Circle{$-a_{41}$}&-a_{51}&0&0&g_{41}&g_{51}&0&0&0&0&0&0&0&0&0&0&0&0&0&0\\
-a_{42}&\Circle{$-a_{52}$}&0&0&g_{42}&g_{52}&0&0&0&0&0&0&0&0&0&0&0&0&0&0\\
-a_{43}&-a_{53}&0&0&\Circle{$g_{43}$}&g_{53}&0&0&0&0&0&0&0&0&0&0&0&0&0&0\\
-a_{44}&-a_{54}&0&\Circle{$-1$}&g_{44}&g_{54}&d_1&0&0&0&0&0&0&0&0&0&0&0&0&0\\
-a_{45}&-a_{55}&\Circle{$-1$}&0&g_{45}&g_{55}&0&d_2&0&0&0&0&0&0&0&0&0&0&0&0\\
0&0&0&0&-a_{41}&\Circle{$-a_{51}$}&g_{41}&g_{51}&0&0&0&0&0&0&0&0&0&0&0&0\\
0&0&0&0&-a_{42}&-a_{52}&\Circle{$g_{42}$}&g_{52}&0&0&0&0&0&0&0&0&0&0&0&0\\
0&0&0&0&-a_{43}&-a_{53}&g_{43}&\Circle{$g_{53}$}&0&0&0&0&0&0&0&0&0&0&0&0\\
0&0&0&0&-a_{44}&-a_{54}&g_{44}&g_{54}&0&\Circle{$-1$}&d_1&0&0&0&0&0&0&0&0&0\\
0&0&0&0&-a_{45}&-a_{55}&g_{45}&g_{55}&\Circle{$-1$}&0&0&d_2&0&0&0&0&0&0&0&0\\
0&0&0&0&0&0&-a_{41}&-a_{51}&0&0&\Circle{$g_{41}$}&g_{51}&0&0&0&0&0&0&0&0\\
0&0&0&0&0&0&-a_{42}&-a_{52}&0&0&g_{42}&\Circle{$g_{52}$}&0&0&0&0&0&0&0&0\\
0&0&0&0&0&0&-a_{43}&-a_{53}&0&0&g_{43}&g_{53}&0&0&0&0&0&0&0&0\\
0&0&0&0&0&0&-a_{44}&-a_{54}&0&0&g_{44}&g_{54}&0&\Circle{$-1$}&d_1&0&0&0&0&0\\
0&0&0&0&0&0&-a_{45}&-a_{55}&0&0&g_{45}&g_{55}&\Circle{$-1$}&0&0&d_2&0&0&0&0\\
0&0&0&0&0&0&0&0&0&0&-a_{41}&-a_{51}&0&0&\Circle{$g_{41}$}&g_{51}&0&0&0&0\\
0&0&0&0&0&0&0&0&0&0&-a_{42}&-a_{52}&0&0&g_{42}&\Circle{$g_{52}$}&0&0&0&0\\
0&0&0&0&0&0&0&0&0&0&-a_{43}&-a_{53}&0&0&g_{43}&g_{53}&0&0&0&0\\
0&0&0&0&0&0&0&0&0&0&-a_{44}&-a_{54}&0&0&g_{44}&g_{54}&0&\Circle{$-1$}&d_1&0\\
0&0&0&0&0&0&0&0&0&0&-a_{45}&-a_{55}&0&0&g_{45}&g_{55}&\Circle{$-1$}&0&0&d_2
\end{array}}\right).
\end{equation}
In~\eqref{eq:L_algebraic_example}, we have circled a subset of nonzero entries of $\bar{L}^*$ whose corresponding edges make up a maximum matching $M$ in $G(\bar{L}^*)$. We have chosen $M$ such that it contains every edge corresponding to a $-1$ entry of $\bar{L}^*$.

We now populate the set of edges $VR$ and $VC$. Note that $r_{13}$ and $r_{18}$ are the only unmatched row vertices; hence, we search for row vertices that are reachable from $r_{13}$ and $r_{18}$ via an alternating path. A crucial observation is that there exists no alternating paths from these row vertices to the row vertices corresponding to rows containing $-1$ entries: indeed, for a walk starting from row $13$, since $(r_{13},c_{i})\not\in M$, for $i\in\{7,8,11,12\}$, the next row vertex $r_k$ in the walk must be such that $(c_i,r_k)\in M$, for $k\in\{1,\dots,12,14,\dots,20\}$, hence $r_k\neq r_j$ for $j\in\{4,5,9,10,14,15,19,20\}$; the exact same argument holds for a walk starting from row $18$. One can easily deduce by the same reasoning that $r_j$ will never be reachable by a (longer) alternating path, for $j\in\{4,5,9,10,14,15,19,20\}$. Furthermore, every other row vertex is reachable by an alternating path from either $r_{13}$ or $r_{18}$! Hence,
\[
VR=\{r_1,r_2,r_3,r_6,r_7,r_8,r_{11},r_{12},r_{13},r_{16},r_{17},r_{18}\},
\]
and it follows that
\[
VC=\{c_1,c_2,c_5,c_6,c_7,c_8,c_{11},c_{12},c_{15},c_{16}\}.
\]
Hence, we arrive at (possibly after a row permutation)
\[\scalemath{0.96}{
\Bigg(\begin{array}{c}\bar{L}^*_{34}\\\bar{L}^*_{44}\end{array}\Bigg)=\left(\begin{array}{c c c c c c c c c c}
-a_{41}&-a_{51}&g_{41}&g_{51}&0&0&0&0&0&0\\
-a_{42}&-a_{52}&g_{42}&g_{52}&0&0&0&0&0&0\\
-a_{43}&-a_{53}&g_{43}&g_{53}&0&0&0&0&0&0\\
0&0&-a_{41}&-a_{51}&g_{41}&g_{51}&0&0&0&0\\
0&0&-a_{42}&-a_{52}&g_{42}&g_{52}&0&0&0&0\\
0&0&-a_{43}&-a_{53}&g_{43}&g_{53}&0&0&0&0\\
0&0&0&0&-a_{41}&-a_{51}&g_{41}&g_{51}&0&0\\
0&0&0&0&-a_{42}&-a_{52}&g_{42}&g_{52}&0&0\\
0&0&0&0&-a_{43}&-a_{53}&g_{43}&g_{53}&0&0\\
0&0&0&0&0&0&-a_{41}&-a_{51}&g_{41}&g_{51}\\
0&0&0&0&0&0&-a_{42}&-a_{52}&g_{42}&g_{52}\\
0&0&0&0&0&0&-a_{43}&-a_{53}&g_{43}&g_{53}
\end{array}\right).}
\]
Notice that $\left(\begin{array}{c}\bar{L}^*_{34}\\\bar{L}^*_{44}\end{array}\right)$ contains $C$, which multiplies the unknowns we wish to recover. However, to recover all of the unknowns, one requires a \underline{square} non--singular matrix, which is not present in $\left(\begin{array}{c}\bar{L}^*_{34}\\\bar{L}^*_{44}\end{array}\right)$ due to $c<h$. Instead, one must ``upgrade" the size of the matrix required to be non--singular to the next smallest square candidate, given by
\[\scalemath{0.96}{
C^\dagger:=\left(\begin{array}{c c c c c c}
-a_{41}&-a_{51}&g_{41}&g_{51}&0&0\\
-a_{42}&-a_{52}&g_{42}&g_{52}&0&0\\
-a_{43}&-a_{53}&g_{43}&g_{53}&0&0\\
0&0&-a_{41}&-a_{51}&g_{41}&g_{51}\\
0&0&-a_{42}&-a_{52}&g_{42}&g_{52}\\
0&0&-a_{43}&-a_{53}&g_{43}&g_{53}
\end{array}\right).}
\]
Requiring instead $C^\dagger$ to have full rank ensures algebraic solvability of~\eqref{eq:algebraically_solvable_2}. This analysis demonstrates the versatile albeit intricate nature of the algebraic solvability technique utilized here: for small systems (e.g., m=5) in low dimensions, a sufficient, readily--verifiable rank condition for algebraic solvability can be derived; for large systems with severe underactuation and in higher dimension, one requires the non-singularity of a sparse matrix, and hence, to the best of the authors' knowledge, no generic rank condition can be provided for the scenario $c<h$ in general.
\end{example}
\section{Conclusion and possible extension}\label{sec:conclusion}
In the first part of this two-part work, we used a powerful technique, the so-called fictitious control method, which allowed us to pose our controllability problem as two interconnected problems. We derived a sufficient condition for the algebraic solvability of a system of coupled parabolic PDEs, where the couplings were constant in space and time and of first and zero-order, when more than half of the equations in the system were actuated. With algebraic solvability established, we can now study the analytic system~\eqref{eq:analytic_problem}; proving its null controllability will help us recover null controllability of the original control system~\eqref{eq:coupled_control_system}.
%for the case of $c\geq h$, where $c$ denotes the number of controls and for $h$ defined in Theorem~\ref{thm:main_theorem}. Furthermore, we demonstrated that for $c<h$, the possibly non-generic nature of this controllability condition is purely technical and is an artifact of our treatment in Section~\ref{sec:algebraic_solvability}.

One could explore different choices of differential operator $\mathcal{N}$, which may yield a milder controllability condition that close the gap between sufficiency and necessity (as in Lemma~\ref{thm:suff_necc_one_underactuation} for one underactuation).

 \section{Acknowledgements}
We express our gratitude to Professor Pierre Lissy, with whom we communicated regarding the algebraic solvability technique used in~\cite{coron2014local}.
\myclearpage
\bibliographystyle{siam}%
\bibliography{DS}